\documentclass[11pt]{article}
\usepackage{amsmath, amsfonts, amsthm, amssymb, color}
\usepackage{graphicx}
\usepackage{float}
\usepackage{verbatim}

\allowdisplaybreaks

\numberwithin{equation}{section}

\newtheorem{lemma}{Lemma}[section]

\newtheorem{prop}[lemma]{Proposition}
\newtheorem{thm}[lemma]{Theorem}
\newtheorem{cor}[lemma]{Corollary}
\theoremstyle{definition}

\newtheorem{example}[lemma]{Example}
\newtheorem{conjecture}[lemma]{Conjecture}

\theoremstyle{remark}

\def\R{\mathbb{R}}
\def\N{\mathbb{N}}
\def\t{\mathbf{t}}
\def\y{\mathbf{y}}
\def\D{\mathcal{D}}
\def\L{\mathcal{L}}

\def\a{\mathbf{a}}
\def\b{\mathbf{b}}
\def\c{\mathbf{c}}
\def\m{\mathfrak{m}}
\def\h{\mathfrak{h}}

\numberwithin{equation}{section} \numberwithin{table}{section}

\title{Overlapping iterated function systems from the perspective of Metric Number Theory}
\author{Simon Baker\\ \\
\emph{School of Mathematics,} \\ \emph{University of Birmingham,} \\ \emph{Birmingham,  B15 2TT, UK.} \\ Email: simonbaker412@gmail.com\\}

\date{\today}
\begin{document}
\maketitle

\begin{abstract}
In this paper we develop a new approach for studying overlapping iterated function systems. This approach is inspired by a famous result due to Khintchine from Diophantine approximation which shows that for a family of limsup sets, their Lebesgue measure is determined by the convergence or divergence of naturally occurring volume sums. For many parameterised families of overlapping iterated function systems, we prove that a typical member will exhibit similar Khintchine like behaviour. Families of iterated function systems that our results apply to include those arising from Bernoulli convolutions, the $\{0,1,3\}$ problem, and affine contractions with varying translation parameter. As a by-product of our analysis we obtain new proofs of some well known results due to Solomyak on the absolute continuity of Bernoulli convolutions, and when the attractor in the $\{0,1,3\}$ problem has positive Lebesgue measure. 

For each $t\in [0,1]$ we let $\Phi_t$ be the iterated function system given by $$\Phi_{t}:=\Big\{\phi_1(x)=\frac{x}{2},\phi_2(x)=\frac{x+1}{2},\phi_3(x)=\frac{x+t}{2},\phi_{4}(x)=\frac{x+1+t}{2}\Big\}.$$ We prove that either $\Phi_t$ contains an exact overlap, or we observe Khintchine like behaviour. Our analysis shows that by studying the metric properties of limsup sets, we can distinguish between the overlapping behaviour of iterated function systems in a way that is not available to us by simply studying properties of self-similar measures.  

Last of all, we introduce a property of an iterated function system that we call being consistently separated with respect to a measure. We prove that this property implies that the pushforward of the measure is absolutely continuous. We include several explicit examples of consistently separated iterated function systems.\\



\end{abstract}

\tableofcontents

\section{Introduction}
Attractors generated by iterated function systems are among the first fractal sets a mathematician encounters. The familiar middle third Cantor set and the Koch curve can both be realised as attractors for appropriate choices of iterated function system. Attractors generated by iterated function systems have the property that they are equal to several scaled down copies of themselves. When these copies are disjoint, or satisfy some weaker separation assumption, then much can be said about the attractor's metric and topological properties. However, when these copies overlap significantly the situation is much more complicated. Measuring how an iterated function system overlaps, and determining properties of the corresponding attractor, are two important problems that are occupying much current research (see for example \cite{Hochman,Hochman2,Shm2,Shm,ShmSol,Var,Varju2}). The purpose of this paper is to develop a new approach for measuring how an iterated function system overlaps. This approach is inspired by classical results from Diophantine approximation and metric number theory. One such result due to Khintchine demonstrates that for a class of limsup sets defined in terms of the rational numbers, their Lebesgue measure is determined by the convergence or divergence of naturally occurring volume sums (see \cite{Khit}). Importantly this result provides a quantitative description of how the rational numbers are distributed within $\mathbb{R}$.  

In this paper we study limsup sets that are defined using iterated function systems (for their definition see Section \ref{two families}). We are motivated by the following goals:
\begin{enumerate}
	\item We would like to determine whether it is the case that for a parameterised family of iterated function systems, a typical member will satisfy an appropriate analogue of Khintchine's theorem.
	\item We would like to answer the question: Does studying the metric properties of these limsup sets allow us to distinguish between the overlapping behaviour of iterated function systems in a way that was not previously available?
	\item We would like to understand how the metric properties of these limsup sets relates to traditional methods for measuring how an iterated function system overlaps, such as the dimension and absolute continuity of self-similar measures.
\end{enumerate} In this paper we make progress with each of these goals. Theorems \ref{1d thm}, \ref{translation thm} and  \ref{random thm} address the first goal. These results demonstrate that for many parametrised families of overlapping iterated function systems, it is the case that a typical member will satisfy an appropriate analogue of Khintchine's theorem. To help illustrate this point, and to motivate what follows, we include here a result which follows from Theorem \ref{1d thm}.
\begin{thm}
For Lebesgue almost every $\lambda\in(1/2,0.668),$ Lebesgue almost every $x\in [\frac{-1}{1-\lambda},\frac{1}{1-\lambda}]$ is contained in  $$\left\{x\in \R:\big|x-\sum_{i=1}^{m}d_i\lambda^{i-1}\big|\leq \frac{1}{2^m\cdot m}\textrm{ for infinitely many }(d_i)_{i=1}^{m}\in\bigcup_{n=1}^{\infty}\{-1,1\}^n\right\}.$$ 
\end{thm} Theorem \ref{precise result} allows us to answer the question stated in our second goal in the affirmative. See the discussion in Section \ref{new methods} for a more precise explanation.  Theorem \ref{Colette thm} addresses the third goal. It shows that if we are given some measure $\m$, and our iterated function system satisfies a strong version of Khintchine's theorem with respect to $\m$, then the pushforward of $\m$ must be absolutely continuous. Moreover, we demonstrate with several examples that this strong version of Khintchine's theorem is not equivalent to the absolute continuity of the pushforward measure.

In the rest of this introduction we provide some more background to this topic, and introduce the limsup sets that will be our main object of study.

\subsection{Attractors generated by iterated function systems}
We call a map $\phi:\mathbb{R}^d\to\mathbb{R}^d$ a contraction if there exists $r\in (0,1)$ such that $|\phi(x)-\phi(y)|\leq r|x-y|$ for all $x,y\in\mathbb{R}^d$. We call a finite set of contractions an iterated function system or IFS for short. A well known result due to Hutchinson \cite{Hut} states that given an IFS $\Phi=\{\phi_i\}_{i=1}^l,$ then there exists a unique, non-empty, compact set $X$ satisfying $$X=\bigcup_{i=1}^l\phi_i(X).$$ We call $X$ the attractor generated by $\Phi$. When an IFS satisfies $\phi_{i}(X)\cap \phi_{j}(X)=\emptyset$ for all $i\neq j$, or is such that there exists an open set $O\subset \mathbb{R}^d,$ for which $\phi_i(O)\subset O$ for all $i$ and $\phi_{i}(O)\cap \phi_{j}(O)=\emptyset$ for all $i\neq j,$ then many important properties of $X$ can be determined (see \cite{Fal1}). This latter property is referred to as the open set condition. Without these separation assumptions determining properties of the attractor can be significantly more complicated. 

The study of attractors generated by iterated function systems is classical within fractal geometry. One of the most important problems in this area is to determine the metric properties of attractors generated by overlapping iterated function systems. To understand the properties of an attractor $X,$ in both the overlapping case and non-overlapping case, it is useful to study measures supported on $X$. A particularly distinguished role is played by the measures described below, that are in a sense dynamically defined. 

Let $\pi:\{1,\ldots,l\}^{\mathbb{N}}\to X$ be given by $$\pi((a_j)_{j=1}^{\infty}):=\lim_{n\to\infty}(\phi_{a_1}\circ \cdots \circ \phi_{a_n})(\mathbf{0}).$$ The map $\pi$ is surjective and is also continuous when $\{1,\ldots,l\}^{\mathbb{N}}$ is equipped with the product topology. The sequence space $\{1,\ldots,l\}^{\mathbb{N}}$ comes with a natural left-shift map $\sigma:\{1,\ldots,l\}^{\mathbb{N}}\to\{1,\ldots,l\}^{\mathbb{N}}$ defined via the equation $\sigma((a_j)_{j=1}^{\infty})=(a_{j+1})_{j=1}^{\infty}$. Given a finite word $\a=a_1\cdots a_n,$ we associate its cylinder set
$$[\a]:=\{(b_j)\in\{1,\ldots,l\}^\N:b_1\cdots b_n=a_1\cdots a_n\}.$$
 We call a measure $\m$ on $\{1,\ldots,l\}^{\mathbb{N}}$ $\sigma$-invariant if $\m([\a])=\m(\sigma^{-1}([\a]))$ for all finite words $\a$. We call a probability measure $\m$ ergodic if $\sigma^{-1}(A)=A$ implies $\m(A)=0$ or $\m(A)=1$.  Given a measure $\m$ on $\{1,\ldots,l\}^{\mathbb{N}},$ we obtain the corresponding pushforward measure $\mu$ supported on $X$ using the map $\pi,$ i.e. $\mu=\m\circ \pi^{-1}$. 
 
 We define the dimension of a measure $\mu$ on $\mathbb{R}^d$ to be $$\dim \mu=\inf \{\dim_{H}(A):\mu(A)>0\}.$$ Note that for any pushforward measure $\mu$ we have $\dim \mu\leq \dim_{H}(X).$ The problem of determining $\dim_{H}(X)$ is often solved by finding a $\sigma$-invariant ergodic probability  measure whose pushforward has dimension equal to some known upper bound for $\dim_{H}(X)$. This approach is especially useful when the iterated function system is overlapping.

When studying attractors of iterated function systems, one of the guiding principles is that if there is no obvious mechanism preventing an attractor from satisfying a certain property, then one should expect this property to be satisfied. This principle is particularly prevalent in the many conjectures which state that under certain reasonable assumptions, the Hausdorff dimension of $X$, and the Hausdorff dimension of dynamically defined pushforward measures supported on $X$, should equal the value predicted by a certain formula. A particular example of this phenomenon is provided by self-similar sets and self-similar measures. We call a contraction $\phi$ a similarity if there exists $r\in(0,1)$ such that $|\phi(x)-\phi(y)|=r|x-y|$ for all $x,y\in \mathbb{R}^d$. If an IFS $\Phi$ consists of similarities then it is known that
\begin{equation}
\label{dimension upper}
\dim_{H}(X)\leq \min \{\dim_{S}(\Phi),d\},
\end{equation}where $\dim_{S}(\Phi)$ is the unique solution to $\sum_{i=1}^l r_i^s=1$. Given a probability vector $\textbf{p}=(p_1,\ldots,p_l),$ we let $\m_{\textbf{p}}$ denote the corresponding Bernoulli measure supported on $\{1,\ldots,l\}^{\mathbb{N}}.$ If an IFS consists of similarities, then we define the self-similar measure corresponding to $\textbf{p}$ to be  $\mu_{\textbf{p}}:=\m_{\textbf{p}}\circ \pi^{-1}.$ The measure $\mu_{\textbf{p}}$ can also be defined as the unique measure satisfying the equation $$\mu_{\textbf{p}}=\sum_{i=1}^l p_i \cdot \mu_{\textbf{p}}\circ \phi_{i}^{-1}.$$ For any self-similar measure $\mu_{\textbf{p}},$ we have the upper bound:
\begin{equation}
\label{expected dimensiona}
\dim \mu_{\textbf{p}}\leq \min\left\{\frac{\sum_{i=1}^l p_i\log p_i}{\sum_{i=1}^l p_i\log r_i},d\right\}.
\end{equation} For an appropriate choice of $\textbf{p},$ it can be shown that equality in \eqref{expected dimensiona} implies equality in \eqref{dimension upper}. An important conjecture states that if an IFS consisting of similarities avoids certain degenerate behaviour, then we should have equality in \eqref{expected dimensiona} for all $\textbf{p},$ and therefore equality in \eqref{dimension upper} (see \cite{Hochman,Hochman2}). In $\mathbb{R}$ this conjecture can be stated succinctly as: If an IFS does not contain an exact overlap, then we should have equality in \eqref{expected dimensiona} for all $\textbf{p}$. Recall that an IFS is said to contain an exact overlap if there exists two distinct words $\a=a_1\cdots a_n$ and $\b=b_1\cdots b_m$ such that $$\phi_{a_1}\circ \cdots \circ \phi_{a_n}=\phi_{b_1}\circ \cdots \circ \phi_{b_m}.$$ In \cite{Hochman} and \cite{Hochman2} significant progress was made towards this conjecture. In particular, in \cite{Hochman} it was shown that for an IFS consisting of similarities acting on $\mathbb{R},$ if strict inequality holds in \eqref{expected dimensiona} for some $\textbf{p},$ then
$$\lim_{n\to\infty}\frac{-\log \Delta_n}{n}=\infty,$$ where $$\Delta_n:=\min_{\a\neq \b\in \{1,\ldots,l\}^n}|(\phi_{a_1}\circ \cdots \circ \phi_{a_n})(0)-(\phi_{b_1}\circ \cdots \circ \phi_{b_n})(0)|.$$ Using this statement, it can be shown that if the parameters defining the IFS are algebraic, and there are no exact overlaps, then equality holds in \eqref{expected dimensiona} for all $\textbf{p},$ and therefore also in \eqref{dimension upper} . 

In addition to expecting equality to hold typically in \eqref{expected dimensiona}, it is expected that if $$\frac{\sum p_i\log p_i}{\sum p_i\log r_i}>d,$$ and the IFS avoids certain obstacles, then $\mu_{\textbf{p}}$ will be absolutely continuous with respect to $d$-dimensional Lebesgue measure. A standard technique for proving an attractor has positive $d$-dimensional Lebesgue measure is to show that there is an absolutely continuous pushforward measure. Note that by a recent result of Simon and V\'{a}g\'{o} \cite{SimVag}, it follows that the list of mechanisms leading to the failure of absolute continuity is strictly greater than the list of mechanisms leading to the failure of equality in \eqref{expected dimensiona}.

The usual methods for gauging how an iterated function system overlaps are to determine whether the Hausdorff dimension of the attractor satisfies a certain formula, to determine whether the dimension of pushforwards of dynamically-defined measures satisfy a certain formula, and to determine whether these measures are absolutely continuous with respect to the $d$-dimensional Lebesgue measure. If an IFS did not exhibit the expected behaviour, then this would be indicative of something degenerate within our IFS that was either preventing $X$ from being well spread out within $\mathbb{R}^d$, or was forcing mass from the pushforward measure into some small subregion of $\mathbb{R}^d$. This method for gauging how an iterated function system overlaps has its limitations. If each of the expected behaviours described above occurs for two distinct IFSs within a family, then we have no method for distinguishing their overlapping behaviour. The approach put forward in this paper shows how we can still make a distinction (see the discussion in Section \ref{new methods}). As previously stated this approach is inspired by results from Diophantine approximation and metric number theory. We now take the opportunity to briefly recall some background from this area.

\subsection{Diophantine approximation and metric number theory}
Given $\Psi:\mathbb{N}\to[0,\infty)$ we can define a limsup set defined in terms of neighbourhoods of rationals as follows. Let

$$J(\Psi):=\Big\{x\in\mathbb{R}:\Big|x-\frac{p}{q}\Big|\leq \Psi(q) \textrm{ for i.m. } (p,q)\in \mathbb{Z}\times\mathbb{N}\Big\}.$$ Here and throughout we use i.m. as a shorthand for infinitely many. If $x\in J(\Psi)$ we say that $x$ is $\Psi$-approximable. An immediate application of the Borel-Cantelli lemma implies that if $\sum_{q=1}^{\infty}q\cdot \Psi(q)<\infty,$ then $J(\Psi)$ has zero Lebesgue measure. The following theorem due to Khintchine shows that a partial converse to this statement holds. This theorem motivates much of the present work.

\begin{thm}[Khintchine \cite{Khit}]
\label{Khintchine}
If $\Psi:\mathbb{N}\to [0,\infty)$ is decreasing and $$\sum_{q=1}^{\infty}q\cdot \Psi(q)=\infty,$$ then Lebesgue almost every $x\in \mathbb{R}$ is $\Psi$-approximable.
\end{thm}Results analogous to Khintchine's theorem are ubiquitous in Diophantine approximation and metric number theory. We refer the reader to \cite{BDV} for more examples. 

By an example of Duffin and Schaeffer, it can be seen that it is not possible to remove the decreasing assumption from Theorem \ref{Khintchine}. Indeed in \cite{DufSch} they constructed a $\Psi$ such that $\sum_{q=1}^{\infty}q\cdot \Psi(q)=\infty,$ yet $J(\Psi)$ has zero Lebesgue measure. This gave rise to a conjecture known as the Duffin-Schaeffer conjecture which was recently proved by Koukoulopoulos and Maynard in \cite{KouMay}.

\begin{thm}[\cite{KouMay}]
If $\Psi:\mathbb{N}\to [0,\infty)$ satisfies $$\sum_{q=1}^{\infty}\varphi(q)\cdot\Psi(q)=\infty,$$ then Lebesgue almost every $x\in \mathbb{R}$ is $\Psi$-approximable.
\end{thm} 
Here $\varphi$ is the Euler totient function. 

By studying the Lebesgue measure of $J(\Psi)$ for those $\Psi$ satisfying $\sum_{q=1}^{\infty}q\cdot \Psi(q)=\infty,$ we obtain a quantitative description of how the rationals are distributed within the reals. The example of Duffin and Schaeffer demonstrates that there exists some interesting non-trivial interactions occurring between fractions of different denominator.  

\subsection{Two families of limsup sets}
\label{two families}
Before defining the limsup sets we study in this paper, it is necessary to introduce some notation. In what follows we let $$\D:=\{1,\ldots,l\},\quad\, \D^{*}:=\bigcup_{j=1}^{\infty}\{1,\ldots,l\}^j,\quad \, \D^{\mathbb{N}}:=\{1,\ldots,l\}^{\mathbb{N}}.$$ Given an IFS $\Phi=\{\phi_i\}_{i\in \D}$ and $\a=a_1\cdots a_n\in \D^{*},$ let $$\phi_\a:=\phi_{a_1}\circ \cdots \circ \phi_{a_n}.$$ Let $|\a|$ denote the length of $\a\in \D^*$. If $\Phi$ has attractor $X,$ then for each $\a\in \D^*$ let $$X_{\a}:=\phi_{\a}(X).$$

\subsubsection{The set $W_{\Phi}(z,\Psi)$}
Given an IFS $\Phi$, $\Psi:\D^*\to [0,\infty),$ and an arbitrary $z\in X,$ we let $$W_{\Phi}(z,\Psi):=\Big\{x\in \mathbb{R}^d: |x-\phi_\a(z)|\leq \Psi(\a) \textrm{ for i.m. }\a \in \D^*\Big\}.$$ Throughout this paper we will always have the underlying assumption that $\Psi$ satisfies $$\lim_{n\to\infty}\max_{\a\in \D^n}\Psi(\a)= 0.$$ This condition guarantees $$W_{\Phi}(z,\Psi)\subseteq X.$$ The study of the metric properties of $W_{\Phi}(z,\Psi)$ will be one of the main focuses of this paper. Proceeding via analogy with Khintchine's theorem, it is natural to wonder what metric properties of $W_{\Phi}(z,\Psi)$ are encoded in the volume sum:
 \begin{equation}
\label{con/div}
\sum_{n=1}^{\infty}\sum_{\a\in \D^n}\Psi(\a)^{\dim_{H}(X)}.
\end{equation}It is an almost immediate consequence of the definition of Hausdorff measure that if we have convergence in \eqref{con/div}, then $\mathcal{H}^{\dim_{H}(X)}(W_{\Phi}(z,\Psi))=0$ for all $z\in X$. Given the results mentioned in the previous section, it is reasonable to expect that divergence in \eqref{con/div} might imply some metric property of $W_{\Phi}(z,\Psi)$ which demonstrates that a typical element of $X$ is contained in $W_{\Phi}(z,\Psi)$. A classification of those $\Psi$ for which divergence in \eqref{con/div} implies a typical element of $X$ is contained in $W_{\Phi}(z,\Psi)$ would provide a quantitative description of how the images of $z$ are distributed within $X$. This in turn provides a description of how the underlying iterated function system overlaps. This idea provides us with a new tool for describing the overlapping behaviour of iterated function systems. We refer the reader to Section \ref{new methods} for further discussions which demonstrate the utility of this idea.

The question of whether divergence in \eqref{con/div} implies a typical element of $X$ is contained in $W_{\Phi}(z,\Psi)$ was studied previously by the author in \cite{Bak2,Bak,Bakerapprox2}. Related work appears in \cite{LSV,PerRev,PerRevA}. In \cite{Bak2} the following theorem was proved:
\begin{thm}\cite[Theorem 1.4]{Bak2}
	\label{conformal theorem}
If $\Phi$ is a conformal iterated function system and satisfies the open set condition, then for any $z\in X,$ if $\theta:\mathbb{N}\to [0,\infty)$ is a decreasing function and satisfies $$\sum_{n=1}^{\infty} \sum_{\a\in\D^{n}} (Diam(X_\a)\theta(n))^{\dim_{H}(X)}=\infty,$$ then $\mathcal{H}^{\dim_{H}(X)}$-almost every $x\in X$ is contained in $W_{\Phi}(z,Diam(X_\a)\theta(|\a|)).$
\end{thm}
Note that for a conformal iterated function system it is known that the open set condition implies $0<\mathcal{H}^{\dim_{H}(X)}(X)<\infty$ (see \cite{PRSS}). For the definition of a conformal iterated function system see Section \ref{self-conformal section}. Note that an iterated function system consisting of similarities is automatically a conformal iterated function system. In \cite[Theorem 6.1]{Bak2} it was also shown that if $\Phi$ is a conformal iterated function system and contains an exact overlap, then there exist many natural choices of $\Psi$ such that we have divergence in \eqref{con/div}, yet $\dim_{H}(W_{\Phi}(z,\Psi))<\dim_{H}(X).$ As such an exact overlap effectively prevents any Khintchine like behaviour.

In \cite{Bak} and \cite{Bakerapprox2} the author studied the family of IFSs $\Phi_{\lambda}:=\{\lambda x,\lambda x +1\},$ where $\lambda\in(1/2,1)$. For each element of this family the corresponding attractor is $[0,\frac{1}{1-\lambda}].$  In \cite{Bak} the author proved that if the reciprocal of $\lambda$ belongs to a special class of algebraic integers known as Garsia numbers, then for a general class of $\Psi$, divergence in \eqref{con/div} implies that for all $z\in [0,\frac{1}{1-\lambda}],$ Lebesgue almost every $x\in [0,\frac{1}{1-\lambda}]$ is contained in $W_{\Phi_{\lambda}}(z,\Psi)$. For more on this result and Garsia numbers we refer the reader to Section \ref{Examples} where this result is recovered using a different argument. The main result of \cite{Bak2} provides strong evidence to suggest that for a general class of $\Psi$, for a typical $\lambda\in (1/2,1),$ we should expect that divergence in \eqref{con/div} implies that Lebesgue almost every $x\in [0,\frac{1}{1-\lambda}]$ is contained in $W_{\Phi_{\lambda}}(z,\Psi)$. A consequence of the main result of \cite{Bak2} is that for Lebesgue almost every $\lambda\in (1/2,0.668),$ for all $z\in [0,\frac{1}{1-\lambda}],$ Lebesgue almost every  $x\in [0,\frac{1}{1-\lambda}]$ is contained in $W_{\Phi_{\lambda}}(z,\frac{\log |\a|}{2^{|\a|}})$. Note that the results in \cite{Bak} and \cite{Bakerapprox2} are phrased for $z=0$ but can easily be adapted to the case of arbitrary $z\in[0,\frac{1}{1-\lambda}]$.

\subsubsection{The set $U_{\Phi}(z,\m,h)$}
\label{auxillary sets}
Instead of studying the sets $W_{\Phi}(z,\Psi)$ directly it is more profitable to study a related family of auxiliary sets. These sets are interesting in their own right and are defined in terms of a measure $\m$ supported on $\D^{\mathbb{N}}$. Our approach doesn't work for all $\m$ and we will require the following additional regularity assumption.

Given a probability measure $\m$ supported on $\D^{\mathbb{N}},$ we let
$$c_{\m}:={\textrm{ess}\inf}\inf_{k\in\mathbb{N}}\frac{\m([a_1\cdots a_{k+1}])}{\m([a_1\cdots a_k])}.$$ We say that $\m$ is slowly decaying if $c_{\m}>0$. If $\m$ is slowly decaying, then for $\m$-almost every $(a_j)\in \D^{\mathbb{N}},$ we have $$\frac{\m([a_1 \ldots a_{k+1}])}{\m([a_1\ldots a_k])}\geq c_{\m},$$ for all $k\in\mathbb{N}.$ Examples of slowly decaying measures include Bernoulli measures, and Gibbs measures for H\"{o}lder continuous potentials (see \cite{Bow}). In fact any measure with the quasi-Bernoulli property is slowly decaying. 

Given a slowly decaying probability measure $\m$, for each $n\in\mathbb{N}$ we let $$L_{\m,n}:=\{\a\in\D^*: \m([a_1\cdots a_{|\a|}])\leq c_{\m}^n<\m([a_1\cdots a_{|\a|-1}])\}$$ and $$R_{\m,n}:=\#L_{\m,n}.$$ The elements of $L_{\m,n}$ are disjoint and the union of their cylinders has full $\m$ measure. Importantly, by the slowly decaying property, the cylinders corresponding to elements of $L_{\m,n}$ have comparable measure up to a multiplicative constant. Note that when $\m$ is the uniform $(1/l,\ldots,1/l)$ Bernoulli measure the set $L_{\m,n}$ is simply $\D^n$. 

Given $z\in X$ and a slowly decaying probability measure $\m,$ we let $$Y_{\m,n}(z):=\{\phi_{\a}(z)\}_{\a\in L_{\m,n}}.$$ Obtaining information on how the elements of $Y_{\m,n}(z)$ are distributed within $X$ for different values of $n$ will occupy a large part of this paper.

Given a slowly decaying measure $\m,$ an IFS $\Phi$, $h:\mathbb{N}\to[0,\infty),$ and $z\in X,$ we can define a limsup set as follows. Let 
$$U_{\Phi}(z,\m,h):=\left\{x\in \mathbb{R}^d:x\in \bigcup_{\a\in L_{\m,n}}B\left(\phi_{\a}(z),(\m([\a])h(n))^{1/d}\right) \textrm{ for i.m. } n\in \N\right\}.$$ Here and throughout $B(x,r)$ denotes the closed Euclidean ball centred at $x$ with radius $r$.
Throughout this paper we will always assume that $\m$ is non-atomic and $h$ is a bounded function. These properties ensure $$U_{\Phi}(z,\m,h)\subseteq X.$$ In this paper we study the metric properties of the sets $U_{\Phi}(z,\m,h)$ for parameterised families of IFSs when the underlying attractor typically has positive $d$-dimensional Lebesgue measure. In which case, for the set $U_{\Phi}(z,\m,h),$ the appropriate volume sum that we expect to determine the Lebesgue measure of $U_{\Phi}(z,\m,h)$ is $$\sum_{n=1}^{\infty}h(n).$$ It can be shown using the Borel-Cantelli lemma that if $\sum_{n=1}^{\infty}h(n)<\infty,$ then $U_{\Phi}(z,\m,h)$ has zero Lebesgue measure. For us the interesting question is: When does $\sum_{n=1}^{\infty}h(n)=\infty$ imply that $U_{\Phi}(z,\m,h)$ has positive or full Lebesgue measure?

The sets $U_{\Phi}(z,\m,h)$ are easier to work with than the sets $W_{\Phi}(z,\Psi)$. In particular we can use properties of the measure $\m$ to aid with our analysis. As we will see, the sets $U_{\Phi}(z,\m,h)$ can be used to prove results for the sets $W_{\Phi}(z,\Psi),$ but only under the following additional assumption. Given a slowly decaying measure $\m$ and $h:\mathbb{N}\to[0,\infty),$ we say that $\Psi$ is equivalent to $(\m,h)$  if
$$\Psi(\a)\asymp \big(\m([\a])h(n)\big)^{1/d}$$
for each $\a\in L_{\m,n}$ for all $n\in \N$. Here and throughout, for two real valued functions $f$ and $g$ defined on some set $S$, we write $f\asymp g$ if there exists a positive constant $C$ such that $$C^{-1}\cdot g(x)\leq f(x)\leq Cg(x)$$ for all $x\in S$. As we will see, if $\Psi$ is equivalent to $(\m,h)$ and $U_{\Phi}(z,\m,h)$ has positive Lebesgue measure, then $W_{\Phi}(z,\Psi)$ will also have positive Lebesgue measure (see Lemma \ref{arbitrarily small}).

\section{Statement of results}

\label{statement of results}

Before stating our theorems we need to define the entropy of a measure $\m$ supported on $\D^{\mathbb{N}}$ and introduce a class of functions that are the natural setting for some of our results.

For any $\sigma$-invariant measure $\m$ supported on $\D^{\N},$ we define the entropy of $\m$ to be
	$$\mathfrak{h}(\m):=\lim_{n\to\infty}-\frac{1}{n}\sum_{\a\in\D^n}\m([\a])\log \m([\a]).$$ The entropy of a $\sigma$-invariant measure always exists. 

Given a set $B\subset \mathbb{N},$ we define the lower density of $B$ to be $$\underline{d}(B):=\liminf_{n\to\infty}\frac{\#\{1\leq j\leq n:j\in B\}}{n},$$ and the upper density of $B$ to be
$$\overline{d}(B):=\limsup_{n\to\infty}\frac{\#\{1\leq j\leq n:j\in B\}}{n}.$$ Given  $\epsilon>0,$ let
$$H_{\epsilon}^*:=\left\{h:\mathbb{N}\to[0,\infty):\sum_{n\in B}h(n)=\infty\,, \forall B\subseteq \N \textrm{ s.t. } \underline{d}(B)>1-\epsilon\right\}$$ and
 $$H_{\epsilon}:=\left\{h:\mathbb{N}\to[0,\infty):\sum_{n\in B}h(n)=\infty\,, \forall B\subseteq \N \textrm{ s.t. } \overline{d}(B)>1-\epsilon\right\}.$$ We also define
 \begin{equation}
 \label{H^* functions}
  H^*:=\bigcup_{\epsilon\in(0,1)}H^*_{\epsilon}
 \end{equation}and 
\begin{equation}
\label{H functions} H:=\bigcup_{\epsilon\in(0,1)}H_{\epsilon}.
\end{equation}For any $\epsilon>0$ we have $H_{\epsilon}\subset H_{\epsilon}^*.$ Therefore $H\subset H^*.$ It can be shown that $H^*$ contains all decreasing functions satisfying $\sum_{n=1}^{\infty}h(n)=\infty$. Most of the time we will be concerned with the class of functions $H$. The class $H^*$ will only appear in Theorem \ref{precise result}.

We say that a function $\Psi:\D^*\to [0,\infty)$ is weakly decaying if $$\inf_{\a\in \D^*}\min_{i\in\D}\frac{\Psi(i\a)}{\Psi(\a)}>0.$$ Given a measure $\m$ supported on $\D^{\N},$ we let $$\Upsilon_{\m}:=\left\{\Psi:\D^*\to[0,\infty): \Psi \textrm{ is weakly decaying and equivalent to }(\m,h) \textrm{ for some }h\in H\right\}.$$ As we will see, the weakly decaying property will allow us to obtain full measure statements.
\subsection{Parameterised families with variable contraction ratios}
Let $D:=\{d_1,\ldots, d_{l}\}$ be a finite set of real numbers. To each $\lambda\in(0,1),$ we associate the iterated function system $$\Phi_{\lambda,D}:=\left\{\phi_i(x)=\lambda x + d_i\right\}.$$ It is straightforward to check that the corresponding attractor for $\Phi_{\lambda,D}$ is $$X_{\lambda,D}:=\left\{\sum_{j=0}^{\infty}d_j\lambda^j: d_j\in D\right\},$$ and the projection map $\pi_{\lambda,D}:\D^{\mathbb{N}}\to X_{\lambda,D}$ takes the form $$\pi_{\lambda,D}((a_j)_{j=1}^{\infty})=\sum_{j=1}^{\infty}d_{a_{j}}\lambda^{j-1}.$$ To study this family of iterated function systems, it is useful to study the set $\Gamma:=D-D$ and the corresponding class of power series $$\mathcal{B}_{\Gamma}:=\Big\{g(x)=\sum_{j=0}^{\infty}g_jx^j:g_j\in \Gamma\Big\}.$$ To each $\mathcal{B}_{\Gamma}$ we associate the set $$\Lambda(\mathcal{B}_{\Gamma}):=\left\{\lambda\in(0,1):\exists g\in \mathcal{B}_{\Gamma}, g\not\equiv 0, g(\lambda)=g'(\lambda)=0\right\}.$$ In other words, $\Lambda(\mathcal{B}_{\Gamma})$ is the set of $\lambda\in(0,1)$ that can be realised as a double zero for a non-trivial function in $\mathcal{B}_{\Gamma}.$ We let $$\alpha(\mathcal{B}_{\Gamma})=\inf \Lambda(\mathcal{B}_{\Gamma}),$$ if $\Lambda(\mathcal{B}_{\Gamma})\neq \emptyset,$ and let $\alpha(\mathcal{B}_{\Gamma})=1$ otherwise. 

These families of iterated function systems were originally studied by Solomyak in \cite{Sol}. He was interested in the absolute continuity of self-similar measures. In particular, he was interested in the pushforward of the uniform $(1/l,\ldots,1/l)$ Bernoulli measure. We denote this measure by $\mu_{\lambda,D}.$ The main result of \cite{Sol} is the following theorem.
\begin{thm}
\label{Solomyak transverality theorem}
For Lebesgue almost every $\lambda\in(1/l,\alpha(\mathcal{B}_{\Gamma})),$ the measure  $\mu_{\lambda,D}$ is absolutely continuous and has a density in $L^{2}(\mathbb{R})$. 
\end{thm}Using Theorem \ref{Solomyak transverality theorem}, Solomyak proved the well known result that for Lebesgue almost every $\lambda\in(1/2,1),$ the unbiased Bernoulli convolution is absolutely continuous and has a density in $L^{2}(\mathbb{R})$. As a by-product of our analysis, in Section \ref{applications} we give a short intuitive proof that for Lebesgue almost every $\lambda\in(1/2,1),$ the unbiased Bernoulli convolution is absolutely continuous. Instead of using the Fourier transform or by differentiating measures, as in \cite{Sol} and \cite{PerSol}, our proof makes use of the fact that self-similar measures are of pure type, i.e. they are either singular or absolutely continuous with respect to the Lebesgue measure. As a further by-product of our analysis, in Section \ref{applications} we recover another result of Solomyak from \cite{Sol}. We prove that for Lebesgue almost every $\lambda\in(1/3,2/5),$ the set $$C_{\lambda}:=\left\{\sum_{j=0}^{\infty}d_j\lambda^j: d_j\in \{0,1,3\}\right\}$$
has positive Lebesgue measure. Interestingly our proof of this statement does not rely on showing that there is an absolutely continuous measure supported on this set. Instead we study a subset of this set, and show that for Lebesgue almost every $\lambda\in(1/3,2/5),$ this set has positive Lebesgue measure.

For the families of iterated function systems introduced in this section, our main result is the following.
\begin{thm}
	\label{1d thm}
	Let $D$ be a finite set of real numbers. The following statements are true:
\begin{enumerate}	\item Let $\m$ be a slowly decaying $\sigma$-invariant ergodic probability measure with $\mathfrak{h}(\m)>0$ and $(a_j)\in \D^{\N}.$ For Lebesgue almost every $\lambda\in(e^{-\h(\m)},\alpha(\mathcal{B}_{\Gamma})),$ for any $h\in H$ the set $U_{\Phi_{\lambda,D}}(\sum_{j=1}^{\infty}d_{a_j}\lambda^{j-1},\m,h)$ has positive Lebesgue measure. 
\item Let $\m$ be the uniform $(1/l,\cdots, 1/l)$ Bernoulli measure. For Lebesgue almost every $\lambda\in(1/l,\alpha(\mathcal{B}_{\Gamma})),$ for any $z\in X_{\lambda,D}$ and $h\in H$, the set $U_{\Phi_{\lambda,D}}(z,\m,h)$ has positive Lebesgue measure.
\item Let $\m$ be a slowly decaying $\sigma$-invariant ergodic probability measure with $\mathfrak{h}(\m)>0$ and $(a_j)\in \D^{\N}.$ For Lebesgue almost every $\lambda\in(e^{-\h(\m)},\alpha(\mathcal{B}_{\Gamma})),$ for any $\Psi\in \Upsilon_{\m}$ Lebesgue almost every $x\in X_{\lambda,D}$ is contained in $W_{\Phi_{\lambda,D}}(\sum_{j=1}^{\infty}d_{a_j}\lambda^{j-1},\Psi).$ 
\item Let $\m$ be the uniform $(1/l,\cdots, 1/l)$ Bernoulli measure. For Lebesgue almost every $\lambda\in(1/l,\alpha(\mathcal{B}_{\Gamma})),$ for any $z\in X_{\lambda,D}$ and $\Psi\in \Upsilon_{\m},$ Lebesgue almost every $x\in X_{\lambda,D}$ is contained in $W_{\Phi_{\lambda,D}}(z,\Psi).$ 
	\end{enumerate}
\end{thm} 
To aid with our exposition we will prove in Section \ref{applications} the following corollary to Theorem \ref{1d thm}.

\begin{cor}
	\label{example cor}
Let $D$ be a finite set of real numbers and $\m$ be a Bernoulli measure corresponding to the probability vector $(p_1,\ldots,p_l)$. Then for any $(a_j)\in \D^{\N}$, for Lebesgue almost every $\lambda\in(\prod_{i=1}^lp_i^{p_i},\alpha(\mathcal{B}_{\Gamma})),$ Lebesgue almost every $x\in X_{\lambda,D}$ is contained in the set $$\Big\{x\in X: \Big|x-\phi_{\a}\Big(\sum_{j=1}^{\infty}d_{a_j}\lambda^{j-1}\Big)\Big|\leq \frac{\prod_{j=1}^{|\a|}p_{a_j}}{|\a|} \textrm{ for i.m. }\a\in \D^{*}\Big\}.$$
\end{cor}
In Section \ref{applications} we will apply these results to obtain more explicit statements in the setting of Bernoulli convolutions and the $\{0,1,3\}$ problem.

Certain lower bounds for the transversality constant $\alpha(B_{\Gamma})$ are known. Let $D$ be a finite set of real numbers and assume $d_j\neq d_k$ for all $j\neq k,$ so $$b(D):=\max\left\{\left|\frac{d_j-d_l}{d_k-d_i}\right|:k\neq i\right\}<\infty.$$ The proposition stated below provides a summary of the lower bounds obtained separately in \cite{PerSol2}, \cite{PolSimon}, and \cite{ShmSol2}.
\begin{prop}
	\label{transversality constants}
Let $D$ be a finite set of real numbers and $b(D)$ be as above. Then the following statements are true:
\begin{itemize}
	\item If $b(D)=1$ then $\alpha(B_{\Gamma})>0.668.$
	\item If $b(D)=2$ then $\alpha(B_{\Gamma})=0.5.$
	\item $\alpha(B_{\Gamma})= (b(D)+1)^{-1}$ whenever $b(D)\geq 3+\sqrt{8}.$
	\item $\alpha(B_{\Gamma})\geq  (b(D)+1)^{-1}$ for all $D$.
\end{itemize}
\end{prop}
\subsection{Parameterised families with variable translations}

Suppose $\{A_i\}_{i=1}^l$ is a collection of $d\times d$ non-singular matrices each satisfying $\|A_i\|<1.$ Here $\|\cdot\|$ denotes the operator norm induced by the Euclidean norm. Given a vector $\t=(t_1,\ldots, t_l)\in \mathbb{R}^{ld}$ we can define an IFS to be the set of contractions $$\Phi_{\t}:=\{\phi_{i}(x)=A_i x+ t_i\}_{i=1}^l.$$  Unlike in the previous section where we obtained a family of iterated function systems by varying the contraction ratio, here we obtain a family by varying the translation parameter $\t$. For each $\t\in\mathbb{R}^{ld}$ we denote the attractor by $X_\t,$ and the corresponding projection map from $\D^\N$ to $X_\t$ by $\pi_{\t}$. The attractor $X_\t$ is commonly referred to as a self-affine set.

This family of iterated function systems was introduced by Falconer in \cite{Fal}, and subsequently studied by Solomyak in \cite{Sol2}, and later by Jordan, Pollicott, and Simon in \cite{JoPoSi}. For this family an important result is the following.

\begin{thm}[Falconer \cite{Fal}, Solomyak \cite{Sol2}]
\label{FalSol}
Assume the $A_i$ satisfy the additional hypothesis that $\|A_i\|<1/2$ for all $1\leq i\leq l$. Then for Lebesgue almost every $\t\in\mathbb{R}^{ld}$ the attractor $X_\t$ satisfies:
$$\dim_{H}(X_\t)=\dim_{B}(X_\t)=\min\{\dim_{A}(A_1,\ldots,A_l),d\}.$$
\end{thm}Here $\dim_{A}(A_1,\ldots,A_l)$ is a quantity known as the affinity dimension. For its definition see \cite{Fal}.
Theorem \ref{FalSol} was originally proved by Falconer in \cite{Fal} under the assumption $\|A_i\|<1/3$ for all $1\leq i\leq l$. This upper bound was improved to $1/2$ by Solomyak in \cite{Sol2}. The bound $1/2$ is known to be optimal (see \cite{Edg,SiSo}). An analogue of Theorem \ref{FalSol} for measures was obtained by Jordan, Pollicott, and Simon in \cite{JoPoSi}. A recent result of B\'{a}r\'{a}ny, Hochman, and Rapaport \cite{BaHoRa} significantly improves upon Theorem \ref{FalSol}. They proved that we have $\dim_{H}(X_\t)=\dim_{B}(X_\t)=\min\{\dim_{A}(A_1,\ldots,A_l),d\}$ under some very general assumptions on the $A_i$ and $\t$. In particular, their result gives rise to many explicit examples where equality is satisfied.

Given $\a=a_1\cdots a_n\in \D^*,$ we let $$A_{\a}:=A_{a_1}\circ \cdots \circ A_{a_{n}},$$ and $$1>\alpha_1(A_{\a})\geq \alpha_2(A_{\a})\geq \cdots \geq \alpha_d(A_{\a})>0$$ denote the singular values of $A_{\a}$. The singular values of a non-singular matrix $A$ are the positive square roots of the eigenvalues of $AA^{T}.$ Alternatively they are the lengths of the semiaxes of the ellipse $A(B(0,1)).$ Given a $\sigma$-invariant ergodic probability measure $\m,$ then there exist positive constants $\lambda_{1}(\m),\cdots, \lambda_{d}(\m),$ such that for $\m$-almost every $(a_j)\in \D^{\mathbb{N}}$ we have $$\lim_{n\to\infty}\frac{\log \alpha_{k}(A_{a_1\cdots a_n})}{n}=\lambda_{k}(\m),$$ for all $1\leq k\leq d$. We call the numbers $\lambda_{1}(\m),\cdots, \lambda_{d}(\m)$ the Lyapunov exponents of $\m$. The existence of Lyapunov exponents for $\sigma$-invariant ergodic measures $\m$ was established in \cite{JoPoSi}. 

The theorem stated below is our main result for this family of iterated function systems.

\begin{thm}
	\label{translation thm}
Suppose $\|A_i\|<1/2$ for all $1\leq i\leq l$. Then the following statements are true:
\begin{enumerate}
	\item Let $\m$ be a slowly decaying $\sigma$-invariant ergodic probability measure with $\h(\m)>-(\lambda_1(\m)+\cdots +\lambda_d(\m))$ and $(a_j)\in \D^{\N}.$ For Lebesgue almost every $\t\in\mathbb{R}^{ld}$, for any $h\in H$ the set $U_{\Phi_\t}(\pi_\t(a_j),\m,h)$ has positive Lebesgue measure. 
	\item Let $\m$ be the uniform $(1/l,\ldots,1/l)$ Bernoulli measure and suppose there exists $A$ such that $A_i=A$ for any $1\leq i\leq l $. 
	If $\log l> -(\lambda_1(\m)+\cdots +\lambda_d(\m)),$
	then for Lebesgue almost every $\t\in\mathbb{R}^{ld}$, for any $z\in X_{\t}$ and $h\in H$, the set $U_{\Phi_\t}(z,\m,h)$ has positive Lebesgue measure.
	\item Let $\m$ be a slowly decaying $\sigma$-invariant ergodic probability measure and $(a_j)\in \D^{\N}$. Suppose that $\h(\m)>-(\lambda_1(\m)+\cdots +\lambda_d(\m))$ and one of the following three properties are satisfied:
	\begin{itemize}
		\item Each $A_i$ is a similarity.
		\item $d=2$ and all the matrices $A_i$ are equal.
		\item All the matrices $A_i$ are simultaneously diagonalisable. 
	\end{itemize} Then for Lebesgue almost every $\t\in\mathbb{R}^{ld}$, for any $\Psi\in \Upsilon_{\m},$ Lebesgue almost every $x\in X_\t$ is contained in $W_{\Phi_\t}(\pi_\t(a_j),\Psi).$ 
    \item Let $\m$ be the uniform $(1/l,\ldots,1/l)$ Bernoulli measure and suppose there exists $A$ such that $A_i=A$ for any $1\leq i\leq l $. Suppose that $\log l>-(\lambda_1(\m)+\cdots +\lambda_d(\m))$ and one of the following three properties are satisfied:
    \begin{itemize}
    	\item $A$ is a similarity.
    	\item $d=2$.
    	\item The matrix $A$ is diagonalisable. 
    \end{itemize} Then for Lebesgue almost every $\t\in\mathbb{R}^{ld}$, for any $z\in X_\t$ and $\Psi\in \Upsilon_{\m}$, Lebesgue almost every $x\in X_\t$ is contained in $W_{\Phi_\t}(z,\Psi).$
	\end{enumerate}
\end{thm}
The following corollary follows immediately from Theorem \ref{translation thm}.
\begin{cor}
	\label{translation cor}
Suppose there exists $\lambda\in(0,1/2)$ and $O\in O(d)$ such that $A_i=\lambda\cdot O$ for all $1\leq i \leq l$. Then if $\frac{\log l}{-\log \lambda}>d,$ we have that for Lebesgue almost every $\t\in \R^{ld}$, for any $z\in X_{\t},$ Lebesgue almost every $x\in X_{\t}$ is contained in the set $$\left\{x\in\mathbb{R}^d:|x-\phi_{\a}(z)|\leq \left( \frac{l^{-|\a|}}{|\a|}\right)^{1/d}\textrm{ for i.m. }\a\in \D^{*}\right\}.$$
\end{cor}
The assumption $\|A_i\|<1/2$ appearing in Theorem \ref{translation thm} is necessary as the example below shows.

\begin{example}
	\label{counterexample}
Consider the iterated function system $\Phi_{\lambda,t_1,t_2}=\{\lambda x +t_1, \lambda x +t_2\},$ where $\lambda\in(1/2,1)$ and $t_1,t_2\in\mathbb{R}$. Whenever $t_1\neq t_2$ we can apply a change of coordinates and identify this iterated function system with $\{\lambda x, \lambda x +1\}.$ For any $\epsilon>0,$ there exists $\lambda^*\in(1/2,1/2+\epsilon)$ such that $\{\lambda x, \lambda x +1\}$ contains an exact overlap. Using this fact and our change of coordinates, it can be shown that $U_{\Phi_{\lambda^*,t_1,t_2}}(\pi(a_j),\m,h)$ has zero Lebesgue measure when $\m$ is the $(1/2,1/2)$ Bernoulli measure and $h$ is any bounded function. 
\end{example} Even though Example \ref{counterexample} demonstrates the condition $\|A_i\|<1/2$ is essential, the author expects Theorem \ref{translation thm} to hold more generally. In this paper we prove a random version of Theorem \ref{translation thm} which supports this claim. This random version is based upon the randomly perturbed self-affine sets studied in \cite{JoPoSi}. Our setup is taken directly from \cite{JoPoSi}.

Fix a set of matrices $\{A_i\}_{i=1}^l$ each satisfying $\|A_i\|<1,$ and a vector $\t=(t_1,\ldots,t_l)\in\mathbb{R}^{ld}$. We obtain a randomly perturbed version of the IFS $\Phi=\{\phi_i(x)=A_i x +t_i\}$ in the following way. Suppose that $\eta$ is an absolutely continuous distribution with density supported on a disc $\mathbf{D}$. The distribution $\eta$ gives rise to a random perturbation of $\phi_{\a}$ via the equation 
$$\phi_{\a}^{y_\a}:=(\phi_{a_1}+y_{a_1})\circ (\phi_{a_2}+y_{a_1a_2})\circ \cdots \circ (\phi_{\a_{|\a|}}+y_{\a}),$$ where the coordinates of $$(y_{a_1},y_{a_1a_2},\ldots,y_{\a})\in \mathbf{D}\times \cdots \times \mathbf{D}$$ are i.i.d. with distribution $\eta$. For notational convenience we enumerate the errors using the natural numbers. Let $\rho:\D^*\to \mathbb{N}$ be an arbitrary bijection. We obtain a sequence of errors $\y=(y_k)_{k=1}^{\infty}\in \mathbf{D}^{\N}$ according to the rule $$y_k:=y_{\a}\textrm{ if }\rho(\a)=k.$$ Given $\y\in \mathbf{D}^{\mathbb{N}},$ we obtain a perturbed version of our original attractor defined via the equation $$X_{\y}:=\bigcap_{n=1}^{\infty}\bigcup_{\a\in \D^n}\phi_{\a}^{y_\a}(B),$$ where $B$ is some sufficiently large ball. We let $\pi_\y:\D^{\N}\to X_\y$ be the projection map given by $$\pi_\y(a_j):=\lim_{n\to\infty}\phi_{a_1\cdots a_n}^{y_{a_1\cdots a_n}}(\textbf{0}).$$ On $\mathbf{D}^{\mathbb{N}}$ we define the measure $$\mathbf{P}:=\eta \times \cdots \times \eta \times \cdots.$$ 

We may now define our limsup sets for these randomly perturbed attractors. Given $\y\in \mathbf{D}^{\mathbb{N}},$ $(a_j)\in \D^{\N},$ and $\Psi:\D^{*}\to [0,\infty),$ we define 
$$W_{\Phi,\y}((a_j),\Psi):=\left\{x\in \mathbb{R}^d: |x-\pi_\y(\a (a_j))|\leq \Psi(\a) \textrm{ for i.m. }\a\in \D^{*}\right\}.$$ Here $\a (a_j)$ denotes the element of $\D^{\N}$ obtained by concatenating the finite word $\a$ with the infinite sequence $(a_j).$ Given a slowly decaying measure $\m$, $\y\in \D^{\mathbb{N}},$ $(a_j)\in \D^{\N},$ and $h:\mathbb{N}\to [0,\infty),$ we let $$U_{\Phi,\y}((a_j),\m,h):=\left\{x\in \mathbb{R}^d: x\in \bigcup_{\a\in L_{\m,n}}B\left(\pi_\y(\a (a_j)),(\m([\a])h(n))^{1/d}\right)\textrm{ for i.m. }n\in \N\right\}.$$ 

 The sets $W_{\Phi,\y}((a_j),\Psi)$ and $U_{\Phi,\y}((a_j),\m,h)$ serve as our analogues of $W_{\Phi}(z,\Psi)$ and $U_{\Phi}(z,\m,h)$ in this random setting. Note here that we have defined our limsup sets in terms of neighbourhoods of $\pi_\y(\a (a_j))$ rather than $\phi_{\a}^{\y_{\a}}(\pi_{\y}(a_j)).$ In the deterministic setting considered above these quantities coincide. In the random setup it is not necessarily the case that  $\pi_\y(\a (a_j))=\phi_{\a}^{\y_{\a}}(\pi_{\y}(a_j)).$ The theorem stated below is the random analogue of Theorem \ref{translation thm}. It suggests that one should be able to replace the assumption $\|A_i\|<1/2$ with some other reasonable conditions. 

\begin{thm}
	\label{random thm}
Fix a set of matrices $\{A_i\}_{i=1}^l$ each satisfying $\|A_i\|<1$ and $\t\in \R^{ld}$. Then the following statements are true:
	\begin{enumerate}
			\item Let $\m$ be a slowly decaying $\sigma$-invariant ergodic probability measure with $\h(\m)>-(\lambda_1(\m)+\cdots +\lambda_d(\m))$ and $(a_j)\in \D^{\N}.$ For $\mathbf{P}$-almost every $\y\in\mathbf{D}^{\mathbb{N}},$ for any $h\in H,$ the set $U_{\Phi,\y}((a_j),\m,h)$ has positive Lebesgue measure.
			\item Let $\m$ be the uniform $(1/l,\ldots,1/l)$ Bernoulli measure and suppose there exists $A$ such that $A_i=A$ for all $1\leq i\leq l$. 
		   If $\log l>-(\lambda_1(\m)+\cdots+\lambda_d(\m))$, then for $\mathbf{P}$-almost every $\y\in\mathbf{D}^{\mathbb{N}}$, for any $(a_j)\in \D^{\mathbb{N}}$ and $h\in H,$ the set $U_{\Phi,\y}((a_j),\m,h)$ has positive Lebesgue measure.
		  
		\item Let $\m$ be a slowly decaying $\sigma$-invariant ergodic probability measure with $\h(\m)>-(\lambda_1(\m)+\cdots +\lambda_d(\m))$ and $(a_j)\in \D^{\N}.$ For $\mathbf{P}$-almost every $\y\in\mathbf{D}^{\mathbb{N}}$, for any $\Psi$ that equivalent to $(\m,h)$ for some $h\in H,$ the set $W_{\Phi,\y}((a_j),\Psi)$ has positive Lebesgue measure.
		\item  Let $\m$ be the uniform $(1/l,\ldots,1/l)$ Bernoulli measure and suppose there exists $A$ such that $A_i=A$ for all $1\leq i\leq l$. If $\log l>-(\lambda_1(\m)+\cdots+\lambda_d(\m))$, then for $\mathbf{P}$-almost every $\y\in\mathbf{D}^{\mathbb{N}}$, for any $(a_j)\in \D^{\mathbb{N}}$ and $\Psi$ that is equivalent to $(\m,h)$ for some $h\in H,$ the set $W_{\Phi,\y}((a_j),\Psi)$ has positive Lebesgue measure.
				\end{enumerate}
		\end{thm}
The reason we cannot obtain the full measure statements from Theorem \ref{translation thm} in our random setting is because of how $X_\y$ is defined. In particular, $X_\y$ cannot necessarily be expressed as finitely many scaled copies of itself like in the deterministic setting. The proof of statements $3$ and $4$ from Theorem \ref{translation thm} rely on the fact that the underlying attractor satisfies the equation $X=\cup_{i=1}^l\phi_i (X)$. 	

\subsection{A specific family of IFSs}
We now introduce a family of iterated function systems for which we can make very precise statements. To each $t\in[0,1]$ we associate the IFS:
$$\Phi_{t}=\left\{\phi_1(x)=\frac{x}{2},\, \phi_{2}(x)=\frac{x+1}{2},\, \phi_{3}(x)=\frac{x+t}{2},\, \phi_{4}(x)=\frac{x+1+t}{2}\right\}.$$ For each $\Phi_t$ the corresponding attractor is $[0,1+t]$. We denote the projection map from $\D^{\mathbb{N}}$ to $[0,1+t]$ by $\pi_t$. For this family of iterated function systems we will be able to replace the almost every statements appearing in Theorems \ref{1d thm} and \ref{translation thm} with something more precise. The reason we can make these stronger statements is because separation properties for $\Phi_t$ can be deduced from the continued fraction expansion of $t$. Recall that for any $t\in [0,1]\setminus \mathbb{Q},$ there exists a unique sequence $(\zeta_m)\in\mathbb{N}^{\mathbb{N}}$ such that 
$$ t=\cfrac{1}{\zeta_1+\cfrac{1}{\zeta_2 +\cfrac{1}
		{\zeta_3 + \cdots }}}.$$
We call the sequence $(\zeta_m)$ the continued fraction expansion of $t$. Given $t$ with continued fraction expansion $(\zeta_m)$, for each $m\in\mathbb{N}$ we let
$$ \frac{p_m}{q_m}:=\cfrac{1}{\zeta_1+\cfrac{1}{\zeta_2 +\cfrac{1}
		{\zeta_3 + \cdots \cfrac{1}
			{\zeta_m }}}}.$$ We call $p_m/q_m$ the $m$-th partial quotient of $t$. We say that $t$ is badly approximable if the integers appearing in the continued fraction expansion of $t$ can be bounded from above.

The main result of this section is the following.
\begin{thm}
	\label{precise result} 
Let $\m$ be the uniform $(1/4,1/4,1/4,1/4)$ Bernoulli measure. The following statements are true:
\begin{enumerate}
	\item If $t\in\mathbb{Q}$ then $\Phi_t$ contains an exact overlap, and for any $z\in[0,1+t]$ the set $U_{\Phi_t}(z,\m,1)$ has Hausdorff dimension strictly less than $1$.
	\item If $t\notin \mathbb{Q},$ then there exists $h:\mathbb{N}\to[0,\infty)$ depending upon the continued fraction expansion of $t,$ such that $\lim_{n\to\infty}h(n)=0,$ and for any $z\in[0,1+t],$ Lebesgue almost every $x\in[0,1+t]$ is contained in $U_{\Phi_t}(z,\m,h)$.
	\item If $t$ is badly approximable, then for any $z\in[0,1+t]$ and $h:\mathbb{N}\to [0,\infty)$ satisfying $\sum_{n=1}^{\infty}h(n)=\infty,$ Lebesgue almost every $x\in[0,1+t]$ is contained in $U_{\Phi_t}(z,\m,h).$
	\item If $t\notin\mathbb{Q}$ and is not badly approximable, then there exists $h:\mathbb{N}\to[0,\infty)$ satisfying $\sum_{n=1}^{\infty}h(n)=\infty,$ yet $U_{\Phi_t}(z,\m,h)$ has zero Lebesgue measure for any $z\in[0,1+t]$.
	\item Suppose $t\notin \mathbb{Q}$ is such that for any $\epsilon>0,$ there exists $L\in\mathbb{N}$ for which the following inequality holds for $M$ sufficiently large: $$\sum_{\stackrel{1\leq m \leq M}{\frac{q_{m+1}}{q_m}\geq L}}\log_{2}(\zeta_{m+1}+1) \leq \epsilon M.$$ Then for any $z\in[0,1+t]$ and $h\in H^*,$ Lebesgue almost every $x\in[0,1+t]$ is contained in $U_{\Phi_t}(z,\m,h)$. 
	\item Suppose $\mu$ is an ergodic invariant measure for the Gauss map and satisfies $$\sum_{m=1}^{\infty}\mu\Big(\left[\frac{1}{m+1},\frac{1}{m}\right]\Big)\log_2 (m +1)<\infty.$$ Then for $\mu$-almost every $t,$ for any $z\in[0,1+t]$ and $h\in H^*,$ Lebesgue almost every $x\in[0,1+t]$ is contained in $U_{\Phi_t}(z,\m,h).$ In particular, for Lebesgue almost every $t\in[0,1],$ for any $z\in[0,1+t]$ and $h\in H^*$, Lebesgue almost every $x\in[0,1+t]$ is contained in $U_{\Phi_t}(z,\m,h)$.
\end{enumerate}
\end{thm}
We include the following corollary to emphasise the strong dichotomy that follows from statement $1$ and statement $2$ from Theorem \ref{precise result}. 

\begin{cor}
Either $t$ is such that $\Phi_t$ contains an exact overlap and for any $z\in[0,1+t]$ $$\dim_{H}\left(\left\{x\in[0,1+t]:|x-\phi_{\a}(z)|\leq \frac{1}{4^{|\a|}}\textrm{ for i.m. }\a\in\D^*\right\}\right)<1,$$ or for any $z\in[0,1+t],$ Lebesgue almost every $x\in [0,1+t]$ is contained in $$\left\{x\in[0,1+t]:|x-\phi_{\a}(z)|\leq \frac{1}{4^{|\a|}}\textrm{ for i.m. }\a\in\D^*\right\}.$$
\end{cor}
Theorem \ref{precise result} is stated in terms of the auxiliary sets $U_{\Phi_t}(z,\m,h)$ rather than in terms of $W_{\Phi_t}(z,\Psi).$ Where here the underlying measure $\m$ is the uniform $(1/4,1/4,1/4,1/4)$ Bernoulli measure. Note however that if $\Psi:\D^*\to[0,\infty)$ is a function that only depends upon the length of the word $\a$, then $\Psi(\a)=h(|\a|)\m([\a])$ for some appropriate choice of $h$. Combining this observation with the fact $L_{\m,n}=\D^n$ for any $n\in \N$ for this choice of $\m$, it follows that $W_{\Phi_t}(z,\Psi)=U_{\Phi_t}(z,\m,h)$ for this choice of $h$. Therefore Theorem \ref{precise result} can be reinterpreted in terms of the sets $W_{\Phi_t}(z,\Psi)$ when $\Psi$ only depends upon the length of the word.

\subsubsection{New methods for distinguishing between the overlapping behaviour of IFSs}
\label{new methods}
In this section we explain how Theorem \ref{precise result} allows us to distinguish between iterated function systems in a way that is not available to us by simply studying properties of self-similar measures. We start this discussion by stating the following result that will follow from the proof of Theorem \ref{precise result}.
\begin{thm}
	\label{overlap or optimal}
Let $t\in[0,1].$ It is the case that either $\Phi_t$ contains an exact overlap, or for infinitely many $n\in\mathbb{N}$ we have $$|\phi_{\a}(z)-\phi_{\a'}(z)|\geq \frac{1}{8\cdot 4^n},$$ for any $z\in [0,1+t]$ for distinct $\a,\a'\in \D^n$.
\end{thm}
Theorem \ref{overlap or optimal} effectively states that for this family of IFSs, we either have an exact overlap, or for infinitely many scales we exhibit the optimal level of separation. This level of separation can be seen to be optimal by the pigeonhole principle, which tells us that for any $z\in[0,1+t]$ and $n\in\mathbb{N},$ there must exist distinct $\a,\a'\in \D^n$ such that $$|\phi_{\a}(z)-\phi_{\a'}(z)|\leq \frac{1+t}{4^n-1}.$$  Because of the strong dichotomy demonstrated by Theorem \ref{overlap or optimal}, we believe that this family of IFSs will serve as a useful toy model for other problems.

For a probability vector $\textbf{p}=(p_1,p_2,p_3,p_4)$ we denote the corresponding self-similar measure for the IFS $\Phi_t$ by $\mu_{\textbf{p},t}$. It follows from Theorem \ref{overlap or optimal} and the work of Hochman \cite[Theorem 1.1.]{Hochman} that the following theorem holds.
\begin{thm}
	\label{Hoccor}
	Either $\Phi_t$ contains an exact overlap, or for any probability vector $\textbf{p}$ we have $$\dim \mu_{\textbf{p},t}= \min\left\{\frac{\sum_{i=1}^4 p_i\log p_i}{-\log 2},1\right\}.$$
\end{thm}  The following theorem follows from the work of Shmerkin and Solomyak \cite[Theorem A]{ShmSol}. 

\begin{thm}
	\label{ShmSolcor}
For every $t\in[0,1]$ outside of a set of Hausdorff dimension $0$, we have that $\mu_{\textbf{p},t}$ is absolutely continuous whenever $$\frac{\sum_{i=1}^4 p_i\log p_i}{-\log 2}>1.$$
\end{thm}To apply Theorem A from \cite{ShmSol} we have to check that a non-degeneracy condition is satisfied. Checking this condition holds is straightforward in our setting so we omit the details.

It is known that the set of badly approximable numbers has Hausdorff dimension $1$ and Lebesgue measure zero. Therefore, applying Theorem \ref{Hoccor} and Theorem \ref{ShmSolcor}, it follows that there exists a badly approximable number $t,$ and some $t'$ that is not badly approximable, such that for any probability vector $\textbf{p}$ we have $$\dim \mu_{\textbf{p},t}=\dim \mu_{\textbf{p},t'}= \min\left\{\frac{\sum_{i=1}^4 p_i\log p_i}{-\log 2},1\right\},$$ and whenever $$\frac{\sum_{i=1}^4 p_i\log p_i}{-\log 2}>1$$ the measures  $\mu_{\textbf{p},t}$ and $\mu_{\textbf{p},t'}$ are both absolutely continuous. As such, the overlapping behaviour of $\Phi_t$ and $\Phi_{t'}$ are indistinguishable from the perspective of self-similar measures. However, we see from statement $3$ and statement $4$ of Theorem \ref{precise result} that there exists $h:\mathbb{N}\to [0,\infty)$ such that $U_{\Phi_t}(z,\m,h)$ has full Lebesgue measure for all $z\in[0,1+t]$, and $U_{\Phi_{t'}}(z,\m,h)$ has zero Lebesgue measure for all $z\in[0,1+t']$. Therefore, we see that by studying the metric properties of limsup sets we can distinguish between the overlapping behaviour of $\Phi_t$ and $\Phi_{t'}$. Studying the metric properties of limsup sets detects some of the finer details of how an iterated function system overlaps.

\subsection{The CS property and absolute continuity.}
We saw in the previous section that by studying IFSs using ideas from metric number theory, one can distinguish between IFSs in a way that is not available to us by simply studying pushforwards of Bernoulli measures. It is natural to wonder how Khintchine like behaviour relates to these measures. In this paper we show that there is a connection between a strong type of Khintchine like behaviour and the absolute continuity of these measures.

Given an IFS $\Phi$ and a slowly decaying measure $\m$, we say that $\Phi$ is consistently separated with respect to $\m$, or $\Phi$ has the CS property with respect to $\m$, if there exists $z\in X$ such that for any $h:\mathbb{N}\to [0,\infty)$ satisfying $$\sum_{n=1}^{\infty}h(n)=\infty,$$ the set $U_{\Phi}(z,\m,h)$ has positive Lebesgue measure. Using this terminology we see that statement $3$ and statement $4$ from Theorem \ref{precise result} imply that an IFS $\Phi_t$ has the CS property with respect to the $(1/4,1/4,1/4,1/4)$ Bernoulli measure if and only if $t$ is badly approximable. The use of the terminology consistently separated will become clearer in Section \ref{Colette section} (see Theorem \ref{new colette}). We prove the following result.

\begin{thm}
	\label{Colette thm}
For a slowly decaying $\sigma$-invariant ergodic probability measure $\m,$ if $\Phi$ has the CS property with respect to $\m,$ then the pushforward of $\m$ is absolutely continuous.
\end{thm}
We emphasise here that an IFS having the CS property with respect to $\m$ and the pushforward of $\m$ being absolutely continuous are not equivalent statements. There are many examples of $\m$ and $\Phi$ such that the pushforward of $\m$ is absolutely continuous, yet $\Phi$ does not have the CS property with respect to $\m.$ In particular, for the family of IFSs $\{\Phi_t\}$ studied in the previous section, it can be shown that the pushforward of the uniform $(1/4,1/4,1/4,1/4)$ Bernoulli measure is absolutely continuous for any $t\in[0,1]$. However as remarked above, $\Phi_t$ has the CS property with respect to this measure if and only if $t$ is badly approximable. We include several explicit examples of consistently separated iterated function systems in Section \ref{Examples}.

\subsection{Overlapping self-conformal sets}
\label{self-conformal section}
Theorems \ref{1d thm}, \ref{translation thm}, and \ref{precise result} are stated in terms of parameterised families of overlapping IFSs where one would expect that for a typical member of this family the corresponding attractor would have positive Lebesgue measure. In Theorem \ref{conformal theorem} the attractor can have arbitrary Hausdorff dimension, but we assume that the underlying IFS satisfies some separation hypothesis. None of these results cover the case when the IFS is overlapping and the attractor is not expected to have positive Lebesgue measure. The purpose of this section is to fill this gap for IFSs consisting of conformal mappings. We recall some background on this class of IFS below. 

Let $V\subset \mathbb{R}^d$ be an open set, a $C^1$ map $\phi:V\to\mathbb{R}^d$ is a conformal mapping if it preserves angles, or equivalently $\Phi$ is a conformal mapping if the differential $\phi'$ satisfies $|\phi'(x)y|=|\phi'(x)||y|$ for all $x\in V$ and $y\in\mathbb{R}^d$. We call an IFS $\Phi=\{\phi_i\}_{i=1}^l$ a conformal iterated function system on a compact set $Y\subset\mathbb{R}^d$ if each $\phi_i$ can be extended to an injective conformal contraction on some open connected neighbourhood $V$ that contains $Y,$ and $$0\leq \inf_{x\in V}|\phi_i'(x)|\leq \sup_{x\in V}|\phi_i'(x)|<1.$$ Throughout this paper we will assume that the differentials are H\"{o}lder continuous, i.e., there exists $\alpha>0$ and $c>0$ such that $$\big||\phi_i'(x)|-|\phi_i'(y)|\big|\leq c|x-y|^{\alpha}$$ for all $x,y\in V$. If our IFS is a conformal iterated function system on some compact set, then we call the corresponding attractor $X$ a self-conformal set. Self-conformal sets are a natural generalisation of self-similar sets. 

To any conformal IFS we associate the family of potentials $f_s:\D^{\mathbb{N}}\to \mathbb{R}$ given by $$f_s((a_j))=s\cdot\log |\phi_{a_1}'(\pi(\sigma(a_j)))|.$$ Where here $s\in(0,\infty)$. We define the topological pressure of $f_s$ to be $$P(f_s):=\sup\left\{\h(\m)+\int f_s d\m: \m \textrm{ is }\sigma\textrm{-invariant}\right\}.$$ For more on topological pressure and thermodynamic formalism we refer the reader to \cite{Bow} and \cite{Fal2}. It can be shown that for any conformal IFS, there exists a unique value of $s$ satisfying the equation $P(f_s)=0.$ We call this parameter the similarity dimension of $\Phi$ and denote it by $\dim_{S}(\Phi)$. When $\Phi$ is a conformal IFS and satisfies the open set condition, it is known that $\dim_{H}(X)=\dim_{B}(X)=\dim_{S}(\Phi)$. Importantly there exists a unique measure $\m_{\Phi}$ such that $$\h(\m_{\Phi})+\int f_{\dim_S(\Phi)}d\m_{\Phi}=0.$$
The pushforward of the measure $\m_{\Phi},$ which we denote by $\mu_{\Phi},$ is a particularly useful tool for determining metric properties of the attractor $X$. In particular, when $\Phi$ satisfies the open set condition it can be shown that $\mu_{\Phi}$ is equivalent to $\mathcal{H}^{\dim_{H}(X)}|_{X}$ (see \cite{PRSS}). Note that when $\Phi$ consists of similarities, i.e. $\Phi=\{\phi_i(x)=r_iO_i x +t_i\}_{i=1}^l$, then $\m_{\Phi}$ is simply the Bernoulli measure corresponding to the probability vector $(r_1^s,\ldots,r_l^s),$ where $s$ is the unique solution to the equation $\sum_{i=1}^lr_i^s=1.$ 

Our main result for conformal iterated function systems is the following theorem.

\begin{thm}
	\label{overlapping conformal theorem}
	If $\Phi$ is a conformal iterated function system, then for any $z\in X,$ if $\theta:\mathbb{N}\to [0,\infty)$ is a decreasing function and satisfies $$\sum_{n=1}^{\infty} \sum_{\a\in\D^{n}} (Diam(X_{\a})\theta(n))^{\dim_{S}(\Phi)}=\infty,$$ then $\mu_{\Phi}$-almost every $x\in X$ is an element of $W_{\Phi}(z,Diam(X_\a)\theta(|\a|)).$
\end{thm} As stated above, when $\Phi$ satisfies the open set condition then $\mu_{\Phi}$ is equivalent to $\mathcal{H}^{\dim_{H}(X)}|_{X},$ it follows therefore that Theorem \ref{overlapping conformal theorem} implies Theorem \ref{conformal theorem}. For our purposes the real value of Theorem \ref{overlapping conformal theorem} is demonstrated in the following corollary.

\begin{cor}
	\label{conformal cor}
Let $\Phi$ be a conformal iterated function system and suppose $\dim \mu_{\Phi}=\dim_{H}(X)$. Then for any $z\in X,$ if $\theta:\mathbb{N}\to [0,\infty)$ is a decreasing function and satisfies $$\sum_{n=1}^{\infty} \sum_{\a\in\D^{n}} (Diam(X_{\a})\theta(n))^{\dim_{S}(\Phi)}=\infty,$$ then $W_{\Phi}(z,Diam(X_\a)\theta(|\a|))$ has Hausdorff dimension equal to $\dim_{H}(X)$.
\end{cor}Corollary \ref{conformal cor} effectively reduces the problem of determining the Hausdorff dimension of $W_{\Phi}(z,Diam(X_\a)\theta(|\a|))$ to determining whether $\dim \mu_{\Phi}=\dim_{H}(X).$ Thankfully there are many results on the latter problem, and we can use these results together with Corollary \ref{conformal cor} to deduce further statements. We mention here only one such statement for the sake of brevity. The following statement follows by combining Theorem 1.1 from \cite{Hochman} and Corollary \ref{conformal cor}.

\begin{cor}
Assume $d=1$ and $\Phi$ consists solely of similarities. If $$\liminf_{n\to\infty}\frac{-\log \Delta_n}{n}<\infty,$$ where $$\Delta_n:=\min_{\a\neq \b\in \D^n}|\phi_{\a}(0)-\phi_{\b}(0)|,$$ then for any $z\in X,$ if $\theta:\mathbb{N}\to [0,\infty)$ is a decreasing function and satisfies $$\sum_{n=1}^{\infty} \sum_{\a\in\D^{n}} (Diam(X_{\a})\theta(n))^{\dim_{S}(\Phi)}=\infty,$$ then $W_{\Phi}(z,Diam(X_\a)\theta(|\a|))$ has Hausdorff dimension equal to $\dim_{H}(X)$.
\end{cor}

\subsection{Structure of the paper}
The rest of the paper is arranged as follows. In Section \ref{Preliminaries} we prove some general results that will allow us to prove our main theorems. In Section \ref{applications} we prove Theorems \ref{1d thm}, \ref{translation thm}, and \ref{random thm}. In Section \ref{Specific family} we prove Theorem \ref{precise result}. Section \ref{Colette section} is then concerned with the proof of Theorem \ref{Colette thm}, and in Section \ref{conformal section} we prove Theorem \ref{overlapping conformal theorem}. In Section \ref{misc} we apply the mass transference principle of Beresnevich and Velani to show how one can use our earlier results to deduce results on the Hausdorff measure and Hausdorff dimension of certain $W_{\Phi}(z,\Psi)$ when $$\sum_{\a\in \D^*}\Psi(\a)^{\dim_{H}(X)}<\infty.$$ In Section \ref{Examples} we include some explicit examples to accompany our main theorems. We conclude with some general discussion and pose some open questions in Section \ref{Final discussion}.

\section{Preliminary results}
\label{Preliminaries}

\subsection{A general framework}
In this section we prove some useful preliminaries that will allow us to prove the main results of this paper. Throughout this section $\Omega$ will denote a metric space equipped with some finite Borel measure $\eta$, and $\tilde{X}$ will denote some compact subset of $\mathbb{R}^d$. For each $n\in\mathbb{N}$ we will assume that there exists a finite set of continuous functions $\{f_{l,n}:\Omega\to \tilde{X}\}_{l=1}^{R_n}.$ For each $\omega\in\Omega$ we let $$Y_n(\omega):=\{f_{l,n}(\omega)\}_{l=1}^{R_n}.$$Before stating our general result we need to introduce some notation. Given $r>0$ we say that $Y\subset\mathbb{R}^d$ is an $r$-separated set if $|z-z'|>r,$ $\forall z,z'\in Y$ such that $z\neq z'.$ Given a finite set $Y\subset \mathbb{R}^d$ and $r>0,$ we let $$T(Y,r):=\sup\{\# Y':Y'\subseteq Y \textrm{ and }Y'\textrm{ is an }r\textrm{-separated set}\}.$$ We call $Y'\subseteq Y$ a maximal $r$-separated subset if $Y'$ is $r$-separated and $\# Y'=T(Y,r).$ Clearly a maximal $r$-separated subset always exists. Given a finite set $Y$ and $r>0,$ we will denote by $S(Y,r)$ an arbitrary choice of maximal $r$-separated subset. 

 The proposition stated below is the main technical result of this section.

\begin{prop}
	\label{general prop}
Suppose the following properties are satisfied:
\begin{itemize}
\item There exists $\gamma>1$ such that $$R_n\asymp \gamma^n.$$
\item There exists $G:(0,\infty)\to (0,\infty)$ such that $\lim_{s\to 0}G(s)=0,$ and for all $n\in\mathbb{N}$ we have $$\eta(\Omega)-\int_{\Omega} \frac{T(Y_n(\omega),\frac{s}{R_{n}^{1/d}})}{R_{n}} d\eta(\omega)\leq G(s).$$
\end{itemize} Then for $\eta$-almost every $\omega\in\Omega,$ for any $h\in H$ the set $$\left\{x\in\mathbb{R}^d:x\in\bigcup_{l=1}^{R_{n}}B\left(f_{l,n}(\omega),\left(\frac{h(n)}{R_{n}}\right)^{1/d}\right)\textrm{ for i.m. } n\in \mathbb{N}\right\}$$ has positive Lebesgue measure. 
\end{prop}Recall that the set of functions $H$ was defined in \eqref{H functions}. 

Given $c>0,s>0,$ and $n\in\mathbb{N},$ we let $$B(c,s,n):=\Big\{\omega\in\Omega:\frac{T(Y_n(\omega),\frac{s}{R_n^{1/d}})}{R_n}>c\Big\}.$$ The following lemma shows that under the hypothesis of Proposition \ref{general prop}, a typical $\omega\in\Omega$ is contained in $B(c,s,n)$ for a large set of $n$ for appropriate choices of $c$ and $s$. This lemma will play an important role in Section \ref{applications} when we recover some results of Solomyak on the absolute continuity of Bernoulli convolutions, and on the Lebesgue measure of the attractor in the $\{0,1,3\}$ problem.
\begin{lemma}
	\label{density separation lemma}
Assume there exists $G:(0,\infty)\to (0,\infty)$ such that $\lim_{s\to 0}G(s)=0,$ and for all $n\in\mathbb{N}$ we have $$\eta(\Omega)-\int_{\Omega} \frac{T(Y_n(\omega),\frac{s}{R_{n}^{1/d}})}{R_{n}} d\eta(\omega)\leq G(s).$$ Then 
$$\eta\left(\bigcap_{\epsilon>0}\bigcup_{c,s>0}\{\omega:\overline{d}(n:\omega\in B(c,s,n))\geq 1-\epsilon\}\right)=\eta(\Omega).$$


\end{lemma}
\begin{proof}
Observe that $$0\leq \frac{T(Y_n(\omega),\frac{s}{R_{n}^{1/d}})}{R_n}\leq 1$$ for all $\omega\in\Omega$ and $n\in\mathbb{N}$. As a result of this inequality and our underling assumption, for any $c>0$, $s>0$, and $n\in N$, we have 
$$\eta(B(c,s,n))+c\cdot \eta(B(c,s,n)^c)\geq \int_{\Omega} \frac{T(Y_n(\omega),\frac{s}{R(n)^{1/d}})}{R(n)} d\eta(\omega)\geq \eta(\Omega)-G(s).$$ This in turn implies $$\eta(B(c,s,n))\geq (1-c)\eta(\Omega)-G(s).$$ It follows that given $\epsilon>0,$ we can pick $c>0$ and $s>0$ independent of $n$ such that 
\begin{equation}
\label{almost full}
\eta(B(c,s,n))\geq \eta(\Omega)-\epsilon.
\end{equation} Applying Fatou's lemma we have 
\begin{align*}
\int_{\Omega} \overline{d}(n:\omega\in B(c,s,n))d\eta &=\int_{\Omega}\limsup_{N\to\infty}\frac{\#\{1\leq n\leq N:\omega\in B(c,s,n)\}}{N}d\eta\\
&=\int_{\Omega}\limsup_{N\to\infty}\frac{\sum_{n=1}^{N}\chi_{B(c,s,n)}(\omega)}{N}d\eta\\
&\stackrel{Fatou}{\geq} \limsup_{N\to\infty}\frac{\sum_{n=1}^{N}\int_{\Omega}\chi_{B(c,s,n)}(\omega)d\eta}{N}\\
&=\limsup_{N\to\infty} \frac{\sum_{n=1}^{N}\eta(B(c,s,n))}{N}\\
&\stackrel{\eqref{almost full}}{\geq} \eta(\Omega)-\epsilon.
\end{align*}Summarising the above, we have shown that for this choice of $c$ and $s$ we have
\begin{equation}
\label{Banach bound}\int_{\Omega} \overline{d}(n:\omega\in B(c,s,n))d\eta \geq \eta(\Omega)-\epsilon.
\end{equation}
For the purpose of obtaining a contradiction, suppose 
\begin{equation}
\label{contradict1}
\eta\Big(\omega:\overline{d}(n:\omega\in B(c,s,n))\leq  1-\sqrt{\epsilon}\Big)> \sqrt{\epsilon}.
\end{equation}
Then using the fact $$0\leq \overline{d}(n:\omega\in B(c,s,n))\leq 1$$ for all $\omega\in \Omega$, we have
\begin{align*}
\int_{\Omega} \overline{d}(n:\omega\in B(c,s,n))d\eta& \leq \eta(\omega:\overline{d}(n:\omega\in B(c,s,n))\leq  1-\sqrt{\epsilon})(1-\sqrt{\epsilon})\\
&+ \eta(\omega:\overline{d}(n:\omega\in B(c,s,n))> 1-\sqrt{\epsilon})\\
&=\eta(\omega:\overline{d}(n:\omega\in B(c,s,n))\leq  1-\sqrt{\epsilon})(1-\sqrt{\epsilon})\\
&+ \eta(\Omega)-\eta(\omega:\overline{d}(n:\omega\in B(c,s,n))\leq 1-\sqrt{\epsilon})\\
&=\eta(\Omega)-\sqrt{\epsilon}\eta(\omega:\overline{d}(n:\omega\in B(c,s,n))\leq  1-\sqrt{\epsilon})\\
&\stackrel{\eqref{contradict1}}{<}\eta(\Omega)-\epsilon.
\end{align*}
 This contradicts \eqref{Banach bound}. Therefore \eqref{contradict1} is not possible and we have that for any $\epsilon>0,$ there exists $c,s>0$ such that
\begin{equation}
\label{shown to hold}
\eta\Big(\omega:\overline{d}(n:\omega\in B(c,s,n))>  1-\sqrt{\epsilon}\Big)\geq \eta(\Omega)-\sqrt{\epsilon}.
\end{equation}
Equation \eqref{shown to hold} in turn implies that for any $\epsilon>0$ we have
\begin{equation}
\label{full measure sep} \eta\Big(\bigcup_{c,s>0}\{\omega:\overline{d}(n:\omega\in B(c,s,n))\geq 1-\epsilon\}\Big)=\eta(\Omega).
\end{equation} One can see how \eqref{full measure sep} follows from \eqref{shown to hold} by first fixing $\epsilon>0$ and then applying \eqref{shown to hold} for a countable collection of $\epsilon_k$ strictly smaller than $\epsilon$. Now intersecting over all $\epsilon>0,$ we see that \eqref{full measure sep} implies the desired equality:
$$\eta\Big(\bigcap_{\epsilon>0}\bigcup_{c,s>0}\{\omega:\overline{d}(n:\omega\in B(c,s,n))\geq 1-\epsilon\}\Big)=\eta(\Omega).$$
\end{proof}

To prove Proposition \ref{general prop} and many other results in this paper, we will rely upon the following useful lemma known as the generalised Borel-Cantelli Lemma.

\begin{lemma}
	\label{Erdos lemma}
	Let $(X,A,\mu)$ be a finite measure space and $E_n\in A$ be a sequence of 
	sets such that $\sum_{n=1}^{\infty}\mu(E_n)=\infty.$ Then
	$$\mu(\limsup_{n\to\infty} E_{n})\geq \limsup_{Q\to\infty}\frac{(\sum_{n=1}^{Q}\mu(E_{n}))^{2}}{\sum_{n,m=1}^{Q}\mu(E_{n}\cap E_m)}.$$
\end{lemma}Lemma \ref{Erdos lemma} is due to Kochen and Stone \cite{KocSto}. For a proof of this lemma see either \cite[Lemma 2.3]{Har} or \cite[Lemma 5]{Spr}. 

Proposition \ref{general prop} will follow from the following proposition. This result will also be useful when it comes to proving some of our later results.

\begin{prop}
	\label{fixed omega}
Let $\omega\in \Omega$ and $h:\mathbb{N}\to[0,\infty).$ Assume the following properties are satisfied:
\begin{itemize}
	\item There exists $\gamma>1$ such that $$R_n\asymp \gamma^n.$$
	\item There exists $c>0$ and $s>0$ such that $$\sum_{n:\omega\in B(c,s,n)}h(n)=\infty.$$
\end{itemize} Then $$\left\{x\in\mathbb{R}^d:x\in\bigcup_{l=1}^{R_{n}}B\left(f_{l,n}(\omega),\left(\frac{h(n)}{R_{n}}\right)^{1/d}\right)\textrm{ for i.m. } n\in \mathbb{N}\right\}$$ has positive Lebesgue measure.
\end{prop}

\begin{proof}[Proof of Proposition \ref{fixed omega}]
We split our proof into individual steps for convenience. \\

\noindent \textbf{Step 1. Replacing our approximating function.}\\
Let $\omega$ and $h$ be fixed, and $c$ and $s>0$ be as in the statement of the proposition. We claim that 
\begin{equation}
\label{divergence1}\sum_{n:\omega\in B(c,s,n)} \sum_{u\in S(Y_n(\omega),\frac{s}{R_{n}^{1/d}})} \mathcal{L}\Big(B\Big(u,\Big(\frac{h(n)}{R_n}\Big)^{1/d}\Big)\Big)=\infty.
\end{equation} This follows from our assumption $$\sum_{n:\omega\in B(c,s,n)}h(n)=\infty,$$ and the following:
\begin{align*}
\sum_{n:\omega\in B(c,s,n)} \sum_{u\in S(Y_n(\omega),\frac{s}{R_{n}^{1/d}})} \mathcal{L}\Big(B\Big(u,\Big(\frac{h(n)}{R_n}\Big)^{1/d}\Big)\Big)&= \sum_{n:\omega\in B(c,s,n)} \sum_{u\in S(Y_n(\omega),\frac{s}{R_{n}^{1/d}})} \frac{h(n)\L(B(0,1))}{R_n}\\
&=\sum_{n:\omega\in B(c,s,n)} T\Big(Y_n(\omega),\frac{s}{R_{n}^{1/d}}\Big)\frac{h(n)\L(B(0,1))}{R_n}\\
&\geq c \L(B(0,1))\sum_{n:\omega\in B(c,s,n)}h(n)\\
&=\infty.
\end{align*} Let us now define $$g(n):=\min \Big\{\Big(\frac{h(n)}{R_n}\Big)^{1/d},\frac{s}{3R_{n}^{1/d}}\Big\}.$$ We claim that we still have 
\begin{equation}
\label{divergence2}\sum_{n:\omega\in B(c,s,n)} \sum_{u\in S(Y_n(\omega),\frac{s}{R_{n}^{1/d}})} \mathcal{L}(B(u,g(n)))=\infty.
\end{equation} If $g(n)=\frac{s}{3R_{n}^{1/d}}$ for finitely many $n\in\N,$ then \eqref{divergence1} would imply \eqref{divergence2}. Suppose therefore that $g(n)=\frac{s}{3R_{n}^{1/d}}$ for infinitely many $n\in\mathbb{N}$. For such an $n$ we would have $$\sum_{u\in S(Y_n(\omega),\frac{s}{R_{n}^{1/d}})} \mathcal{L}(B(u,g(n)))=T\Big(Y_n(\omega),\frac{s}{R_{n}^{1/d}}\Big) \frac{s^d\L(B(0,1))}{3^dR_{n}}> \frac{cs^d\L(B(0,1))}{3^d}.$$ This lower bound is strictly positive and independent of $n$. As such summing over it for infinitely many $n$ guarantees divergence. Therefore \eqref{divergence2} holds.

Note that it follows from the definition of $g$ that we will have proved our result if we can show that
\begin{equation}
\label{WANT1}
\L\left(\left\{x\in\mathbb{R}^d:x\in\bigcup_{l=1}^{R_{n}}B\left(f_{l,n}(\omega),g(n)\right)\textrm{ for i.m. } n\in \mathbb{N}\right\}\right)>0.
\end{equation}\\
	
\noindent \textbf{Step 2: Constructing our $E_n$.}\\
Since $g(n)\leq \frac{s}{3R_{n}^{1/d}}$ for all $n\in\mathbb{N}$, it follows that for any distinct $u,v\in S(Y_n(\omega),\frac{s}{R_{n}^{1/d}}),$ we must have 
\begin{equation}
\label{disjoint balls}B(u,g(n))\cap B(v,g(n))=\emptyset.
\end{equation} For each $n$ such that $\omega\in B(c,s,n),$ let $$E_{n}=\bigcup_{u\in S(Y_n(\omega),\frac{s}{R(n)^{1/d}})} B(u,g(n)).$$ We will show that 
\begin{equation}
\label{WANT2}
\L\Big(\Big\{x\in \mathbb{R}^d: x\in E_n \textrm{ for i.m. }n\in\mathbb{N} \textrm{ such that } \omega\in B(c,s,n)\Big\}\Big)>0.
\end{equation} Equation \eqref{WANT2} implies \eqref{WANT1}, so to complete our proof it suffices to show that \eqref{WANT2} holds. It follows from \eqref{divergence2} and \eqref{disjoint balls} that
\begin{equation}
\label{divergencez}
\sum_{n:\omega\in B(c,s,n)}\mathcal{L}(E_n)=\sum_{n:\omega\in B(c,s,n)} \sum_{u\in S(Y_{n}(\omega),\frac{s}{R_{n}^{1/d}})} \mathcal{L}(B(u,g(n))=\infty.
\end{equation}
 Equation \eqref{divergencez} shows that our collection of sets $\{E_n\}_{n:\omega\in B(c,s,n)}$ satisfies the hypothesis of Lemma \ref{Erdos lemma}. 
 
 We record here for later use the following fact: for each $n\in\mathbb{N}$ such that $\omega\in B(c,s,n),$ we have  
\begin{equation}
\label{E_n measure}
\L(E_n)\asymp R_ng(n)^d.
\end{equation}Equation \eqref{E_n measure} follows from \eqref{disjoint balls} and the fact that for each $n\in \N$ such that $\omega\in B(c,s,n),$ we have $$c R_n\leq T\Big(Y_n(\omega),\frac{s}{R_{n}^{1/d}}\Big)\leq R_n .$$\\

\noindent \textbf{Step 3: Bounding $\mathcal{L}(E_n\cap E_m)$.}\\
Assume $n$ is such that $\omega\in B(c,s,n),$ $m$ is such that $\omega\in B(c,s,m),$ and $m\neq n.$ Fix $u\in S\big(Y_n(\omega),\frac{s}{R_{n}^{1/d}}\big).$ We want to bound the quantity: $$\#\Big\{v\in S\Big(Y_m(\omega),\frac{s}{R_{m}^{1/d}}\Big):B(u,g(n))\cap B(v,g(m)) \neq \emptyset \Big\}.$$ If $g(m)\geq g(n),$ then every $v\in S\big(Y_m(\omega),\frac{s}{R_{m}^{1/d}}\big)$ satisfying  $B(u,g(n))\cap B(v,g(m)) \neq \emptyset$ must also satisfy $B(v,g(m))\subset B(u,3g(m)).$ It follows therefore from \eqref{disjoint balls} and a volume argument that  
\begin{equation}
\label{estimate1}\#\Big\{v\in S\Big(Y_m(\omega),\frac{s}{R_{m}^{1/d}}\Big):B(u,g(n))\cap B(v,g(m)) \neq \emptyset \Big\}=\mathcal{O}(1).
\end{equation}If $g(m)<g(n),$ then every $v\in S\big(Y_m(\omega),\frac{s}{R_{m}^{1/d}}\big)$ satisfying  $B(u,g(n))\cap B(v,g(m)) \neq \emptyset$ must also satisfy $B(v,g(m))\subseteq B(u,3g(n)).$ Since the elements of $S\big(Y_m(\omega),\frac{s}{R_{m}^{1/d}}\big)$ are by definition separated by a factor $\frac{s}{R_{m}^{1/d}},$ it follows from a volume argument that 
\begin{equation}
\label{estimate2}\#\Big\{v\in S\Big(Y_m(\omega),\frac{s}{R_{m}^{1/d}}\Big):B(u,g(n))\cap B(v,g(m)) \neq \emptyset \Big\}= \mathcal{O}\left(\frac{g(n)^dR_{m}}{s^d}+1\right).
\end{equation} Combining \eqref{estimate1} and \eqref{estimate2}, we see that for any $n\in\N$ such that $\omega\in B(c,s,n),$ and $m\neq n$ such that $\omega\in B(c,s,m),$ we have
\begin{equation}
\label{count bound}
 \#\Big\{v\in S\Big(Y_m(\omega),\frac{s}{R_{m}^{1/d}}\Big):B(u,g(n))\cap B(v,g(m)) \neq \emptyset \Big\}= \mathcal{O}\left(\frac{g(n)^dR_{m}}{s^d}+1\right).
\end{equation}
We now use \eqref{count bound} to bound $\L(E_n\cap E_m):$
\begin{align*}
\L(E_n\cap E_m)&\stackrel{\eqref{disjoint balls}}{=}\sum_{u\in S(Y_n(\omega),\frac{s}{R_{n}^{1/d}})} \L(B(u,g(n))\cap E_m)\\
&\leq \sum_{u\in S(Y_n(\omega),\frac{s}{R_{n}^{1/d}})} \L(0,g(m))\#\Big\{v\in S\Big(Y_m(\omega),\frac{s}{R_{m}^{1/d}}\Big):B(u,g(n))\cap B(v,g(m)) \neq \emptyset \Big\}\\
&\stackrel{\eqref{count bound}}{=} \sum_{u\in S(Y_n(\omega),\frac{s}{R_{n}^{1/d}})} \mathcal{O}\Big(g(m)^{d}\Big(\frac{g(n)^dR_{m}}{s^d}+1\Big)\Big)\\
&=\mathcal{O}\Big(R_{n}g(m)^{d}\Big(\frac{g(n)^dR_{m}}{s^d}+1\Big)\Big).
\end{align*} Summarising the above, we have shown that for any $n\in\mathbb{N}$ such that $\omega\in B(c,s,n),$ and $m\neq n$ such that $\omega\in B(c,s,m),$ we have:
\begin{equation}
\label{intersection bound}
\L(E_n\cap E_m)=\mathcal{O}\Big(R_{n}g(m)^{d}\Big(\frac{g(n)^dR_{m}}{s^d}+1\Big)\Big).
\end{equation}

\noindent \textbf{Step 4. Applying Lemma \ref{Erdos lemma}.}\\
By Lemma \ref{Erdos lemma}, to prove that \eqref{WANT2} holds, and to finish our proof, it suffices to show that 
\begin{equation}
\label{Erdos bound}
\sum_{\stackrel{n,m=1}{n:\omega\in B(c,s,n),\,m:\omega\in B(c,s,m)}}^{Q}\L(E_n\cap E_m)=\mathcal{O}\Big(\big(\sum_{\stackrel{n=1}{n:\omega\in B(c,s,n)}}^Q\L(E_n)\big)^2\Big).
\end{equation}This we do below. We start by separating terms:
\begin{equation}
\label{separating terms}
\sum_{\stackrel{n,m=1}{n:\omega\in B(c,s,n),\,m:\omega\in B(c,s,m)}}^{Q}\L(E_n\cap E_m)=\sum_{\stackrel{n=1}{n:\omega\in B(c,s,n)}}^{Q}\L(E_n)+2\sum_{\stackrel{m=2}{m:\omega\in B(c,s,m)}}^{Q}\sum_{\stackrel{n=1}{n:\omega\in B(c,s,n)}}^{m-1}\L(E_n\cap E_m).
\end{equation}Focusing on the first term on the right hand side of \eqref{separating terms}, we know that $$\sum_{\stackrel{n=1}{n:\omega\in B(c,s,n)}}^{\infty}\L(E_n)=\infty$$ by \eqref{divergencez}. Therefore, for all $Q$ sufficiently large we have $$\sum_{\stackrel{n=1}{n:\omega\in B(c,s,n)}}^{Q}\L(E_n)\geq 1.$$ This implies that 
\begin{equation}
\label{square bound}
\sum_{\stackrel{n=1}{n:\omega\in B(c,s,n)}}^{Q}\L(E_n)=\mathcal{O}\Big( \big(\sum_{\stackrel{n=1}{n:\omega\in B(c,s,n)}}^{Q}\L(E_n)\big)^2\Big).
\end{equation}Focusing on the second term in \eqref{separating terms}, we have 
\begin{align*}
\sum_{\stackrel{m=2}{m:\omega\in B(c,s,m)}}^{Q}\sum_{\stackrel{n=1}{n:\omega\in B(c,s,n)}}^{m-1}\L(E_n\cap E_m)&\stackrel{\eqref{intersection bound}}{=}\mathcal{O}\Big(\sum_{\stackrel{m=2}{m:\omega\in B(c,s,m)}}^{Q}\sum_{\stackrel{n=1}{n:\omega\in B(c,s,n)}}^{m-1}R_ng(m)^{d}\Big(\frac{g(n)^dR_m}{s^d}+1\Big)\Big)\\
&=\mathcal{O}\Big(\sum_{\stackrel{m=2}{m:\omega\in B(c,s,m)}}^{Q}\sum_{\stackrel{n=1}{n:\omega\in B(c,s,n)}}^{m-1}R_ng(n)^d R_mg(m)^d\Big)\\
&+\mathcal{O}\Big(\sum_{\stackrel{m=2}{m:\omega\in B(c,s,m)}}^{Q}\sum_{\stackrel{n=1}{n:\omega\in B(c,s,n)}}^{m-1}R_ng(m)^{d}\Big)
\end{align*}
Focusing on the first term in the above, we see that \begin{align*}
\sum_{\stackrel{m=2}{m:\omega\in B(c,s,m)}}^{Q}\sum_{\stackrel{n=1}{n:\omega\in B(c,s,n)}}^{m-1}R_ng(n)^d R_mg(m)^d&=\sum_{\stackrel{m=2}{m:\omega\in B(c,s,m)}}^{Q}R_mg(m)^d\sum_{\stackrel{n=1}{n:\omega\in B(c,s,n)}}^{m-1}R_ng(n)^d\\
&\leq \Big(\sum_{\stackrel{n=1}{n:\omega\in B(c,s,n)}}^{Q}R_ng(n)^d\Big)^2\\
&\stackrel{\eqref{E_n measure}}{=}\mathcal{O}\Big(\big(\sum_{\stackrel{n=1}{n:\omega\in B(c,s,n)}}^Q\L(E_n)\big)^2\Big).
\end{align*}
Focusing on the second term, we have 
\begin{align*}
\sum_{\stackrel{m=2}{m:\omega\in B(c,s,m)}}^{Q}\sum_{\stackrel{n=1}{n:\omega\in B(c,s,n)}}^{m-1}R_ng(m)^{d}&=\mathcal{O}\Big( \sum_{\stackrel{m=2}{m:\omega\in B(c,s,m)}}^Q R_mg(m)^d\Big)\\
&\stackrel{\eqref{E_n measure}}{=}\mathcal{O}\Big(\sum_{\stackrel{m=1}{m:\omega\in B(c,s,m)}}^{Q} \L(E_m)\Big)\\
&\stackrel{\eqref{square bound}}{=}\mathcal{O}\Big( \big(\sum_{\stackrel{m=1}{m:\omega\in B(c,s,m)}}^{Q} \L(E_m)\big)^2\Big).
\end{align*}In the first equality above we used the assumption that $R_n\asymp \gamma^n,$ and therefore by properties of geometric series $$\sum_{\stackrel{n=1}{n:\omega\in B(c,s,n)}}^{m-1}R_n\leq \sum_{n=1}^{m-1}R_n=\mathcal{O}(R_m).$$ This is the only point in the proof where we use the assumption $R_n\asymp \gamma^n$. 

Collecting the bounds obtained above, we see that 
\begin{equation}
\label{square bound2}
\sum_{\stackrel{m=2}{m:\omega\in(B(c,s,m)}}^{Q}\sum_{\stackrel{n=1}{n:\omega\in B(c,s,n)}}^{m-1}\L(E_n\cap E_m)=\mathcal{O}\Big(\big (\sum_{\stackrel{n=1}{n:\omega\in B(c,s,n)}}^{Q}\L(E_n)\big)^2\Big).
\end{equation}
Substituting the bounds  \eqref{square bound} and \eqref{square bound2} into \eqref{separating terms}, we see that \eqref{Erdos bound} holds as required. This completes our proof.

\end{proof}
	
With Proposition \ref{fixed omega} we can now prove Proposition \ref{general prop}.	
\begin{proof}[Proof of Proposition \ref{general prop}]

Let $$P:=\bigcap_{\epsilon>0}\bigcup_{c,s>0}\{\omega:\overline{d}(n:\omega\in B(c,s,n))\geq 1-\epsilon\}.$$ Fix $\omega\in P$. For any $h\in H$, by definition there exists $\epsilon>0$ such that $h\in H_{\epsilon}$. It follows from the definition of $P,$ that there exists $c,s>0$ such that $$\overline{d}(n:\omega\in B(c,s,n))> 1-\epsilon.$$ In which case, by the definition of $H_{\epsilon},$ we must have $$\sum_{n:\omega\in B(c,s,n)}h(n)=\infty.$$ Applying Proposition \ref{fixed omega} it follows that 
$$\left\{x\in\mathbb{R}^d:x\in\bigcup_{l=1}^{R_{n}}B\left(f_{l,n}(\omega),\left(\frac{h(n)}{R_{n}}\right)^{1/d}\right)\textrm{ for i.m. } n\in \mathbb{N}\right\}$$
 has positive Lebesgue measure. Our result now follows since $\omega\in P$ was arbitrary and we know by Lemma \ref{density separation lemma} that $\eta(P)=\eta(\Omega).$
\end{proof}

\subsubsection{Verifying the hypothesis of Proposition \ref{general prop}.}
To prove Theorems \ref{1d thm}, \ref{translation thm} and \ref{random thm}, we will apply Proposition \ref{general prop}. Naturally to do so we need to verify the hypothesis of Proposition \ref{general prop}. The exponential growth condition on the number of elements in our set will be automatically satisfied. Verifying the second integral condition is more involved. We will show that this integral condition holds via a transversality argument. Unfortunately the quantity $T(Y_{n}(\omega),\frac{s}{R_{n}^{1/d}})$ is not immediately amenable to transversality techniques. Instead we study the quantity:
$$R(\omega,s,n):=\left\{(l,l')\in\{1,\ldots,R_n\}^2 :|f_{l,n}(\omega)-f_{l',n}(\omega)|\leq\frac{s}{R_{n}^{1/d}}\,\textrm{ and } l\neq l'\right\}.$$ The following lemma allows us to deduce the integral bound appearing in Proposition \ref{general prop} from a similar bound for $R(\omega,s,n)$.
\begin{lemma}
	\label{integral bound}
	Assume there exists $G:(0,\infty)\to(0,\infty)$ satisfying $\lim_{s\to 0}G(s)=0,$ such that for all $n\in N$ we have $$\int_{\Omega} \frac{\#R(\omega,s,n)}{R_n}d\eta\leq G(s).$$ Then for all $n\in N$ we have 
	$$\eta(\Omega)-\int_{\Omega}\frac{T(Y_{n}(\omega),\frac{s}{R_{n}^{1/d}})}{R_n}d\eta \leq  G(s).$$
	
\end{lemma}

\begin{proof}Let
	$$W(\omega,s,n):=\left\{l\in\{1,\ldots,R_n\}: |f_{l,n}(\omega)-f_{l',n}(\omega)|>\frac{s}{R_{n}^{1/d}}\, \forall l'\neq l\right\}.$$ Since for any distinct $l,l'\in W(\omega,s,n),$ the distance between $f_{l,n}(\omega)$ and $f_{l',n}(\omega)$ is at least $\frac{s}{R_{n}^{1/d}}$, it follows that $W(\omega,s,n)$ is a $\frac{s}{R_{n}^{1/d}}$-separated set. Therefore
	\begin{equation}
	\label{countbounda}
	\# W(\omega,s,n)\leq T\Big(Y_n(\omega),\frac{s}{R_{n}^{1/d}}\Big).
	\end{equation} Importantly we also have 
	\begin{equation}
	\label{countboundb}\#W(\omega,s,n)^c=\#\left\{l\in\{1,\ldots,R_n\}:|f_{l,n}(\omega)-f_{l',n}(\omega)|\leq \frac{s}{R_{n}^{1/d}}\,\textrm{ for some }l'\neq l\right\}\leq \#R(\omega,s,n).
	\end{equation} This follows because the map $f:R(\omega,s,n)\to W(\omega,s,n)^c$ defined by $f(l,l')=l$ is a surjective map. 
	
	Now suppose we have $G:(0,\infty)\to(0,\infty)$ satisfying the hypothesis of our proposition. Then for any $s>0$ and $n\in \N$ we have 
	\begin{align*}
	\eta(\Omega)=\int_{\Omega} \frac{\# W(\omega,s,n)+\# W(\omega,s,n)^c}{R_n}d\eta &\stackrel{\eqref{countbounda}\, \eqref{countboundb}}\leq \int_{\Omega} \frac{T(Y_{n}(\omega),\frac{s}{R_{n}^{1/d}})}{R_n}d\eta+\int_{\Omega}\frac{\#R(\omega,s,n)}{R_n}d\eta\\
	&\leq \int_{\Omega} \frac{T(Y_{n}(\omega),\frac{s}{R_{n}^{1/d}})}{R_n}d\eta+G(s).
	\end{align*}This implies 
	$$\eta(\Omega)-\int_{\Omega}\frac{T(Y_{n}(\omega),\frac{s}{R_{n}^{1/d}})}{R_n}d\eta \leq  G(s).$$
	
\end{proof}

\subsubsection{The non-existence of a Khintchine like result}
The purpose of this section is to prove the following proposition. It will be used in the proof of Theorem \ref{precise result} and Theorem \ref{Colette thm}. It demonstrates that a lack of separation along a subsequence can lead to the non-existence of a Khintchine like result.

\begin{prop}
	\label{fail prop}
Let $\omega\in\Omega$ and suppose that for some $s>0$ we have $$\liminf_{n\to\infty} \frac{T(Y_{n}(\omega),\frac{s}{R_{n}^{1/d}})}{R_n}=0.$$ Then there exists $h:\mathbb{N}\to[0,\infty)$ such that $$\sum_{n=1}^{\infty}h(n)=\infty,$$ yet $$\left\{x\in\mathbb{R}^d:x\in\bigcup_{l=1}^{R_{n}}B\left(f_{l,n}(\omega),\left(\frac{h(n)}{R_{n}}\right)^{1/d}\right)\textrm{ for i.m. } n\in \mathbb{N}\right\}$$ has zero Lebesgue measure. 
\end{prop}
\begin{proof}
Let $\omega\in \Omega$ and $s>0$ be as above. By our assumption, there exists a strictly increasing sequence $(n_j)$ such that 
\begin{equation}
\label{separated upper bound}
T\Big(Y_{n_j}(\omega),\frac{s}{R_{n_j}^{1/d}}\Big)\leq \frac{R_{n_j}}{j^2}
\end{equation} for all $j\in\mathbb{N}$. By the definition of a maximal $s\cdot R_{n_j}^{-1/d}$-separated set, we know that for each $l\in\{1,\ldots,R_{n_j}\},$ there exists $u\in S(Y_{n_j}(\omega),\frac{s}{R_{n_j}^{1/d}})$ such that $|u-f_{l,n_j}(\omega)|\leq s\cdot R_{n_j}^{-1/d}.$ It follows that 
\begin{equation}
\label{inclusion z}
\bigcup_{l=1}^{R_{n_{j}}}B\left(f_{l,n_j}(\omega),\frac{s}{R_{n_j}^{1/d}}\right)\subseteq \bigcup_{u\in S(Y_{n_j}(\omega),\frac{s}{R_{n_j}^{1/d}})}B\left(u,\frac{3s}{R_{n_j}^{1/d}}\right).
\end{equation}We now define our function $h:\mathbb{N}\to [0,\infty):$
$$h(n)=
\left\{
\begin{array}{ll}
s^d  & \mbox{if } n=n_j \textrm{ for some }j\in\mathbb{N}\\
0 &  \textrm{otherwise}
\end{array}
\right.
$$This function obviously satisfies $$\sum_{n=1}^{\infty}h(n)=\infty.$$ By \eqref{inclusion z} and the definition of $h,$ we see that
\begin{align}
&\L\left(\left\{x\in\mathbb{R}^d:x\in\bigcup_{l=1}^{R_{n}}B\left(f_{l,n}(\omega),\left(\frac{h(n)}{R_{n}}\right)^{1/d}\right)\textrm{ for i.m. } n\in \mathbb{N}\right\}\right)\nonumber \\
\leq &\L\left(\left\{x:x\in \bigcup_{u\in S(Y_{n_j}(\omega),\frac{s}{R_{n_j}^{1/d}})}B\left(u,\frac{3s}{R_{n_j}^{1/d}}\right)\textrm{ for i.m. }j\right\}\right)\label{inclusion zz}.
\end{align} So to prove our result it suffices to show that the right hand side of \eqref{inclusion zz} is zero. This fact now follows from the Borel-Cantelli lemma and the following inequalities:
\begin{align*}\sum_{j=1}^{\infty}\sum_{u\in S(Y_{n_j}(\omega),\frac{s}{R_{n_j}^{1/d}})}\L\Big(B\Big(u,\frac{3s}{R_{n_j}^{1/d}}\Big)\Big)&=\sum_{j=1}^{\infty}T\Big(Y_{n_j}(\omega),\frac{s}{R_{n_j}^{1/d}}\Big)\frac{(3s)^d\L(B(0,1))}{R_{n_j}}\\
&\stackrel{\eqref{separated upper bound}}{\leq}\sum_{j=1}^{\infty}\frac{(3s)^d\L(B(0,1))}{j^2}\\
&<\infty.
\end{align*}

\end{proof}

\subsection{Full measure statements}
The main result of the previous section was Proposition \ref{general prop}. This result provides sufficient conditions for us to conclude that for a parameterised family of points, almost surely each member of a class of limsup sets defined in terms of neighbourhoods of these points will have positive Lebesgue measure. We will eventually apply Proposition \ref{general prop} to the sets $U_{\Phi}(z,\m,h)$ defined in the introduction. Instead of just proving positive measure statements, we would like to be able to prove full measure results. The purpose of this section is to show how one can achieve this goal. Proposition \ref{full measure} achieves this by imposing some extra assumptions on the function $\Psi$. Proposition \ref{separated full measure} achieves this by imposing some stronger separation hypothesis.

The following lemma follows from Lemma 1 of \cite{BerVel2}. It is a consequence of the Lebesgue density theorem. 

\begin{lemma}[\cite{BerVel2}]
	\label{arbitrarily small}The following statements are true:
\begin{enumerate}
	\item Let $(x_j)$ be a sequence of points in $\mathbb{R}^d$ and $(r_j),(r_j')$ be two sequences of positive real numbers both converging to zero. If $r_j\asymp r_j'$ then $$\L(x:x\in B(x_j,r_j) \textrm{ for i.m. } j)=\L(x:x\in B(x_j,r_j') \textrm{ for i.m. } j).$$
	\item Let $B(x_j,r_j)$ be a sequence of balls in $\mathbb{R}^d$ such that $r_j\to 0$. Then $$\L(x:x\in B(x_j,r_j) \textrm{ for i.m. } j)=\L\left(\bigcap _{0<c<1}\{x:x\in B(x_j,cr_j) \textrm{ for i.m. } j\}\right).$$
\end{enumerate}
\end{lemma} 
Lemma \ref{arbitrarily small} implies the following useful fact. If $\Psi:\D^*\to[0,\infty)$ is equivalent to $(\m,h)$ and $U_{\Phi}(z,\m,h)$ has positive Lebesgue measure, then $W_{\Phi}(z,\Psi)$ has positive Lebesgue measure. We will use this fact several times throughout this paper.

Lemma \ref{arbitrarily small} will be used in the proof of the following proposition and in the proofs of our main theorems. Recall that we say that a function $\Psi:\D^*\to[0,\infty)$ is weakly decaying if 
$$\inf_{\a\in \D^*}\min_{i\in \D}\frac{\Psi(i\a)}{\Psi(\a)}>0.$$

\begin{prop}
	\label{full measure}
The following statements are true:
\begin{enumerate}
		\item Assume $\Phi$ is a collection of similarities with attractor $X$. If $z\in X$ is such that $\L(W_{\Phi}(z,\Psi))>0$ for some $\Psi$ that is weakly decaying, then Lebesgue almost every $x\in X$ is contained in $W_{\Phi}(z,\Psi)$. 
	\item Assume $\Phi$ is an arbitrary IFS and there exists $\mu,$ the pushforward of a $\sigma$-invariant ergodic probability measure $\m,$ satisfying $\mu\sim \L|_{X}$. Then if $z\in X$ is such that $\L(W_{\Phi}(z,\Psi))>0$ for some $\Psi$ that is weakly decaying, then Lebesgue almost every $x\in X$ is contained in $W_{\Phi}(z,\Psi)$. 
	\end{enumerate}
\end{prop}
\begin{proof}
We prove each statement separately. \\

\noindent \textbf{Proof of statement 1.}\\
Let $\Phi$ be an IFS consisting of similarities and suppose $z$ and $\Psi$ satisfy the hypothesis of the proposition. Let $$A:=\bigcap_{0<c<1}W_{\Phi}(z,c\Psi).$$ It follows from Lemma \ref{arbitrarily small} that $$\L(A)=\L(W_{\Phi}(z,\Psi))>0.$$  We claim that 
\begin{equation}
\label{full measurea}\L\Big(\bigcup_{\a\in \D^*}\phi_{\a}(A)\Big)=\L(X).
\end{equation} To see that \eqref{full measurea} holds, suppose otherwise and assume $$\L\Big(X\setminus \bigcup_{\a\in \D^*}\phi_{\a}(A)\Big)>0.$$ Moreover, let $x^*$ be a density point of $$X\setminus \bigcup_{\a\in \D^*}\phi_{\a}(A).$$ Such a point has to exist by the Lebesgue density theorem.

Let $(b_j)\in \D^\N$ be such that $\pi((b_j))=x^*,$ and let $0<r<Diam(X)$ be arbitrary. We let $n(r)\in\N$ be such that $$Diam(X) \prod_{j=1}^{n(r)}r_{b_j}<r\leq Diam(X) \prod_{j=1}^{n(r)-1}r_{b_j}.$$ The parameter $n(r)$ satisfies the following: 
\begin{equation}
\label{inclusion1}
(\phi_{b_1}\circ \cdots \circ \phi_{b_{n(r)}})(X)\subseteq B(x^*,r),
\end{equation} and 
\begin{equation}
\label{prodapprox} \frac{r \cdot \min_{i\in D}r_i}{Diam(X)}\leq \prod_{j=1}^{n(r)}r_{b_j}. 
\end{equation}
Using \eqref{inclusion1} and \eqref{prodapprox} we can now bound $$\L\Big(B(x^{*},r)\cap \Big(X\setminus \bigcup_{\a\in \D^*}\phi_{\a}(A)\Big)\Big).$$ Observe
\begin{align*}
\L\Big(B(x^{*},r)\cap \Big(X\setminus \bigcup_{\a\in \D^*}\phi_{\a}(A)\Big)\Big)&\stackrel{\eqref{inclusion1}}{\leq}  \L(B(x^{*},r))- \L((\phi_{b_1}\circ \cdots \circ \phi_{b_{n(r)}})(A))\\
&=\L(B(0,1))r^d -\Big(\prod_{j=1}^{n(r)}r_{b_j}\Big)^d\L(A)\\
&\stackrel{\eqref{prodapprox}}{\leq} \L( B(0,1))r^d -\frac{r^d\L(A)\min_{i\in \D}r_i^d }{Diam(X)^d}\\
&\leq \L( B(0,1))r^d\Big(1-\frac{\L(A)\min_{i\in \D}r_i^d }{\L( B(0,1))Diam(X)^d}\Big).
\end{align*}
Therefore $$\limsup_{r\to 0}\frac{\L\big(B(x^{*},r)\cap \big(X\setminus \bigcup_{\a\in \D^*}\phi_{\a}(A)\big)\big)}{\L(B(x^*,r))}<1.$$ This implies that $x^*$ cannot be a density point of $$X\setminus \bigcup_{\a\in \D^*}\phi_{\a}(A),$$ and we may conclude that \eqref{full measurea} holds. 

We will now show that 
\begin{equation}
\label{limsup inclusion}\bigcup_{\a\in \D^*}\phi_{\a}(A)\subseteq W_{\Phi}(z,\Psi). 
\end{equation}Let $$y\in \bigcup_{\a\in \D^*}\phi_{\a}(A)$$ be arbitrary. Let $b_1\cdots b_k\in \D^*$ and $v\in A$ be such that $$(\phi_{b_1}\circ \cdots \circ \phi_{b_k})(v)=y.$$ Let 
\begin{equation}
\label{delightful}
d:=\inf_{\a\in \D^*}\min_{i\in \D}\frac{\Psi(i\a)}{\Psi(\a)}.
\end{equation} Since $\Psi$ is weakly decaying $d>0$. 

Now suppose $\a\in \D^*$ is such that $$|\phi_{\a}(z)-v|\leq c\Psi(\a),$$ where 
\begin{equation}
\label{c equation}
c:=d^k \max_{i\in \D}\{r_i\}^{-k}.
\end{equation} Then 
\begin{align*}
|(\phi_{b_1}\circ \cdots \circ \phi_{b_k}\circ \phi_{\a})(z)-y|&=|(\phi_{b_1}\circ \cdots \circ \phi_{b_k}\circ \phi_{\a}(z)-\phi_{b_1}\circ \cdots \circ \phi_{b_k}(v)|\\
&= |\phi_{\a}(z)-v|\prod_{j=1}^k r_{b_j}\\
&\leq c\Psi(\a)\prod_{j=1}^k r_{b_j}\\
&\stackrel{\eqref{delightful}}{\leq} cd^{-k}\Psi(b_1\cdots b_k \a)\prod_{j=1}^k r_{b_j}\\
&\stackrel{\eqref{c equation}}{\leq}\Psi(b_1\cdots b_k\a).
\end{align*}
It follows that for this choice of $c,$ whenever 
\begin{equation}
\label{solutiona}
|\phi_{\a}(z)-v|\leq c\Psi(\a),
\end{equation} we also have 
\begin{equation}
\label{solutionb}|(\phi_{b_1}\circ \cdots \circ \phi_{b_k}\circ \phi_{\a}(z)-y|\leq \Psi(b_1\cdots b_k\a).
\end{equation} It follows from the definition of $A$ that $v$ has infinitely many solutions to \eqref{solutiona}, therefore $y$ has infinitely many solutions to \eqref{solutionb} and $y\in W_{\Phi}(z,\Psi)$. It follows now from \eqref{full measurea} and \eqref{limsup inclusion} that Lebesgue almost every $x\in X$ is contained in $W_{\Phi}(z,\Psi)$ \\

\noindent \textbf{Proof of statement 2.}\\
Let $\Phi$ be an IFS and $\mu$ be the pushforward of some $\sigma$-invariant ergodic probability measure $\m.$ We assume that assume $\mu\sim \L|_{X}$. Let $z$ and $\Psi$ satisfy the hypothesis of our proposition. Let $A$ be as in the proof of statement $1$. It follows from our assumptions and Lemma \ref{arbitrarily small} that $\L(A)>0$. Since $\mu\sim \L|_{X}$ we also have $\mu(A)>0$. We will prove that 
\begin{equation}
\label{full measure10}
\mu\Big(\bigcup_{\a\in \D^*}\phi_{\a}(A)\Big)=1.
\end{equation} By our assumption $\m(\pi^{-1}(A))=\mu(A)>0.$ By the ergodicity of $\m$ we have 
\begin{equation}
\label{ergodicity}
\m\Big(\bigcup_{n=0}^{\infty}\sigma^{-n}(\pi^{-1}(A))\Big)=1.
\end{equation} Now observe that
\begin{equation*}
\mu\Big(\bigcup_{\a\in \D^*}\phi_{\a}(A)\Big)=\m\Big(\pi^{-1}\Big(\bigcup_{\a\in \D^*}\phi_{\a}(A)\Big)\Big)\geq \m\Big(\bigcup_{n=0}^{\infty}\sigma^{-n}(\pi^{-1}(A))\Big).
\end{equation*} Therefore \eqref{ergodicity} implies \eqref{full measure10}. Since $\mu \sim \L|_{X}$ it follows that 
\begin{equation}
\label{full measureABC} \L\Big(\bigcup_{\a\in \D^*}\phi_{\a}(A)\Big)=\L(X).
\end{equation}We now let $$y\in \bigcup_{\a\in \D^*}\phi_{\a}(A)$$ be arbitrary. Defining appropriate analogues of $b_1\cdots b_k$ and $c$ as in the proof of statement $1,$ it will follow that $y\in W_{\Phi}(z,\Psi)$. Therefore $$\bigcup_{\a\in \D^*}\phi_{\a}(A)\subseteq W_{\Phi}(z,\Psi).$$ Combining this fact with \eqref{full measureABC} completes the proof of statement $2$. 
\end{proof}

Proposition \ref{full measure} is a useful technique for proving full measure statements, but there is an additional cost as we require the function $\Psi$ to be weakly decaying. The following proposition requires no extra condition on the function $\Psi$, but does require some stronger separation assumptions. Before stating this proposition we recall some of our earlier definitions. Given a slowly decaying probability measure $\m$ we define
$$L_{\m,n}=\{\a\in\D^*: \m([a_1\cdots a_{|\a|}])\leq c_{\m}^n<\m([a_1\cdots a_{|\a|-1}])\}$$ and $$R_{\m,n}=\#L_{\m,n}.$$ Moreover, given $z\in X$ we define $$Y_{\m,n}(z):=\{\phi_{\a}(z)\}_{\a\in L_{\m,n}}.$$
\begin{prop}
\label{separated full measure}
Suppose $\Phi=\{A_i x+t_i\}$ is a collection of affine contractions and $\mu$ is the pushforward of a Bernoulli measure $\m$. Assume that one of following properties is satisfied:
\begin{itemize}
	\item $\Phi$ consists solely of similarities.
	\item $d=2$ and all the matrices $A_i$ are equal.
	\item All the matrices $A_i$ are simultaneously diagonalisable.
\end{itemize} Let $z\in X$ and suppose that for some $s>0$ there exists a subsequence $(n_k)$ satisfying $$\lim_{k\to\infty}\frac{T(Y_{\m,n_k}(z),\frac{s}{R_{\m,n_k}^{1/d}})}{R_{\m,n_k}}=1.$$ Then $\mu\sim \L|_{X},$ and for any $h$ that satisfies $$\sum_{k=1}^{\infty}h(n_k)=\infty,$$ we have that Lebesgue almost every $x\in X$ is contained in $U_{\Phi}(z,\m,h).$


\end{prop}
The proof of Proposition \ref{separated full measure} is more involved than Proposition \ref{full measure} and will rely on the following technical result. 
\begin{lemma}
	\label{equivalent measures}
\begin{enumerate}
	\item Let $\mu$ be a self-similar measure. Then either $\mu\sim \L|_{X}$ or $\mu$ is singular. 
	\item Suppose $\Phi=\{A_i x+t_i\}$ is a collection of affine contractions and one of following properties is satisfied:
	\begin{itemize}
		\item $d=2$ and all the matrices $A_i$ are equal.
		\item All the matrices $A_i$ are simultaneously diagonalisable.
	\end{itemize} Then if $\mu$ is the pushforward of a Bernoulli measure we have either $\mu\sim \L|_{X}$ or $\mu$ is singular. 
	\item Let $\Phi$ be an arbitrary iterated function system and $\mu$ be the pushforward of a $\sigma$-invariant ergodic probability measure $\m$. Then either $\mu\ll\L$ or $\mu$ is singular (i.e. $\mu$ is of pure type).
\end{enumerate}	
\end{lemma}
\begin{proof}
	A proof of statement $1$ can be found in \cite{PeScSo}. It makes use of an argument originally appearing in \cite{MauSim}. Statement $2$ was proved in \cite[Section 4.4.]{Shm3} using ideas of Guzman \cite{Guz} and Fromberg \cite{NSW}.

	We could not find a proof of statement $3$ so we include one for completeness. Suppose that $\mu$ is not singular, then by the Lebesgue decomposition theorem $\mu= \mu_0+\mu_1$ where $\mu_0\ll\L,$ $\mu_1\perp \L,$ and $\mu_0(X)>0.$ Suppose that $\mu\neq \mu_0$. Then there exists $A$ such that $\mu_1(A)>0.$ Since $\mu_1\perp \L,$ we may assume without loss of generality that $\L(A)=0.$ Using the ergodicity of $\m,$ it follows from an analogous argument to that used in the proof of statement $2$ from Proposition \ref{full measure} that $\mu(\cup_{\a\in \D^*}\phi_{\a}(A))=1.$ Therefore we must have $\mu_{0}(\cup_{\a\in \D^*}\phi_{\a}(A))>0,$ and by absolute continuity $\L(\cup_{\a\in \D^*}\phi_{\a}(A))>0.$ Since each $\phi_{\a}$ is a contraction, $\L(A)=0$ implies that $\L(\phi_{\a}(A))=0$ for all $\a\in \D^*$. This contradicts that $\L(\cup_{\a\in \D^*}\phi_{\a}(A))>0.$ Therefore we must have $\mu= \mu_0.$
\end{proof}Only statement $1$ and statement $2$ from Lemma \ref{equivalent measures} will be needed in the proof of Proposition \ref{separated full measure}. Statement $3$ is needed in the proof of the following result which we formulate as generally as possible. 

\begin{prop}
	\label{absolute continuity}
Let $\mu$ be the pushforward of a slowly decaying $\sigma$-invariant ergodic probability measure $\m$. If for some $z\in X$ and $s>0$ we have $$\limsup_{n\to\infty} \frac{T\big(Y_{\m,n}(z),\frac{s}{R_{\m,n}^{1/d}}\big)}{R_{\m,n}}>0,$$ then $\mu\ll\L$. 
\end{prop}
\begin{proof}
	We start our proof by remarking that for any $z\in X$, 
	\begin{equation}
	\label{weak star}
	\mu=\lim_{n\to\infty} \sum_{\a\in L_{\m,n}}\m([\a])\cdot \delta_{\phi_{\a}(z)},
	\end{equation} where the convergence is meant with respect to the weak star topology\footnote{Equation \eqref{weak star} can be verified by checking that for each $n\in \N$ the measure $\sum_{\a\in L_{\m,n}}\m([\a])\cdot \delta_{\phi_{\a}(z)}$ is the pushforward of an appropriately chosen measure $\m_{n}$ on $\D^{\mathbb{N}},$ and that this sequence of measures satisfies $\lim_{n\to\infty}\m_{n}=\m$.}.  By our assumption, for some $z\in X$ and $s>0,$ there exists a sequence $(n_k)$ and $c>0$ such that 
	\begin{equation}
	\label{mass equation}
	\frac{T\big(Y_{\m,n_k}(z),\frac{s}{R_{\m,n_k}^{1/d}}\big)}{R_{\m,n_k}}>c
	\end{equation} for all $k$. Define $$\mu_{n_k}':=\sum_{\stackrel{\a\in L_{\m,n_k}}{\phi_{\a}(z)\in S\big(Y_{\m,n_k}(z),\frac{s}{R_{\m,n_k}^{1/d}}\big)}}\m([\a])\cdot \delta_{\phi_{\a}(z)}$$ and $$\mu_{n_k}'':=\sum_{\stackrel{\a\in L_{\m,n_k}}{\phi_{\a}(z)\notin S\big(Y_{\m,n_k}(z),\frac{s}{R_{\m,n_k}^{1/d}}\big)}}\m([\a])\cdot \delta_{\phi_{\a}(z)}.$$ Then $$\sum_{\a\in L_{\m,n_k}}\m([\a])\cdot \delta_{\phi_{\a}(z)}=\mu_{n_k}'+\mu_{n_k}''.$$ By taking subsequences if necessary, we may also assume without loss of generality that there exists two finite measures $\nu'$ and $\nu''$ such that $\lim_{k\to\infty}\mu'_{n_k}= \nu'$ and $\lim_{k\to\infty}\mu''_{n_k}= \nu''.$ Therefore by \eqref{weak star} we have $\mu=\nu'+\nu''$. We will prove that $\nu'(X)>0$ and $\nu'$ is absolutely continuous with respect to the Lebesgue measure. Since $\mu$ is either singular or absolutely continuous by Lemma \ref{equivalent measures}, it will follow that $\mu\ll \L$.
	
	
	
It follows from the definition of $L_{\m,n}$ that for any $\a,\a'\in L_{\m,n}$ we have $$\m([\a])\asymp \m([\a']).$$ This implies that for any $\a\in\L_{\m,n}$ we have 
\begin{equation}
\label{BBC}
\m([\a])\asymp R_{\m,n}^{-1}.
\end{equation} Using \eqref{mass equation} and \eqref{BBC}, we have that $$\mu'_{n_k}(X)=\sum_{\stackrel{\a\in L_{\m,n_k}}{\phi_{\a}(z)\in S\big(Y_{\m,n_k}(z),\frac{s}{R_{\m,n_k}^{1/d}}\big)}}\m([\a])\asymp \frac{T\big(Y_{\m,n_k}(z),\frac{s}{R_{\m,n_k}^{1/d}}\big)}{R_{\m,n_k}}\asymp \frac{c\cdot R_{\m,n_k}}{R_{\m,n_k}}\asymp 1.$$ Therefore $\nu'(X)\geq \lim_{k\to\infty}\mu'_{n_k}(X) >0$. Now we prove that $\nu'$ is absolutely continuous. Fix an arbitrary open $d$-dimensional cube $(x_1,x_1+r)\times\cdots \times (x_d,x_d+r)\subset \mathbb{R}^d,$ we have
	\begin{align}
&	\mu'_{n_k}((x_1,x_1+r)\times \cdots\times (x_d,x_d+r))\nonumber\\
&=\sum_{\stackrel{\a\in L_{\m,n}}{\phi_{\a}(z)\in S\big(Y_{\m,n_k}(z),\frac{s}{R_{\m,n}^{1/d}}\big)\cap (x_1,x_1+r)\times\cdots\times (x_d,x_d+r)}} \m([\a])\nonumber\\
	&=\mathcal{O}\left(\frac{\#\Big\{\phi_{\a}(z)\in S\big(Y_{\m,n_k}(z),\frac{s}{R_{\m,n}^{1/d}}\big)\cap (x_1,x_1+r)\times\cdots\times (x_d,x_d+r)\Big\}}{R_{\m,n}}\right)\label{jiggly}.
	\end{align}In the last line we used \eqref{BBC}. Since the elements of $S\big(Y_{\m,n_k}(z),\frac{s}{R_{\m,n}^{1/d}}\big)$ are separated by a factor $\frac{s}{R_{\m,n}^{1/d}},$ it follows from a volume argument that we must have
	$$\#\Big\{\phi_{\a}(z)\in S\big(Y_{\m,n_k}(z),\frac{s}{R_{\m,n}^{1/d}}\big)\cap (x_1,x_1+r)\times\cdots\times (x_d,x_d+r)\Big\}=\mathcal{O}\left(\frac{r^dR_{\m,n}}{s^d}\right).$$ 
	Substituting this bound into \eqref{jiggly} we have $$\mu'_{n_k}((x_1,x_1+r)\times \cdots\times (x_d,x_d+r))=\mathcal{O}\left(\frac{r^dR_{\m,n}}{s^dR_{\m,n}}\right)=\mathcal{O}\left(\frac{r^d}{s^d}\right).$$
	Letting $k\to\infty$, it follows that for any $d$-dimensional cube we have $$\nu'((x_1,x_1+r)\times \cdots\times (x_d,x_d+r))=\mathcal{O}\left(\frac{r^d}{s^d}\right).$$ Since $s$ is fixed $\nu'$ must be absolutely continuous. This completes our proof. 
\end{proof}
As well as Proposition \ref{absolute continuity} being used in our proof of Proposition \ref{separated full measure}, it can be seen as a new tool for proving that measures are absolutely continuous. Proposition \ref{absolute continuity} can be used in conjunction with Lemma \ref{density separation lemma} and Lemma \ref{integral bound} to recover known results on the absolute continuity of measures within a parameterised family. We include one such instance of this in Section \ref{applications}, where we recover the well known result due to Solomyak that for almost every $\lambda\in(1/2,1),$ the unbiased Bernoulli convolution is absolutely continuous \cite{Sol}.

With these preliminary results we are now in a position to prove Proposition \ref{separated full measure}.
\begin{proof}[Proof of Proposition \ref{separated full measure}]
Let $\Phi$ be an IFS satisfying one of our conditions and $\mu$ be the pushforward of a Bernoulli measure $\m$. Let $z\in X,$  $s>0,$ and $(n_k)$ satisfy the hypothesis of our proposition. By an application of Proposition \ref{absolute continuity}, we know that $\mu\ll \L.$ Moreover, by Lemma \ref{equivalent measures} we also know that $\mu\sim \L|_{X}$.

To prove our result, it is sufficient to show that
\begin{equation}
\label{Want to showa}
\L\left(\left\{x\in\mathbb{R}^d:x\in \bigcup_{\a\in L_{\m,n_k}}B\left(\phi_{\a}(z),\Big(\frac{h(n_k)}{R_{\m,n_k}}\Big)^{1/d}\right)\textrm{for i.m. }k\in \N\right\}\right)=\L(X),
\end{equation} for any $h$ satisfying \begin{equation}
\label{h divergence}
\sum_{k=1}^{\infty}h(n_k)=\infty.
\end{equation}
It will then follow from Lemma \ref{arbitrarily small}, and the fact that $\m([\a])\asymp R_{\m,n}^{-1}$ for any $\a\in L_{\m,n},$ that \eqref{Want to showa} implies that Lebesgue almost every $x\in X$ is contained in $U_{\Phi}(z,\m,h)$ for any $h$ satisfying \eqref{h divergence} 

Our proof of \eqref{Want to showa} will follow from a similar type of argument to that given in the proof of Proposition \ref{fixed omega}. Where necessary to avoid repetition we will omit certain details. Our strategy for proving \eqref{Want to showa} holds is to prove that for Lebesgue almost every $y\in X,$ there exists $c_y>0,$ such that for all $r$ sufficiently small we have 
\begin{equation}
\label{Want to showb}\L\left(B(y,2r)\cap\left\{x\in\mathbb{R}^d:x\in \bigcup_{\a\in L_{\m,n_k}}B\left(\phi_{\a}(z),\Big(\frac{h(n_k)}{R_{\m,n_k}}\Big)^{1/d}\right)\textrm{for i.m. }k\in \N\right\}\right)\geq c_y r^d.
\end{equation}Importantly $c_y$ will not depend upon $r$. It follows then by an application of the Lebesgue density theorem that \eqref{Want to showb} implies \eqref{Want to showa}. As in the proof of Proposition \ref{fixed omega}, we split our proof of \eqref{Want to showb} into smaller steps.\\

\noindent \textbf{Step $1$. Local information.}\\
We have already established that $\mu\sim \L|_{X}.$ Let $d$ denote the Radon-Nikodym derivative $d\mu/d\L$. For $\mu$-almost every $y$ we must have $d(y)>0$. It follows now by the Lebesgue differentiation theorem, and the fact that $\mu\sim \L|_{X},$ that for Lebesgue almost every $y\in X$ we have $$\lim_{r\to 0}\frac{\mu(B(y,r))}{\L(B(y,r))}= d(y)>0.$$ In what follows $y$ is a fixed element of $X$ satisfying this property. Let $r^*$ be such that for all $r\in(0,r^*),$ we have 
\begin{equation} 
\label{quack}
\frac{d(y)}{2}<\frac{\mu(B(y,r))}{\L(B(y,r))}<  2d(y).
\end{equation} Now using that $\mu$ is the weak star limit of the sequence of measures $$\mu_{n_k}:=\sum_{\a\in L_{\m,n_k}}\m([\a])\cdot\delta_{\phi_{\a}(z)},$$ together with \eqref{quack}, we can assert that for each $r\in(0,r^*),$ for $k$ sufficiently large we have  \begin{equation}
\label{dumbdumb}
\mu_{n_k}(B(y,r))=\sum_{\stackrel{\a\in L_{\m,n_k}}{\phi_{\a}(z)\in B(y,r)}}\m([\a])\asymp d(y)r^d.
\end{equation} By construction we know that each $\a\in L_{\m,n_k}$ satisfies $$\m([\a])\asymp R_{\m,n_k}^{-1}.$$ Therefore it follows from \eqref{dumbdumb} that for all $k$ sufficiently large
\begin{equation}
\label{local growth near}\frac{\#\{\a\in L_{\m,n_k}:\phi_{\a}(z)\in B(y,r)\}}{R_{\m,n_k}}\asymp d(y)r^d.
\end{equation} 
Let 
$$A(y,r,k):=\left\{\a\in L_{\m,n_{k}}:\phi_{\a}(z)\in S\Big(Y_{\m,n_k}(z),\frac{s}{R_{\m,n_k}^{1/d}}\Big)\cap B(y,r)\right\}.$$
Since $$\lim_{k\to\infty}\frac{T(Y_{\m,n_k}(z),\frac{s}{R_{\m,n_k}^{1/d}})}{R_{\m,n_k}}=1,$$ it follows from \eqref{local growth near} that there exists $K(r)\in\mathbb{N},$ such that for any $k\geq K(r)$ we have
\begin{equation}
\label{local growth}
\# A(y,r,k)\asymp  R_{\m,n_k}d(y)r^d .
\end{equation} Equation \eqref{local growth} shows that for each $k\geq K(r),$ there is a large separated set that is local to $B(y,r)$. 

We will prove that there exists $c_y>0$ such that

\begin{equation}
\label{WTS3}
\L\left(B(y,2r)\cap\left\{x\in\mathbb{R}^d:x\in \bigcup_{\a\in A(y,r,k)}B\left(\phi_{\a}(z),\Big(\frac{h(n_k)}{R_{\m,n_k}}\Big)^{1/d}\right)\textrm{for i.m. }k\in \N\right\}\right)\geq c_yr^d.
\end{equation} Equation \eqref{WTS3} implies \eqref{Want to showb}. So to complete our proof it suffices to show that \eqref{WTS3} holds.

For later use we note that \begin{equation}
\label{divergenceaa}
\sum_{k=K(r)}^{\infty}\sum_{\a\in A(y,r,k)}\L\left(B\big(\phi_{\a}(z),\Big(\frac{h(n_k)}{R_{\m,n_k}}\Big)^{1/d}\big)\right)=\infty.
\end{equation}Equation \eqref{divergenceaa} is true because
\begin{align*}
\sum_{k=K(r)}^{\infty}\sum_{\a\in A(y,r,k)}\L\Big(B\big(\phi_{\a}(z),\Big(\frac{h(n_k)}{R_{\m,n_k}}\Big)^{1/d}\big)\Big)&= \sum_{k=1}^{\infty}\# A(y,r,k)\frac{\L(B(0,1))h(n_k)}{R_{\m,n_k}}\\
&\stackrel{\eqref{local growth}}{\asymp} d(y)r^d\L(B(0,1))\sum_{k=K(r)}^{\infty} h(n_k)\\
&\stackrel{\eqref{h divergence}}{=}\infty.
\end{align*}\\

\noindent \textbf{Step $2.$ Replacing our approximating function.}\\
Let $$g(n_k):=\min\Big\{\Big(\frac{h(n_k)}{R_{\m,n_k}}\Big)^{1/d},\frac{s}{3R_{\m,n_k}^{1/d}}\Big\}.$$ For each $k\geq K(r)$ we define the set $$E_{n_k}:=\bigcup_{\a\in A(y,r,k)}B(\phi_{\a}(z),g(n_k)).$$ By construction the balls in this union are disjoint. Therefore by \eqref{local growth}, for each $k\geq K(r)$
 \begin{equation}
\label{local measure growth}\L(E_{n_k})\asymp g(n_k)^dR_{\m,n_k}d(y)r^d.
\end{equation} By a similar argument to that given in the proof of Proposition \ref{fixed omega}, it follows from \eqref{divergenceaa} that we have 
\begin{equation}
\label{divergencebb}\sum_{k=K(r)}^{\infty}\L(E_{n_k})=\infty.
\end{equation}
Therefore, the assumptions of Lemma \ref{Erdos lemma} are satisfied. We will use this lemma to show that 
\begin{equation}
\label{WTS4}
\L\Big( \limsup_{k\to\infty}E_{n_k}\Big)\geq c_yr^d.
\end{equation}Since $$\limsup_{k\to\infty}E_{n_k}\subseteq B(y,2r),$$ because $\lim_{k\to\infty}g(n_k)=0$, we see that \eqref{WTS4} implies \eqref{WTS3}. So verifying \eqref{WTS4} will complete our proof. \\

\noindent \textbf{Step $3$. Bounding $L(E_{n_k}\cap E_{n_{k'}})$.}\\
By an analogous argument to that given in the proof of Proposition \ref{fixed omega}, we can show that for any $\a\in A(y,r,k),$ we have $$\#\left\{\a'\in A(y,r,k'):B(\phi_{\a}(z),g(n_k))\cap B(\phi_{\a'}(z),g(n_{k'}))\neq \emptyset\right\}=\mathcal{O}\Big(\frac{g(n_k)^d R_{\m,n_{k'}}}{s^d}+1\Big).$$ Using this estimate and \eqref{local growth}, it can be shown that for any distinct $k,k'\geq K(r)$ we have
\begin{equation}
\label{intersection bounda}\L(E_{n_k}\cap E_{n_k'})=\mathcal{O}\left(d(y) r^dR_{\m,n_k}g(n_{k'})^d\left(\frac{g(n_k)^d R_{\m,n_{k'}}}{s^d}+1\right)\right).
\end{equation}\\

\noindent \textbf{Step $4$. Applying Lemma \ref{Erdos lemma}.}\\
Using \eqref{intersection bounda}, we can then replicate the arguments used in the proof of Proposition \ref{fixed omega} to show that  
\begin{equation}
\label{localbound1}\sum_{k,k'=K(r)}^{Q}\L(E_{n_k}\cap E_{n_k'})=\mathcal{O}\Big(d(y)r^d\Big(\sum_{k=K(r)}^QR_{\m,n_k}g(n_k)^d\Big)^2\Big).
\end{equation}We emphasise here that the underlying constants in \eqref{localbound1} do not depend upon $r$. By \eqref{local measure growth} we have 
\begin{equation}
\label{localbound2}\Big(\sum_{k=K(r)}^{Q}\L(E_{n_k})\Big)^2\asymp r^{2d}d(y)^2 \Big(\sum_{k=K(r)}^{Q}R_{\m,n_k}g(n_k)\Big)^2.
\end{equation} Applying Lemma \ref{Erdos lemma} in conjunction with \eqref{localbound1} and \eqref{localbound2} yields $$\L\Big(\limsup_{k\to\infty}E_{n_k}\Big)\geq c_yr^d,$$ for some $c_y>0$ that does not depend upon $r$. Therefore \eqref{WTS4} holds and we have completed our proof.



\end{proof}

\section{Applications of Proposition \ref{general prop}}
\label{applications}
In this section we apply the results of Section \ref{Preliminaries} to prove Theorems \ref{1d thm}, \ref{translation thm} and \ref{random thm}. We begin by briefly explaining why the exponential growth condition appearing in Proposition \ref{general prop} will always be satisfied in our proofs. 

Let $\m$ be a slowly decaying probability measure supported on $\D^{\N}$. We remark that each $\a\in L_{\m,n}$ satisfies $$\m([\a])\asymp c_{\m}^n.$$ Recall that $c_{\m}$ is defined in Section \ref{auxillary sets}. Importantly the cylinders corresponding to elements of $L_{\m,n}$ are disjoint, and we have $\m(\cup_{\a\in L_{\m,n}}[\a])=1$. It follows from these observations that $$R_{\m,n}\asymp c_{\m}^{-n}.$$  Similarly, if $\tilde{L}_{\m,n}\subseteq L_{\m,n}$ and  $\m(\cup_{\a\in \tilde{L}_{\m,n}}[\a])>d$ for some $d>0$ for each $n$, then $$\#\tilde{L}_{\m,n}=:\tilde{R}_{\m,n}\asymp c_{\m}^{-n}.$$ Where the underlying constants depend upon $d$ but are independent of $n$. In our proofs of Theorems \ref{1d thm}, \ref{translation thm} and \ref{random thm}, it will be necessary to define an $\tilde{L}_{\m,n}$ contained in $L_{\m,n},$ whose union of cylinders has measure uniformly bounded away from zero. By the above discussion the exponential growth condition appearing in Proposition \ref{general prop} will automatically be satisfied by $\tilde{L}_{\m,n}$.


\subsection{Proof of Theorem \ref{1d thm}}
Before proceeding with our proof of Theorem \ref{1d thm} we recall some useful results from \cite{Sol}.


\begin{lemma}[Lemma 2.1 \cite{Sol}]
	\label{delta lemma}
	For any $\epsilon_1>0,$ there exists $\delta=\delta(\epsilon_1)>0$ such that if $g\in \mathcal{B}_{\Gamma}$, and $g(0)\neq 0,$ then $$x\in (0,\alpha(\mathcal{B}_{\Gamma})-\epsilon_1], |g(x)|<\delta \implies |g'(x)|>\delta.$$
\end{lemma}

Lemma \ref{delta lemma} has the following useful consequence.

\begin{lemma}
	\label{zero intervals}
	Let $\epsilon_1>0$ and $\delta(\epsilon_1)>0$ be as in Lemma \ref{delta lemma}. Then for any $\epsilon_2>0$ and $g\in \mathcal{B}_{\Gamma}$ such that $g(0)\neq 0$, we have $$\L\big(\{\lambda\in(0,\alpha(\mathcal{B}_{\Gamma})-\epsilon_1]: |g(\lambda)|\leq \epsilon_2\}\big)=\mathcal{O}(\epsilon_2).$$ Where the underlying constant depends upon $\delta(\epsilon_1).$	
\end{lemma}Lemma \ref{zero intervals} follows from the analysis given in Section 2.4. from \cite{Sol}. Equipped with Lemma \ref{zero intervals} and the results of Section \ref{Preliminaries}, we can now prove Theorem \ref{1d thm}.

\begin{proof}[Proof of Theorem \ref{1d thm}]We treat each statement in this theorem individually. We start with the proof of statement $1$.\\
	
\noindent \textbf{Proof of statement $1$.}\\
Let us start by fixing $\m$ a slowly decaying $\sigma$-invariant ergodic probability measure with $\h(\m)>0,$ and $(a_j)\in \D^{\mathbb{N}}.$ Let $\epsilon_1>0$ be arbitrary. We now choose $\epsilon_2>0$ sufficiently small so that we have
\begin{equation}
\label{epsilons ratio}
e^{\h(\m)-\epsilon_2}(e^{-\h(\m)}+\epsilon_1)>1.
\end{equation}
By the Shannon-McMillan-Breiman theorem, we know that for $\m$-almost every $\a\in \D^{\mathbb{N}}$ we have 
\begin{equation}
\label{SMB}
\lim_{n\to\infty}\frac{-\log \m([a_1\cdots a_n])}{n}= \h(\m).
\end{equation}
It follows from \eqref{SMB} and Egorov's theorem, that there exists $C=C(\epsilon_2)>0$ such that $$\m\Big(\a\in \D^{\N}: \frac{e^{k(-\h(\m)-\epsilon_2)}}{C}\leq \m([a_1\cdots a_k])\leq C e^{k(-\h(\m)+\epsilon_2)},\, \forall k\in\mathbb{N}\Big)>1/2.$$ Let $$\tilde{L}_{\m,n}:=\Big\{\a\in L_{\m,n}: \frac{e^{k(-\h(\m)-\epsilon)}}{C}\leq \m([a_1\cdots a_{k}])\leq C e^{k(-\h(\m)+\epsilon)},\, \forall 1\leq k\leq |\a|\Big\}$$and 
$$\tilde{R}_{\m,n}:=\#\tilde{L}_{\m,n}.$$
Since $$\m\Big(\bigcup_{\a\in L_{\m,n}}[\a]\Big)=1,$$ it follows from the above that 
\begin{equation}
\label{half bound}
\m\Big(\bigcup_{\a\in \tilde{L}_{\m,n}}[\a]\Big)>1/2. 
\end{equation}By the discussion at the beginning of this section we know that $\tilde{R}_{\m,n}$ satisfies the exponential growth condition of Proposition \ref{general prop}. It also follows from this discussion that 
\begin{equation}
\label{equivalen cardinality}\tilde{R}_{\m,n}\asymp R_{\m,n}.
\end{equation}
Recalling the notation used in Section \ref{Preliminaries}, let $$R(\lambda,s,n):=\left\{(\a,\a')\in \tilde{L}_{\m,n}\times \tilde{L}_{\m,n}:\left|\phi_{\a}\Big(\sum_{j=1}^{\infty}d_{a_j}\lambda^{j-1}\Big)-\phi_{\a'}\Big(\sum_{j=1}^{\infty}d_{a_j}\lambda^{j-1}\Big)\right|\leq \frac{s}{\tilde{R}_{\m,n}}\textrm{ and }\a\neq\a'\right\}.$$The main step in our proof of statement $1$ is to show that  
\begin{equation}
\label{transversality separation}
\int_{e^{-\h(\m)}+\epsilon_1}^{\alpha(\mathcal{B}_{\Gamma})-\epsilon_1} \frac{\#R(\lambda,s,n)}{\tilde{R}_{\m,n}}d\lambda=\mathcal{O}(s).
\end{equation} We will then be able to employ the results of Section \ref{Preliminaries} to prove our theorem. Our proof of \eqref{transversality separation} is based upon an argument of Benjamini and Solomyak \cite{BenSol}, which in turn is based upon an argument of Peres and Solomyak \cite{PerSol}.\\

\noindent \textbf{Step 1. Proof of \eqref{transversality separation}.}\\
Observe the following:
\begin{align}
\label{2b substituted}
&\int_{e^{-\h(\m)}+\epsilon_1}^{\alpha(\mathcal{B}_{\Gamma})-\epsilon_1} \frac{\#R(\lambda,s,n)}{\tilde{R}_{\m,n}}d\lambda \nonumber\\
&=\int_{e^{-\h(\m)}+\epsilon_1}^{\alpha(\mathcal{B}_{\Gamma})-\epsilon_1}\frac{1}{\tilde{R}_{\m,n}}\sum_{\stackrel{\a,\a'\in \tilde{L}_{\m,n}}{\a\neq \a'}}\chi_{[-\frac{s}{\tilde{R}_{\m,n}},\frac{s}{\tilde{R}_{\m,n}}]}\left(\phi_{\a}\Big(\sum_{j=1}^{\infty}d_{a_j}\lambda^{j-1}\Big)-\phi_{\a'}\Big(\sum_{j=1}^{\infty}d_{a_j}\lambda^{j-1}\Big)\right)\, d\lambda \nonumber \\
&=\mathcal{O}\left(\tilde{R}_{\m,n}\int_{e^{-\h(\m)}+\epsilon_1}^{\alpha(\mathcal{B}_{\Gamma})-\epsilon_1}\sum_{\stackrel{\a,\a'\in \tilde{L}_{\m,n}}{\a\neq \a'}}\chi_{[-\frac{s}{\tilde{R}_{\m,n}},\frac{s}{\tilde{R}_{\m,n}}]}\left(\phi_{\a}\Big(\sum_{j=1}^{\infty}d_{a_j}\lambda^{j-1}\Big)-\phi_{\a'}\Big(\sum_{j=1}^{\infty}d_{a_j}\lambda^{j-1}\Big)\right)\m([\a])\m([\a'])\, d\lambda\right)\nonumber \\
&=\mathcal{O}\left(\tilde{R}_{\m,n}\sum_{\stackrel{\a,\a'\in \tilde{L}_{\m,n}}{\a\neq \a'}} \m([\a])\m([\a'])\int_{e^{-\h(\m)}+\epsilon_1}^{\alpha(\mathcal{B}_{\Gamma})-\epsilon_1} \chi_{[-\frac{s}{\tilde{R}_{\m,n}},\frac{s}{\tilde{R}_{\m,n}}]}\left(\phi_{\a}\Big(\sum_{j=1}^{\infty}d_{a_j}\lambda^{j-1}\Big)-\phi_{\a'}\Big(\sum_{j=1}^{\infty}d_{a_j}\lambda^{j-1}\Big)\right)d\lambda\right).
\end{align}
In the penultimate line we used that for any $\a\in \tilde{L}_{\m,n}$ we have $\m([\a])\asymp \tilde{R}_{\m,n}^{-1}$. 

Note that for any distinct $\a,\a'\in \tilde{L}_{\m,n}$ we have  $$\phi_{\a}\Big(\sum_{j=1}^{\infty}d_{a_j}\lambda^{j-1}\Big)-\phi_{\a'}\Big(\sum_{j=1}^{\infty}d_{a_j}\lambda^{j-1}\Big)\in \mathcal{B}_{\Gamma}.$$  Let $|\a\wedge \a'|:=\inf\{k:a_k\neq a_k'\}.$ Then $$\phi_{\a}\Big(\sum_{j=1}^{\infty}d_{a_j}\lambda^{j-1}\Big)-\phi_{\a'}\Big(\sum_{j=1}^{\infty}d_{a_j}\lambda^{j-1}\Big)=\lambda^{|\a\wedge \a'|-1}g(\lambda),$$ for some $g\in\mathcal{B}_{\Gamma}$ satisfying $g(0)\neq 0$. Therefore, for any distinct $\a,\a'\in \tilde{L}_{\m,n}$ we have
\begin{align*}
&\int_{e^{-\h(\m)}+\epsilon_1}^{\alpha(\mathcal{B}_{\Gamma})-\epsilon_1} \chi_{[-\frac{s}{\tilde{R}_{\m,n}},\frac{s}{\tilde{R}_{\m,n}}]}\left(\phi_{\a}\Big(\sum_{j=1}^{\infty}d_{a_j}\lambda^{j-1}\Big)-\phi_{\a'}\Big(\sum_{j=1}^{\infty}d_{a_j}\lambda^{j-1}\Big)\right)d\lambda\\
&=\L\left(\left\{\lambda\in (e^{-\h(\m)}+\epsilon_1,\alpha(\mathcal{B}_{\Gamma})-\epsilon_1):\phi_{\a}\Big(\sum_{j=1}^{\infty}d_{a_j}\lambda^{j-1}\Big)-\phi_{\a'}\Big(\sum_{j=0}^{\infty}d_{a_j}\lambda^{j-1}\Big)\in \Big[-\frac{s}{\tilde{R}_{\m,n}},\frac{s}{\tilde{R}_{\m,n}}\Big]\right\}\right)\\
&=\L\left(\left\{\lambda\in (e^{-\h(\m)}+\epsilon_1,\alpha(\mathcal{B}_{\Gamma})-\epsilon_1)):\lambda^{|\a\wedge \a'|-1}g(\lambda)\in \Big[-\frac{s}{\tilde{R}_{\m,n}},\frac{s}{\tilde{R}_{\m,n}}\Big]\right\}\right)\\
&=\L\left(\left\{\lambda\in (e^{-\h(\m)}+\epsilon_1,\alpha(\mathcal{B}_{\Gamma})-\epsilon_1)):g(\lambda)\in \Big[-\frac{s\lambda^{-|\a\wedge \a'|+1}}{\tilde{R}_{\m,n}},\frac{s\lambda^{-|\a\wedge \a'|+1}}{\tilde{R}_{\m,n}}\Big]\right\}\right)\\
&\leq \L\left(\left\{\lambda\in (e^{-\h(\m)}+\epsilon_1,\alpha(\mathcal{B}_{\Gamma})-\epsilon_1)):g(\lambda)\in \Big[-\frac{s(e^{-\h(\m)}+\epsilon_1)^{-|\a\wedge \a'|+1}}{\tilde{R}_{\m,n}},\frac{s(e^{-\h(\m)}+\epsilon_1)^{-|\a\wedge \a'|+1}}{\tilde{R}_{\m,n}}\Big]\right\}\right)\\
&=\mathcal{O}\left(\frac{s(e^{-\h(\m)}+\epsilon_1)^{-|\a\wedge\a'|}}{\tilde{R}_{\m,n}}\right).
\end{align*}
Where in the last line we used Lemma \ref{zero intervals}. Summarising the above, we have shown that 
\begin{equation}
\label{transversality bound}
\int_{e^{-\h(\m)}+\epsilon_1}^{\alpha(\mathcal{B}_{\Gamma})-\epsilon_1} \chi_{[-\frac{s}{\tilde{R}_{\m,n}},\frac{s}{\tilde{R}_{\m,n}}]}\left(\phi_{\a}(\sum_{j=1}^{\infty}d_{a_j}\lambda^{j-1})-\phi_{\a'}(\sum_{j=1}^{\infty}d_{a_j}\lambda^{j-1})\right)d\lambda=\mathcal{O}\left(\frac{s(e^{-\h(\m)}+\epsilon_1)^{-|\a\wedge\a'|}}{\tilde{R}_{\m,n}}\right).
\end{equation}
Substituting \eqref{transversality bound} into \eqref{2b substituted} we obtain:
\begin{align*}
\int_{e^{-\h(\m)}+\epsilon_1}^{\alpha(\mathcal{B}_{\Gamma})-\epsilon_1} \frac{\#R(\lambda,s,n)}{\tilde{R}_{\m,n}}d\lambda&=\mathcal{O}\left(\tilde{R}_{\m,n}\sum_{\stackrel{\a,\a'\in \tilde{L}_{\m,n}}{\a\neq \a'}} \m([\a])\m([\a']) \frac{s(e^{-\h(\m)}+\epsilon_1)^{-|\a\wedge \a'|}}{\tilde{R}_{\m,n}}\right)\\
&=\mathcal{O}\left(s\sum_{\stackrel{\a,\a'\in \tilde{L}_{\m,n}}{\a\neq \a'}} \m([\a])\m([\a']) (e^{-\h(\m)}+\epsilon_1)^{-|\a\wedge \a'|}\right)\\
&=\mathcal{O}\left(s\sum_{\a\in \tilde{L}_{\m,n}}\m([\a])\sum_{k=1}^{|\a|-1}\sum_{\stackrel{\a'\in\tilde{L}_{\m,n} }{|\a\wedge\a'|=k}}\m([\a']) (e^{-\h(\m)}+\epsilon_1)^{-k}\right)\\
&=\mathcal{O}\left(s\sum_{\a\in \tilde{L}_{\m,n}}\m([\a])\sum_{k=1}^{|\a|-1}\m([a_1\cdots a_{k-1}])(e^{-\h(\m)}+\epsilon_1)^{-k}\right)\\ 
&=\mathcal{O}\left(s\sum_{\a\in \tilde{L}_{\m,n}}\m([\a])\sum_{k=1}^{|\a|-1} e^{-k(\h(\m)-\epsilon_2)}(e^{-\h(\m)}+\epsilon_1)^{-k}\right)\\
&=\mathcal{O}\left(s\sum_{a\in \tilde{L}_{\m,n}}\m([\a])\sum_{k=1}^{\infty}e^{-k(\h(\m)-\epsilon_2)}(e^{-\h(\m)}+\epsilon_1)^{-k}\right).
\end{align*}
By \eqref{epsilons ratio} we know that $$\sum_{k=1}^{\infty}e^{-k(\h(\m)-\epsilon_2)}(e^{-\h(\m)}+\epsilon_1)^{-k}<\infty.$$ Therefore 
$$\int_{e^{-\h(\m)}+\epsilon_1}^{\alpha(\mathcal{B}_{\Gamma})-\epsilon_1} \frac{\#R(\lambda,s,n)}{\tilde{R}_{\m,n}}d\lambda=\mathcal{O}\left(s\sum_{a\in \tilde{L}_{\m,n}}\m([\a])\right)=\mathcal{O}\left(s\right)$$ as required.\\

\noindent\textbf{Step 2. Applying \eqref{transversality separation}.}\\
Combining \eqref{transversality separation} and Lemma \ref{integral bound} we obtain
\begin{equation}
\L([e^{-\h(\m)}+\epsilon_1,\alpha(\mathcal{B}_{\Gamma})-\epsilon_1])-\int_{e^{-\h(\m)}+\epsilon_1}^{\alpha(\mathcal{B}_{\Gamma})-\epsilon_1}\frac{T\Big(\big\{\phi_{\a}\big(\sum_{j=1}^{\infty}d_{a_j}\lambda^{j-1}\big)\big\}_{\a\in \tilde{L}_{\m,n}},\frac{s}{\tilde{R}_{\m,n}}\Big)}{\tilde{R}_{\m,n}}d\lambda=\mathcal{O}(s).
\end{equation} Therefore by Proposition \ref{general prop}, we know that for Lebesgue almost every $\lambda\in [e^{-\h(\m)}+\epsilon_1,\alpha(\mathcal{B}_{\Gamma})-\epsilon_1]$ the set$$\left\{x\in\mathbb{R}:x\in \bigcup_{\a\in \tilde{L}_{\m,n}}B\left(\phi_{\a}\left(\sum_{j=1}^{\infty}d_{a_j}\lambda^{j-1}\right),\frac{h(n)}{\tilde{R}_{\m,n}}\right)\textrm{ for i.m. }n\in\N\right\}$$
has positive Lebesgue measure for any $h\in H$. Since $\epsilon_1$ was arbitrary, we can assert that for Lebesgue almost every $\lambda\in (e^{-\h(\m)},\alpha(\mathcal{B}_{\Gamma})),$ for any $h\in H$ the above set has positive Lebesgue measure. By \eqref{equivalen cardinality} we know that $\tilde{R}_{\m,n}\asymp R_{\m,n}$. Which by the discussion given at the start of this section implies $\tilde{R}_{\m,n}^{-1}\asymp \m([\a])$ for each $\a\in L_{\m,n}$. Therefore by Lemma \ref{arbitrarily small}, we may conclude that for Lebesgue almost every $\lambda\in (e^{-\h(\m)},\alpha(\mathcal{B}_{\Gamma})),$ for any $h\in H$ the set $U_{\Phi_{\lambda,D}}(\sum_{j=1}^{\infty}d_{a_j}\lambda^{j-1},\m,h)$ has positive Lebesgue measure.  \\


\noindent \textbf{Proof of statement 2.}\\
We start our proof of this statement by remarking that since $\m$ is the uniform $(1/l,\ldots,1/l)$ Bernoulli measure, we have $L_{\m,n}=\D^n$ for each $n\in\mathbb{N}$. Since the words in $\D^n$ have the same length and each similarity contracts by a factor $\lambda$, it can be shown that $$\phi_{\a}(z)-\phi_{\a}(z')=\lambda^n(z-z'),$$ for all $\a\in \D^n$ for any $z,z'\in X_{\lambda,D}$. Importantly this difference does not depend upon $\a.$ Therefore the sets $\{\phi_{\a}(z)\}_{\a\in\D^n}$ and $\{\phi_{\a}(z')\}_{\a\in \D^n}$ are translates of each other. In which case 
\begin{equation}
\label{happy}
T\Big(\{\phi_{\a}(z)\}_{\a\in \D^n},\frac{s}{l^n}\Big)=T\Big(\{\phi_{\a}(z')\}_{\a\in \D^n},\frac{s}{l^n}\Big)
\end{equation} for any $z,z'\in X_{\lambda,D}$. 



By \eqref{transversality separation}, the assumptions of Proposition \ref{general prop} are satisfied and by Lemma \ref{density separation lemma}, for any $(a_j)\in \D^{\mathbb{N}},$ for Lebesgue almost every $\lambda\in[1/l+\epsilon_1,\alpha(\mathcal{B}_{\Gamma})-\epsilon_1],$ given an $\epsilon>0$ we can pick $c,s>0$ such that 
\begin{equation}
\label{happier}\overline{d}\left(n:\frac{T\Big(\big\{\phi_{\a}\big(\sum_{j=1}^{\infty}d_{a_j}\lambda^{j-1}\big)\big\}_{\a\in \D^n},\frac{s}{l^n}\Big)}{l^n}\geq c \right)>1-\epsilon.
\end{equation} If $\lambda$ is such that \eqref{happier} holds for a specific sequence $(a_j)\in \D^{\mathbb{N}},$ then \eqref{happy} implies that it must hold for all $(a_j)\in \D^{\mathbb{N}}$ simultaneously. Therefore, we may assert that for Lebesgue almost every $\lambda\in[1/l+\epsilon_1,\alpha(\mathcal{B}_{\Gamma})-\epsilon_1],$ given an $\epsilon>0$ we can pick $c,s>0,$ such that for any $z\in X_{\lambda,D}$ we have 
\begin{equation}
\label{home stretch} \overline{d}\left(n:\frac{T\Big(\big\{\phi_{\a}(z)\big\}_{\a\in \D^n},\frac{s}{l^n}\Big)}{l^n}\geq c \right)>1-\epsilon.
\end{equation} Examining the proof of Proposition \ref{general prop}, we see that \eqref{home stretch} implies that for Lebesgue almost every $\lambda\in [1/l+\epsilon_1,\alpha(\mathcal{B}_{\Gamma})-\epsilon_1],$ for any $z\in X_{\lambda,D}$ and $h\in H$, the set
$$\left\{x\in\mathbb{R}:x\in\bigcup_{\a\in \D^{n}}B\left(\phi_{\a}(z),\frac{h(n)}{l^{n}}\right)\textrm{ for i.m. }n\in \N\right\}$$ 
has positive Lebesgue measure. In other words, for Lebesgue almost every $\lambda\in [1/l+\epsilon_1,\alpha(\mathcal{B}_{\Gamma})-\epsilon_1],$ for any $z\in X_{\lambda,D}$ and $h\in H$, the set $U_{\Phi_{\lambda,D}}(z,\m,h)$ has positive Lebesgue measure. Since $\epsilon_1$ was arbitrary we can conclude our result for Lebesgue almost every $\lambda\in(1/l,\alpha(\mathcal{B}_{\Gamma})).$\\



\noindent \textbf{Proof of statement 3.}\\
 By statement $1$ we know that for any $(a_j)\in \D^{\N},$ for Lebesgue almost every $\lambda\in(e^{-\h(\m)},\alpha(\mathcal{B}_{\Gamma})),$ for any $h\in H$ the set $U_{\Phi_{\lambda,D}}(\sum_{j=1}^{\infty}d_{a_j}\lambda^{j-1},\m,h)$ has positive Lebesgue measure. It follows therefore by Lemma \ref{arbitrarily small} that for any $(a_j)\in \D^{\N},$ for Lebesgue almost every $\lambda\in(e^{-\h(\m)},\alpha(\mathcal{B}_{\Gamma})),$ for any $\Psi$ that is equivalent to $(\m,h)$ for some $h\in H,$ the set $W_{\Phi_{\lambda,D}}(\sum_{j=1}^{\infty}d_{a_j}\lambda^{j-1},\Psi)$ has positive Lebesgue measure. Applying Proposition \ref{full measure} we may conclude that for any $(a_j)\in \D^{\N},$ for Lebesgue almost every $\lambda\in(e^{-\h(\m)},\alpha(\mathcal{B}_{\Gamma})),$ for any $\Psi\in \Upsilon_{\m}$ Lebesgue almost every $x\in X_{\lambda,D}$ is contained in  $W_{\Phi_{\lambda,D}}(\sum_{j=1}^{\infty}d_{a_j}\lambda^{j-1},\Psi)$.\\
 
\noindent \textbf{Proof of statement 4.}\\
The proof of statement $4$ is analogous to the proof of statement $3$. The only difference is that instead of using statement $1$ at the beginning we use statement $2$.
\end{proof}

We now explain how Corollary \ref{example cor} follows from Theorem \ref{1d thm}.

\begin{proof}[Proof of Corollary \ref{example cor}]
Let us start by fixing $h:\mathbb{N}\to[0,\infty)$ to be $h(n)=1/n$. We remark that this function $h$ is an element of $H$. This can be proved using the well known fact $$\sum_{n=1}^{N}\frac{1}{n}\sim \log N.$$ Let us now fix a Bernoulli measure $\m$ as in the statement of Corollary \ref{example cor}. Observe that for any $\a\in \D^*$ we have 
\begin{equation}
\label{cheap decay}(\min_{i\in \D} p_i)^{|\a|}\leq \m([\a])\leq (\max_{i\in \D} p_i)^{|\a|}.
\end{equation} Using \eqref{cheap decay} and the fact that each $\a\in L_{\m,n}$ satisfies $\m([\a])\asymp c_{\m}^{-n},$ it can be shown that each $\a\in L_{\m,n}$ satisfies $$|\a|\asymp n.$$ This implies that for any $\a\in L_{\m,n}$ we have $$\frac{\prod_{j=1}^{|\a|}p_{\a_j}}{|\a|}\asymp \frac{\m([\a])}{n}.$$ In other words, the function $\Psi:\D^*\to[0,\infty)$ given by $$\Psi(\a)=\frac{\prod_{j=1}^{|\a|}p_{\a_j}}{|\a|}$$ is equivalent to $(\m,h)$ for our choice of $h$. One can verify that our function $\Psi$ is weakly decaying and hence $\Psi\in \Upsilon_{\m}.$ Therefore by Theorem \ref{1d thm}, for any $(a_j)\in \D^{\mathbb{N}},$  for almost every $\lambda\in(\prod_{i=1}^{l}p_i^{p_i},\alpha(\mathcal{B}_{\Gamma})),$ Lebesgue almost every $x\in X_{\lambda,D}$ is contained in the set $W_{\Phi_{\lambda,D}}(\sum_{j=1}^{\infty}d_{a_j}\lambda^{j-1},\Psi).$

\end{proof}

\subsubsection{Bernoulli Convolutions}
Given $\lambda\in(0,1)$ and $p\in(0,1)$, let $\mu_{\lambda,p}$ be the distribution of the random sum $$\sum_{j=0}^{\infty}\pm \lambda^{j},$$ where $+$ is chosen with probability $p$, and $-$ is chosen with probability $(1-p)$. When $p=1/2$ we simply denote $\mu_{\lambda,1/2}$ by $\mu_{\lambda}$. We call $\mu_{\lambda,p}$ a Bernoulli convolution. When we want to emphasise the case when $p=1/2$ we call $\mu_{\lambda}$ the unbiased Bernoulli convolution. Importantly, for each $p\in(0,1)$ the Bernoulli convolution $\mu_{\lambda,p}$ is a self-similar measure for the iterated function system $\Phi_{\lambda,\{-1,1\}}=\{\lambda x -1,\lambda x+1\}.$ 

The study of Bernoulli convolutions dates back to the $1930s$ and to the important work of Jessen and Wintner \cite{JesWin}, and  Erd\H{o}s \cite{Erdos1, Erdos2}. When $\lambda\in(0,1/2)$ then $\mu_{\lambda,p}$ is supported on a Cantor set and determining the dimension of $\mu_{\lambda,p}$ is relatively straightforward. When $\lambda\in(1/2,1)$ the support of $\mu_{\lambda,p}$ is the interval $[\frac{-1}{1-\lambda},\frac{1}{1-\lambda}].$  Analysing a Bernoulli convolution for $\lambda\in(1/2,1)$ is a more difficult task. The important problems in this area are:
\begin{itemize}
	\item To classify those $\lambda\in(1/2,1)$ and $p\in(0,1)$ such that \begin{equation}
	\label{expected dimension}\dim_{H}\mu_{\lambda,p}=\min\Big\{ \frac{p\log p+(1-p)\log(1-p)}{\log \lambda},1\Big\}.
	\end{equation}
	\item To classify those $\lambda\in(1/2,1)$ and $p\in(0,1)$ such that $\mu_{\lambda,p}\ll\L$.
\end{itemize}  Initial progress was made on the second problem by Erd\H{o}s in \cite{Erdos1}. He proved that whenever $\lambda$ is the reciprocal of a Pisot number then $\mu_{\lambda}\perp \L$. This result was later improved upon in two papers by Alexander and Yorke \cite{AleYor}, and Garsia \cite{Gar2}, who independently proved that $\dim_{H}\mu_{\lambda}<1$ when $\lambda$ is the reciprocal of a Pisot number. Garsia in \cite{Gar} also provided an explicit class of algebraic integers for which $\mu_{\lambda}\ll\L$. The next breakthrough came in a result of Solomyak \cite{Sol} who proved that $\mu_{\lambda}\ll\L$ with a density in $L^2$ for almost every $\lambda\in(1/2,1)$. His proof relied on studying the Fourier transform of $\mu_{\lambda}$. A simpler proof of this result was subsequently obtained by Peres and Solomyak in \cite{PerSol}. This proof relied upon a characterisation of absolute continuity in terms of differentiation of measures (see \cite{Mat}). Improvements and generalisations of this result appeared subsequently in \cite{PS}, \cite{PerSol2}, and \cite{Rams}. Over the last few years dramatic progress has been made on the problems listed above. In particular, Hochman in \cite{Hochman} proved that for a set $E$ of packing dimension $0$, it is the case that if $\lambda\in(1/2,1)\setminus E$ then we have equality in \eqref{expected dimension} for any $p\in(0,1)$. Building upon this result, Shmerkin in \cite{Shm} proved that $\mu_{\lambda}\ll \L$ for every $\lambda\in(1/2,1)$ outside of a set of Hausdorff dimension zero. This result was later generalised to the case of general $p$ by Shmerkin and Solomyak in \cite{ShmSol}. Similarly building upon the result of Hochman, Varju recently proved in \cite{Varju2} that $\dim_{H}\mu_{\lambda}=1$ whenever $\lambda$ is a transcendental number. Varju has also recently provided new explicit examples of $\lambda$ and $p$ such that $\mu_{\lambda,p}\ll \L$ (see \cite{Var}).

Theorem \ref{1d thm} can be applied to the IFS $\{\lambda x -1,\lambda x+1\}.$ In \cite{Sol} Solomyak proved that $\alpha(\mathcal{B}(\{-1,0,1\}))> 0.639,$ this was subsequently improved upon by Shmerkin and Solomyak in \cite{ShmSol2} who proved that $\alpha(\mathcal{B}(\{-1,0,1\}))> 0.668\ldots.$ Using this information we can prove the following result. 

\begin{thm}
\label{BC cor}
Let $\Psi:\D^*\to[0,\infty)$ be given by $\Psi(\a)=\frac{1}{2^{|\a|}\cdot |\a|}$. Then for Lebesgue almost every $\lambda\in(1/2,0.668),$ we have that for any $z\in [\frac{-1}{1-\lambda},\frac{1}{1-\lambda}],$ Lebesgue almost every $x\in[\frac{-1}{1-\lambda},\frac{1}{1-\lambda}]$ is contained in $W_{\Phi_{\lambda,\{-1,1\}}}(z,\Psi)$. 
\end{thm}The proof of Theorem \ref{BC cor} is an adaptation of the proof of Corollary \ref{example cor} and is therefore omitted.

As a by-product of our analysis we can recover the result of Solomyak that for Lebesgue almost every $\lambda\in(1/2,1)$ the unbiased Bernoulli convolution is absolutely continuous. Our approach does not allow us to assert anything about the density. However our approach does have the benefit of being particularly simple and intuitive. Instead of relying on the Fourier transform, differentiation of measures, or the advanced entropy methods of Hochman \cite{Hochman}, the proof given below appeals to the fact that $\mu_{\lambda}$ is of pure type and makes use of a decomposition argument due to Solomyak. For the sake of brevity, the proof below only focuses on the important features of the argument.

\begin{thm}[Soloymak \cite{Sol}]
For Lebesgue almost every $\lambda\in(1/2,1)$ we have $\mu_{\lambda}\ll\L$.
\end{thm}
\begin{proof}
We split our proof into individual steps.\\
	
\noindent \textbf{Step 1. Proof that $\mu_{\lambda}\ll \L$ for Lebesgue almost every $\lambda\in(1/2,0.668)$.}\\
Fix $(a_j)\in \D^\N$. We know by our proof of Theorem \ref{1d thm} that for any $\epsilon_1>0$ we have 
\begin{equation}
\label{tran bound} \L([1/2+\epsilon_1,0.668-\epsilon_1])-\int_{1/2+\epsilon_1}^{0.668-\epsilon_1}\frac{T(\{\phi_{\a}(\sum_{j=1}^{\infty}a_j\lambda^{j-1})\}_{\a\in\D^n},\frac{s}{2^n})}{2^n}d\lambda=\mathcal{O}(s).
\end{equation} Combining \eqref{tran bound} with Lemma \ref{density separation lemma}, we may conclude that 
\begin{equation}
\label{nearly finished}\L\Big(\bigcap_{\epsilon>0}\bigcup_{c,s>0}\{\lambda\in ([1/2+\epsilon_1,0.668-\epsilon_1]):\overline{d}(n:\omega\in B(c,s,n))\geq 1-\epsilon\}\Big)=\L([1/2+\epsilon_1,0.668-\epsilon_1]).
\end{equation} Here $$B(c,s,n)=\Big\{\lambda\in [1/2+\epsilon_1,0.668-\epsilon_1]: \frac{T(\{\phi_{\a}(\sum_{j=1}^{\infty}a_j\lambda^{j-1})\}_{\a\in\D^n},\frac{s}{2^n})}{2^n}\geq c\Big\}.$$ In particular, \eqref{nearly finished} implies that for Lebesgue almost every $\lambda\in [1/2+\epsilon_1,0.668-\epsilon_1],$ there exists $c>0$ and $s>0$ such that $$\limsup_{n\to\infty}\frac{T(\{\phi_{\a}(\sum_{j=1}^{\infty}a_j\lambda^{j-1})\}_{\a\in\D^n},\frac{s}{2^n})}{2^n}\geq c.$$ Applying Proposition \ref{absolute continuity}, it follows that for Lebesgue almost every $\lambda\in [1/2+\epsilon_1,0.668-\epsilon_1],$ the measure $\mu_{\lambda}$ is absolutely continuous. Since $\epsilon_1$ is arbitrary we know that for Lebesgue almost every $\lambda\in (1/2,0.668)$ the measure $\mu_{\lambda}$ is absolutely continuous. \\

\noindent \textbf{Step 2. Proof that $\mu_{\lambda}\ll \L$ for Lebesgue almost every $\lambda\in(2^{-2/3},0.713)$.}\\
Let $\eta_{\lambda}$ denote the distribution of the random sum $$\sum_{\stackrel{j=0}{j\neq 2 \textrm{ mod }3}}^{\infty}\pm \lambda^{j},$$ where each digit is chosen with probability $1/2$. One can show that $\mu_{\lambda}=\eta_{\lambda}\ast \nu_{\lambda}$ for some measure $\nu_{\lambda}$ corresponding to the remaining terms (see \cite{PerSol, Sol}). Since the convolution of an absolutely continuous measure with an arbitrary measure is still absolutely continuous, to prove $\mu_{\lambda}\ll \L$ for Lebesgue almost every $\lambda\in(2^{-2/3},0.713)$, it suffices to shown that $\eta_{\lambda}$ is absolutely continuous for Lebesgue almost every $\lambda\in(2^{-2/3},0.713)$. Importantly $\eta_{\lambda}$ can be realised as the self-similar measure for the iterated function system
$$\Big\{\rho_{1}(x)=\lambda^3x+1+\lambda,\,\, \rho_{2}(x)=\lambda^3x-1+\lambda,\,\, \rho_{3}(x)=\lambda^3x+1-\lambda,\,\,\rho_{4}(x)=\lambda^3x-1-\lambda\Big\}$$ and the uniform $(1/4,1/4,1/4,1/4)$ Bernoulli measure. Because the translation parameter depends upon $\lambda,$ this family of iterated function systems does not immediately fall into the class considered by Theorem \ref{1d thm}. However this distinction is only superficial, and one can adapt the argument used in the proof of \eqref{tran bound} to prove that for any $\epsilon_1>0$ and $(a_j)\in\{1,2,3,4\}^{\mathbb{N}},$ we have 
\begin{equation}
\L([2^{-2/3}+\epsilon_1,0.713-\epsilon_1])-\label{dumb}\int_{2^{-2/3}+\epsilon_1}^{0.713-\epsilon_1}\frac{T(\{\rho_{\a}(\pi(a_j))\}_{\a\in\{1,2,3,4\}^n},\frac{s}{4^n})}{4^n}d\lambda=\mathcal{O}(s).
\end{equation} The parameter $0.713$ comes from \cite{Sol} and is a lower bound for the appropriate analogue of $\alpha(\mathcal{B}(\{-1,0,1\}))$ for the family of iterated function systems $\{\rho_1,\rho_2,\rho_3,\rho_4\}$. Without going into details, it can be shown that appropriate analogues of Lemma \ref{delta lemma} and Lemma \ref{zero intervals} persist for this family of iterated function systems. These statements can then be used to deduce that \eqref{dumb} holds. By the arguments used in step $1$, we can use \eqref{dumb} in conjunction with Lemma \ref{density separation lemma} and Proposition \ref{absolute continuity} to deduce that $\eta_{\lambda}$ is absolutely continuous for Lebesgue almost every $\lambda\in(2^{-2/3},0.713)$.\\

\noindent \textbf{Step 3. Proof that $\mu_{\lambda}\ll \L$ for Lebesgue almost every $\lambda\in(1/2,1)$.}\\
Since $(1/2,1/\sqrt{2})\subset (1/2,0.668)\cup (2^{-2/3},0,713),$ we know by the two previous steps that for Lebesgue almost every $\lambda\in(1/2,1/\sqrt{2}),$ the measure $\mu_{\lambda}$ is absolutely continuous. For any $\lambda\in(2^{-1/k},2^{-1/2k})$ for some $k\geq 2,$ we can express $\mu_{\lambda}$ as $\mu_{\lambda^k}\ast \nu_{\lambda}$ for some measure $\nu_{\lambda}$ (see \cite{PerSol, Sol}). Since for Lebesgue almost every $\lambda\in(1/2,1/\sqrt{2})$ the measure $\mu_{\lambda}$ is absolutely continuous, it follows that for Lebesgue almost every $\lambda\in (2^{-1/k},2^{-1/2k})$ the measure $\mu_{\lambda^k}$ is also absolutely continuous. Since $\mu_{\lambda}=\mu_{\lambda^k}\ast \nu_{\lambda}$ it follows that $\mu_{\lambda}$ is absolutely continuous for Lebesgue almost every $\lambda\in (2^{-1/k},2^{-1/2k})$. Importantly the intervals $(2^{-1/k},2^{-1/2k})$ exhaust $(1/\sqrt{2},1)$. It follows therefore that $\mu_{\lambda}$ is absolutely continuous for Lebesgue almost every $\lambda\in(1/\sqrt{2},1).$ Our previous steps cover the interval $(1/2,1/\sqrt{2}),$ so we may conclude that $\mu_{\lambda}$ is absolutely continuous for Lebesgue almost every $\lambda\in(1/2,1).$

\end{proof}
\subsubsection{The $\{0,1,3\}$ problem}
Let $\lambda\in(0,1)$ and $$C_{\lambda}:=\left\{\sum_{j=0}^{\infty}a_j\lambda^j:a_j\in\{0,1,3\}\right\}.$$ $C_{\lambda}$ is the attractor of the IFS $\{\lambda x,\lambda x +1,\lambda x +3\}$. When $\lambda\in(0,1/4)$ the IFS satisfies the strong separation condition and one can prove that $\dim_{H}C_{\lambda}=\frac{\log 3}{-\log \lambda}.$ When $\lambda\geq 2/5$ the set $C_{\lambda}$ is the interval $[0,\frac{3}{1-\lambda}]$. The two main problems in the study of $C_{\lambda}$ are:
\begin{itemize}
	\item Classify those $\lambda\in(1/4,1/3)$ such that $\dim_{H}C_{\lambda}=\frac{\log 3}{-\log \lambda}.$ 
	\item Classify those  $\lambda\in(1/3,2/5)$ such that $C_{\lambda}$ has positive Lebesgue measure. 
\end{itemize}Initial progress on these problems was made by Pollicott and Simon in \cite{PolSimon}, Keane, Smorodinsky and Solomyak in \cite{KeSmSo}, and Solomyak in \cite{Sol}. In \cite{PolSimon} it was shown that for Lebesgue almost every $\lambda\in(1/4,1/3)$ we have $\dim_{H}C_{\lambda}=\frac{\log 3}{-\log \lambda}.$ In \cite{Sol} it was shown that for Lebesgue almost every $\lambda\in(1/3,2/5)$ the set $C_{\lambda}$ has positive Lebesgue measure. It follows from the recent work of Hochman \cite{Hochman}, and Shmerkin and Solomyak \cite{ShmSol}, that the set of exceptions for both of these statements has zero Hausdorff dimension. 

In \cite{Sol} it was shown that $\alpha(B(\{0,\pm 1,\pm 2, \pm 3\}))>0.418.$ Using this information we can prove the following result.
\begin{thm}
\label{013thm}	
Let $\Psi:\D^*\to[0,\infty)$ be given by $\Psi(\a)=\frac{1}{3^{|\a|}|\a|}$. Then for Lebesgue almost every $\lambda\in(1/3,0.418),$ we have that for any $z\in C_{\lambda},$ Lebesgue almost every $x\in C_{\lambda}$ is contained in $W_{\Phi_{\lambda,\{0,1,3\}}}(z,\Psi).$ 
\end{thm}Just like the proof of Theorem \ref{BC cor}, the proof of Theorem \ref{013thm} is an adaptation of the proof of Corollary \ref{example cor} and is therefore omitted. 

As stated above, in \cite{Sol} it was shown that for Lebesgue almost every $\lambda\in(1/3,2/5)$ the set $C_{\lambda}$ has positive Lebesgue measure. This was achieved by proving $C_{\lambda}$ supported an absolutely continuous self-similar measure. To the best of the author's knowledge, all results establishing that $C_{\lambda}$ has positive Lebesgue measure for some $\lambda\in(1/3,2/5)$ do so by proving that $C_{\lambda}$ supports an absolutely continuous self-similar measure. It is interesting therefore to note that our methods yield a simple proof of the fact stated above without any explicit mention of a measure. In the proof below, we instead construct a subset of $C_{\lambda}$ that has positive Lebesgue measure for Lebesgue almost every $\lambda\in(1/3,2/5)$.

\begin{thm}[Solomyak \cite{Sol}]
For Lebesgue almost every $\lambda\in(1/3,2/5)$ the set $C_{\lambda}$ has positive Lebesgue measure.	
\end{thm}
\begin{proof}
Taking $(a_j)$ to be the sequence consisting of all zeros in our proof of Theorem \ref{1d thm}, so that $\pi(a_j)=0$ for all $\lambda$, it can be shown that for any $\epsilon_1>0$ we have 
\begin{equation}
\label{tran bound2}
\L([1/3+\epsilon_1,4/5-\epsilon_1])-\int_{1/3+\epsilon}^{4/5-\epsilon_1}\frac{T(\{\sum_{j=0}^{n-1}a_j\lambda^j\}_{\a\in\{0,1,3\}^n},\frac{s}{3^n})}{3^n}d\lambda=\mathcal{O}(s).	
\end{equation}Therefore by Lemma \ref{density separation lemma}, we have
$$\L\Big(\bigcap_{\epsilon>0}\bigcup_{c,s>0}\{\lambda\in ([1/3+\epsilon_1,4/5-\epsilon_1]):\overline{d}(n:\omega\in B(c,s,n))\geq 1-\epsilon\}\Big)=\L([1/3+\epsilon_1,4/5-\epsilon_1]).$$ This implies that for Lebesgue almost every $\lambda\in [1/3+\epsilon_1,4/5-\epsilon_1],$ there exists $c>0$ and $s>0$ such that for infinitely many $n\in\mathbb{N}$ we have 
\begin{equation}
\label{n equation}
\frac{T(\{\sum_{j=0}^{n-1}a_j\lambda^j\}_{\a\in\{0,1,3\}},\frac{s}{3^n})}{3^n}\geq c.
\end{equation}
Let $\lambda'\in [1/3+\epsilon_1,4/5-\epsilon_1]$ be a $\lambda$ satisfying \eqref{n equation} for infinitely many $n$. For any $n\in \N$ satisfying \eqref{n equation} we must also have 
\begin{align}
\label{cheap measure}cs&\leq \L\left(\bigcup_{u\in S(\{\sum_{j=0}^{n-1}a_j\lambda'^j\}_{\a\in\{0,1,3\}},\frac{s}{3^n})}\Big(u-\frac{s}{2\cdot 3^n},u+\frac{s}{2\cdot 3^n}\Big)\right)\nonumber\\
&\leq \L\left(\bigcup_{\a\in \{0,1,3\}^n}\Big(\sum_{j=0}^{n-1}a_j\lambda'^j-\frac{s}{2\cdot 3^n},\sum_{j=0}^{n-1}a_j\lambda'^j+\frac{s}{2\cdot 3^n}\Big)\right).
\end{align}
In which case it follows from \eqref{n equation} and \eqref{cheap measure} that
\begin{align*}
&\L\left(x: |x-\phi_{\a}(0)|\leq \frac{s}{2\cdot 3^n}\textrm{ for i.m. }\a\in \{0,1,3\}^*\right)\\
=&\L\left(\bigcap_{N=1}^{\infty}\bigcup_{n=N}\bigcup_{\a\in \{0,1,3\}^n}\Big(\sum_{j=0}^{n-1}a_j\lambda'^j-\frac{s}{2\cdot 3^n},\sum_{j=0}^{n-1}a_j\lambda'^j+\frac{s}{2\cdot 3^n}\Big)\right)\\
=&\lim_{N\to\infty} \L\left(\bigcup_{n=N}\bigcup_{\a\in \{0,1,3\}^n}\Big(\sum_{j=0}^{n-1}a_j\lambda'^j-\frac{s}{2\cdot 3^n},\sum_{j=0}^{n-1}a_j\lambda'^j+\frac{s}{2\cdot 3^n}\Big)\right)\\
\geq &cs.
\end{align*}
In the penultimate equality we used that Lebesgue measure is continuous from above. In the final inequality we used that there are infinitely many $n\in\mathbb{N}$ such that \eqref{n equation} holds, and therefore infinitely many $n\in\mathbb{N}$ such that \eqref{cheap measure} holds. 

Since $$\left\{x: |x-\phi_{\a}(0)|\leq \frac{s}{2\cdot 3^n}\textrm{ for i.m. }\a\in\{0,1,3\}^*\right\}\subset C_{\lambda'},$$ it follows $C_{\lambda'}$ has positive Lebesgue measure. Since $\lambda'$ was arbitrary, it follows that for Lebesgue almost every $\lambda\in[1/3+\epsilon_1,4/5-\epsilon_1],$ the set $C_{\lambda}$ has positive Lebesgue measure. Since $\epsilon_1$ was arbitrary we can upgrade this statement and conclude that for Lebesgue almost every $\lambda\in(1/3,4/5),$ the set $C_{\lambda}$ has positive Lebesgue measure.
\end{proof}

\subsection{Proof of Theorem \ref{translation thm}}
In this section we prove Theorem \ref{translation thm}. Recall that in the setting of Theorem \ref{translation thm} we obtain a family of IFSs by first of all fixing a set of $d\times d$ non-singular matrices $\{A_i\}_{i=1}^l$ each satisfying $\|A_i\|<1$. For any $\t=(t_1,\ldots,t_l)\in\mathbb{R}^{ld}$ we then define $\Phi_{\t}$ to be the IFS consisting of the contractions $$\phi_{i}(x)=A_ix +t_i.$$ The parameter $\t$ is allowed to vary. We denote the corresponding attractor by $X_\t$ and the projection map from $\D^{\mathbb{N}}$ to $X_\t$ by $\pi_\t$.

To prove Theorem \ref{translation thm} we will need a technical result due to Jordan, Pollicott, and Simon from \cite{JoPoSi}. It is rephrased for our purposes.
 \begin{lemma}\cite[Lemma 7]{JoPoSi}
 	\label{translation transversality}
Assume that $\|A_i\|<1/2$ for all $1\leq i\leq l$ and let $U$ be an arbitrary open ball in $\mathbb{R}^{ld}$. Then for any two distinct sequences $\a,\b\in\D^{\mathbb{N}}$ we have
$$\L\left(\{\t\in U:|\pi_{t}(\a)-\pi_t(\b)|\leq r\}\right)=\mathcal{O}\left(\frac{r^d}{\prod_{i=1}^d \alpha_{i}(A_{a_1\cdots  a_{|\a\wedge \b|-1}})}\right).$$
 \end{lemma}
 With Lemma \ref{translation transversality} we are now in a position to prove Theorem \ref{translation thm}.
 
 \begin{proof}[Proof of Theorem \ref{translation thm}]
Let us start by fixing a set of $d\times d$ non-singular matrices $\{A_i\}_{i=1}^l$ such that $\|A_i\|<1/2$ for all $1\leq i\leq l$. We prove each statement appearing in this theorem individually.\\

\noindent \textbf{Proof of statement $1$}\\
Instead of proving our result for Lebesgue almost every $\t\in\mathbb{R}^{ld},$ it is sufficient to prove our result for Lebesgue almost every $\t\in U,$ where $U$ is an arbitrary ball in $\mathbb{R}^{ld}$. In what follows we fix such a $U$.

By the Shannon-McMillan-Breiman theorem, and the definition of the Lyapunov exponent, we know that for $\m$-almost every $\a\in\D^{\mathbb{N}}$ we have $$\lim_{k\to\infty}\frac{-\log \m([a_1\cdots a_k])}{k}=\h(\m)$$ and for each $1\leq i\leq d$
$$\lim_{k\to\infty}\frac{\log \alpha_{i}(A_{a_1\cdots a_k})}{k}=\lambda_{i}(\m).$$ Applying Egorov's theorem, it follows that for any $\epsilon>0,$ there exists $C>0$ such that the set of $\a\in \D^{\mathbb{N}}$ satisfying 
\begin{equation}
\label{exponential SHM}
\frac{e^{k(-\h(\m)-\epsilon)}}{C}\leq \m([a_1\cdots a_k])\leq Ce^{k(-\h(\m)+\epsilon)}
\end{equation}and 
\begin{equation}
\label{exponential Lyapunov}
\frac{e^{k(\lambda_i(\m)-\epsilon)}}{C}\leq \alpha_{i}(A_{a_1\cdots a_k})\leq Ce^{k(\lambda_i(\m)+\epsilon)}
\end{equation}for each $1\leq i\leq d$ for all $n\in\mathbb{N},$ has $\m$-measure strictly larger than $1/2$. In what follows we will assume that $\epsilon$ has been picked to be sufficiently small so that we have
\begin{equation}
\label{epsilon small}
\h(\m)-\epsilon>-\lambda_1(\m)-\cdots -\lambda_d(\m)+d\epsilon.
\end{equation}
 Such an $\epsilon$ exists because of our underlying assumption $\h(\m)>-\lambda_1(\m)-\cdots -\lambda_d(\m)$. 
 
 For each $n\in \N$ let $$\tilde{L}_{\m,n}=\{\a\in L_{\m,n}: \eqref{exponential SHM} \textrm{ and } \eqref{exponential Lyapunov} \textrm{ hold for }1\leq k\leq |\a|\}$$ and $$\tilde{R}_{\m,n}:=\# \widetilde{L}_{\m,n}.$$ It follows from the above that $$\m\Big(\bigcup_{\a\in \tilde{L}_{\m,n}}[\a]\Big)>1/2.$$ By the discussion given at the beginning of this section, we known $\tilde{R}_{\m,n}$ satisfies the exponential growth condition of Proposition \ref{general prop}. It also follows from our construction that 
 \begin{equation}
 \label{hiphop}
 \tilde{R}_{\m,n}\asymp R_{\m,n}.
 \end{equation} Let us now fix $(a_j)\in \D^{\mathbb{N}}$ and let 
$$ R(\t,s,n):=\left\{(\a,\a')\in \tilde{L}_{\m,n}\times \tilde{L}_{\m,n} :|\phi_{\a}(\pi_{\t}(a_j))-\phi_{\a}(\pi_{\t}(a_j))|\leq \frac{s}{\tilde{R}_{\m,n}^{1/d}} \textrm{ and }\a\neq \a'\right\}.$$
 Our goal now is to prove the bound:
 \begin{equation}
 \label{WTSAA}\int_{U}\frac{\#R(\t,s,n)}{\tilde{R}_{\m,n}}d\L=\mathcal{O}(s^d).
 \end{equation} Repeating the arguments given in the proof statement $1$ from Theorem \ref{1d thm}, it can be shown that 
 $$\int_{U}\frac{\#R(\t,s,n)}{\tilde{R}_{\m,n}}d\L=\mathcal{O}\left(\tilde{R}_{\m,n}\sum_{\stackrel{\a,\a'\in \tilde{L}_{\m,n}}{\a\neq \a'}}\m([\a])\m([\a'])\L(\t\in U:|\pi_\t(\a (a_j))-\pi_\t(\a' (a_j))|\leq \frac{s}{\tilde{R}_{\m,n}^{1/d}})\right).$$ Applying the bound given by Lemma \ref{translation transversality}, we obtain  
 \begin{align*}
 \int_{U}\frac{\#R(\t,s,n)}{\tilde{R}_{\m,n}}d\L&= \mathcal{O}\left(\tilde{R}_{\m,n}\sum_{\stackrel{\a,\a'\in \tilde{L}_{\m,n}}{\a\neq \a'}}\m([\a])\m([\a'])\frac{s^d}{\tilde{R}_{\m,n} \prod_{i=1}^d \alpha_{i}(A_{a_1\cdots a_{|\a\wedge \a'|-1}})} \right)\\
 &= \mathcal{O}\left(\sum_{\stackrel{\a,\a'\in \tilde{L}_{\m,n}}{\a\neq \a'}}\m([\a])\m([\a'])\frac{s^d}{\prod_{i=1}^d \alpha_{i}(A_{a_1\cdots  a_{|\a\wedge \a'|-1}})} \right)\\
 &= \mathcal{O}\left(s^d\sum_{\a\in \tilde{L}_{\m,n}}\m([\a])\sum_{k=1}^{|\a|-1}\sum_{\a':|\a\wedge \a'|=k}\frac{\m([\a'])}{\prod_{i=1}^d \alpha_{i}(A_{a_1\cdots a_{k-1}})} \right)\\
 &= \mathcal{O}\left(s^d\sum_{\a\in \tilde{L}_{\m,n}}\m([\a])\sum_{k=1}^{|\a|-1}\frac{\m([a_1\cdots a_{k-1}])}{\prod_{i=1}^d \alpha_{i}(A_{a_1\cdots a_{k-1}})} \right).
 \end{align*} We now substitute in the bounds provided by \eqref{exponential SHM} and \eqref{exponential Lyapunov} to obtain 
 \begin{align*}
  \int_{U}\frac{\#R(\t,s,n)}{\tilde{R}_{\m,n}}d\L&= \mathcal{O}\left(s^d\sum_{\a\in \tilde{L}_{\m,n}}\m([\a])\sum_{k=1}^{|\a|-1}\frac{e^{k(-\h(\m)+\epsilon)}}{\prod_{i=1}^d e^{k(\lambda_i(\m)-\epsilon)}} \right) \\
  &=\mathcal{O}\left(s^d\sum_{\a\in \tilde{L}_{\m,n}}\m([\a])\sum_{k=1}^{|\a|-1}\frac{e^{k(-\h(\m)+\epsilon)}}{ e^{k(\sum_{i=1}^d\lambda_i(\m)-d\epsilon)}} \right)\\
  &=\mathcal{O}\left(s^d\sum_{\a\in \tilde{L}_{\m,n}}\m([\a])\sum_{k=1}^{\infty}\frac{e^{k(-\h(\m)+\epsilon)}}{ e^{k(\sum_{i=1}^d\lambda_i(\m)-d\epsilon)}}\right)\\
  &=\mathcal{O}\left(s^d\sum_{\a\in \tilde{L}_{\m,n}}\m([\a])\right)\\
  &=\mathcal{O}(s^d).
 \end{align*} In our penultimate equality we used \eqref{epsilon small} to assert that $$\sum_{k=1}^{\infty}\frac{e^{k(-\h(\m)+\epsilon)}}{ e^{k(\sum_{i=1}^d\lambda_i(\m)-d\epsilon)}}<\infty.$$
 We have shown that \eqref{WTSAA} holds. It follows now from \eqref{WTSAA} and Lemma \ref{integral bound} that 
 \begin{equation}
 \label{integral bound2}
 \L(U)-\int_{U}\frac{T\big(\{\phi_{\a}(\pi_\t(a_j))\}_{\a\in \tilde{L}_{\m,n}},\frac{s}{\tilde{R}_{\m,n}^{1/d}}\big)}{\tilde{R}_{\m,n}}d\lambda=\mathcal{O}(s^d).
 \end{equation} Therefore, by Proposition \ref{general prop}, we have that for Lebesgue almost every $\t\in U,$ the set 
 $$\left\{x\in X:x\in\bigcup_{\a\in \tilde{L}_{\m,n}}B\left(\phi_{\a}(\pi_\t(a_j)),\Big(\frac{h(n)}{\tilde{R}_{\m,n}}\Big)^{1/d}\right)\textrm{ for i.m. } n\in \N \right\}$$
 has positive Lebesgue measure for any $h\in H$. By \eqref{hiphop} we know that $\tilde{R}_{\m,n}\asymp R_{\m,n}.$ Which by the discussion given at the beginning of this section implies $\tilde{R}_{\m,n}^{-1}\asymp \m([\a])$ for each $\a\in L_{\m,n}$.  Combining this fact with Lemma \ref{arbitrarily small} and the above, we can conclude that for Lebesgue almost every $\t\in U,$ for any $h\in H$ the set $U_{\Phi_\t}(\pi_{\t}(a_j),\m,h)$ has positive Lebesgue measure. \\
 
 
 \noindent \textbf{Proof of statement 2}\\
Under the assumptions of statement $2,$ it can be shown that for any $\a\in \D^n$ the difference $\phi_{\a}(z)-\phi_{\a}(z')$ is independent of $\a$ for any $z,z'\in X$. Therefore $\{\phi_{\a}(\pi_\t(z))\}_{\a\in \D^n}$ is a translation of $\{\phi_{\a}(\pi_\t(z'))\}_{\a\in \D^n}$ for any $z,z'\in X$. The proof of statement $2$ now follows by the same reasoning as that given in the proof of statement $2$ from Theorem \ref{1d thm}.\\
 
\noindent \textbf{Proof of statement 3}\\
As in the proof of statement $1$, it suffices to show that statement $3$ holds for Lebesgue almost every $\t\in U$ where $U$ is an arbitrary ball. We know by statement $1$ that for any $(a_j)\in \D^{\N},$ for Lebesgue almost every $t\in U$, the set $U_{\Phi_\t}(\pi_{\t}(a_j),\m,h)$ has positive Lebesgue measure for any $h\in H$. Applying Lemma \ref{arbitrarily small} it follows that for any $(a_j)\in \D^{\N},$ for Lebesgue almost every $t\in U$, the set $W_{\Phi_\t}(\pi_{\t}(a_j),\Psi)$ has positive Lebesgue measure for any $\Psi$ that is equivalent to $(\m,h)$ for some $h\in H$. If each $A_i$ is a similarity then we apply the first part of Proposition \ref{full measure} to assert that for any $(a_j)\in \D^{\N},$ for Lebesgue almost every $\t\in U$, for any $\Psi\in \Upsilon_{\m}$ Lebesgue almost every $x\in X_{\t}$ is contained in $W_{\Phi_\t}(\pi_{\t}(a_j),\Psi)$. 

To prove statement $3$ in the case when $d=2$ and all the matrices are equal, and in the case when all the matrices are simultaneously diagonalisable, we will apply the second part of Proposition \ref{full measure}. We need to show that under either of these conditions, for Lebesgue almost every $\t\in U$ the measure $\mu,$ the pushforward of our $\m,$ is equivalent to $\L|_{X_\t}$. Now let us assume our set of matrices satisfies either of these conditions. By \eqref{integral bound2} and Lemma \ref{density separation lemma} we know that $$\L\left(\bigcap_{\epsilon>0}\bigcup_{c,s>0}\{\t\in U: \overline{d}(n:\t\in B(c,s,n))\geq 1-\epsilon\}\right)=\L(U).$$ In particular, this implies that for Lebesgue almost every $\t\in U,$ there exists some $c,s>0$ such that $$\frac{T(\{\phi_{\a}(\pi_t(a_j))\}_{\a\in \tilde{L}_{\m,n}},\frac{s}{\tilde{R}_{\m,n}^{1/d}})}{\tilde{R}_{\m,n}}>c$$ for infinitely many $n\in\mathbb{N}$. By Proposition \ref{absolute continuity} it follows that $\mu\ll \L$ for Lebesgue almost every $\t\in U$. By our hypothesis and Lemma \ref{equivalent measures} we can improve this statement to $\mu\sim \L|_{X_\t}$ for Lebesgue almost every $\t\in U$. Now applying Proposition \ref{full measure} we can conclude that for any $(a_j)\in \D^{\N},$ for Lebesgue almost every $t\in U$, for any $\Psi\in \Upsilon_{\m},$ Lebesgue almost every $x\in X_{\t}$ is contained in $W_{\Phi_\t}(\pi_{\t}(a_j),\Psi).$ \\
 
 \noindent \textbf{Proof of statement 4}\\
The proof of statement $4$ is an adaptation of statement $3,$ where the role of statement $1$ is played by statement $2$. 
 	
  \end{proof}

The proof of Corollary \ref{translation cor} is analogous to the proof of Corollary \ref{example cor} and is therefore omitted. 
\subsection{Proof of Theorem \ref{random thm}}
The proof of Theorem \ref{random thm} mirrors the proof of Theorem \ref{translation thm}. As such we will only state the appropriate analogue of Lemma \ref{translation transversality} and leave the details to the interested reader. The following lemma was proved in \cite{JoPoSi}.
\begin{lemma}\cite[Lemma 6]{JoPoSi}
Assume that $\|A_i\|<1$ for all $1\leq i\leq l$. For any two distinct sequences $\a,\b\in \D^{\mathbb{N}}$ we have $$\mathbf{P}(\y\in\mathbf{D}^{\mathbb{N}}:|\pi_\y(\a)-\pi_\y(\b)|\leq r)=\mathcal{O}\left(\frac{r^d}{\prod_{i=1}^d \alpha_{i}(A_{a_1\cdots a_{|\a\wedge \b|-1}})}\right).$$	
\end{lemma}

\section{A specific family of IFSs}
\label{Specific family}
In this section we focus on the following family of IFSs:
$$\Phi_t:=\Big\{\phi_1(x)=\frac{x}{2},\,\phi_2(x)=\frac{x+1}{2},\,\phi_3(x)=\frac{x+t}{2},\,\phi_4(x)=\frac{x+1+t}{2}\Big\},$$ where $t\in [0,1]$. We also fix $\m$ throughout to be the uniform $(1/4,1/4,1/4,1/4)$ Bernoulli measure. To each $t\in[0,1]\setminus\mathbb{Q}$ we associate the continued fraction expansion $(\zeta_m)\in \N^{\N}$ and the corresponding sequence of partial quotients $(p_m/q_m)$. In this section we will make use of the following well known properties of continued fractions.  

\begin{itemize}
	\item For any $m\in\mathbb{N}$ we have 
	\begin{equation}
	\label{property1}\frac{1}{q_m(q_{m+1}+q_m)}<\left|t-\frac{p_m}{q_m}\right|<\frac{1}{q_mq_{m+1}}.
	\end{equation}
	\item If we set $p_{-1}=1, q_{-1}=0, p_0=0, q_0=1$, then for any $m\geq 1$ we have 
	\begin{align}
	\label{property2}
	p_m&=\zeta_m p_{m-1}+p_{m-2}\\
	q_m&=\zeta_m q_{m-1}+q_{m-2}. \nonumber
	\end{align}
	\item If $t$ is such that $(\zeta_m)$ is bounded, i.e. $t$ is badly approximable, then there exists $c_t>0$ such that for any $(p,q)\in\mathbb{Z}\times \mathbb{N},$ we have 
	\begin{equation}
	\label{property3}
	\Big|t-\frac{p}{q}\Big|\geq \frac{c_t}{q^2}.
	\end{equation}
	\item If $q< q_{m+1}$ then \begin{equation}
	\label{property4}
	|qt-p|\geq |q_{m}t-p_m|
	\end{equation}for any $p\in\mathbb{Z}$.
\end{itemize}  For a proof of these properties we refer the reader to \cite{Bug} and \cite{Cas}. 

Let us now remark that for any $\a\in \D^n,$ there exist two sequences  $(b_j), (c_j)\in\{0,1\}^{n}$ satisfying 
\begin{equation}
\label{separate terms}
\phi_{\a}(x)=\frac{x}{2^n}+\sum_{j=1}^n\frac{b_j}{2^j}+t\sum_{j=1}^n\frac{c_j}{2^j}.
\end{equation}Importantly for each $\a\in \D^n$ the sequences $(b_j)$ and $(c_j)$ satisfying \eqref{separate terms} are unique. What is more, for any  $(b_j),(c_j)\in\{0,1\}^{n},$ there exists a unique $\a\in\D^n$ such that $(b_j)$ and $(c_j)$ satisfy \eqref{separate terms} for this choice of $\a$.

We separate our proof of Theorem \ref{precise result} into individual propositions. Statement $1$ of Theorem \ref{precise result} is contained in the following result.

\begin{prop}
	\label{overlap prop}
$\Phi_t$ contains an exact overlap if and only if $t\in \mathbb{Q}$. Moreover if $t\in \mathbb{Q},$ then for any $z\in[0,1+t],$ the set $U_{\Phi_t}(z,\m,1)$ has Hausdorff dimension strictly less than $1$.
\end{prop}
\begin{proof}
If $\Phi_t$ contains an exact overlap then there exists distinct $\a,\a'\in \D^{*}$ such that $\phi_{\a}=\phi_{\a'}$. By considering $\a\a'$ and $\a'\a$ if necessary, we can assume that $|\a|=|\a'|.$  Using \eqref{separate terms} we see that the following equivalences hold:
\begin{align*}
&\textrm{There exists distinct }\a,\a'\in \D^n \textrm{ such that }\phi_{\a}=\phi_{\a'}.\\
\iff & \textrm{There exists }(b_j),(c_j),(b_j'),(c_j')\in \{0,1\}^n \textrm{ such that } \sum_{j=1}^n\frac{b_j}{2^j}+t\sum_{j=1}^n\frac{c_j}{2^j}=\sum_{j=1}^n\frac{b_j'}{2^j}+t\sum_{j=1}^n\frac{c_j'}{2^j}\\
& \textrm{ and either } (b_j)\neq (b_j') \textrm{ or } (c_j)\neq (c_j').\\
\iff &\textrm{There exists }(b_j),(c_j),(b_j'),(c_j')\in \{0,1\}^n \textrm{ such that  }\sum_{j=1}^n\frac{b_j-b_j'}{2^j}=t\sum_{j=1}^n\frac{c_j'-c_j}{2^j}\\
& \textrm{ and either } (b_j)\neq (b_j') \textrm{ or } (c_j)\neq (c_j').\\
\iff &\textrm{There exists }(b_j),(c_j),(b_j'),(c_j')\in \{0,1\}^n \textrm{ such that }\sum_{j=1}^n2^{n-j}(b_j-b_j')=t\sum_{j=1}^n2^{n-j}(c_j'-c_j)\\
& \textrm{ and either } (b_j)\neq (b_j') \textrm{ or } (c_j)\neq (c_j').\\
\iff &\textrm{There exists }1\leq p,q\leq 2^{n}-1 \textrm{ such that } p=qt.
\end{align*}It follows from these equivalences that there is an exact overlap if any only if $t\in\mathbb{Q}$.

We now prove the Hausdorff dimension part of our proposition. By the first part we know that $t\in\mathbb{Q}$ if and only if $\Phi_t$ contains an exact overlap. It follows from the presence of an exact overlap that for each $t\in \mathbb{Q},$ there exists $j\in \mathbb{N}$ such that $\#\{\phi_{\a}(z):\a\in \D^{j}\}\leq 4^j-1$ for any $z\in[0,1+t]$. This in turn implies that for any $k\in \mathbb{N}$ we have $\#\{\phi_{\a}(z):\a\in \D^{jk}\}\leq (4^j-1)^k$ for any $z\in[0,1+t]$. It follows now from this latter inequality that for an appropriate choice of $c>0,$ for any $z\in[0,1+t]$ we have 
\begin{equation}\label{slower growth}
\#\{\phi_{\a}(z):\a\in \D^n\}=\mathcal{O}((4-c)^n).
\end{equation} For any $z\in[0,1+t]$ and $N\in \N$, the set of intervals $$\{[\phi_{\a}(z)-4^{-n},\phi_{\a}(z)+4^{-n}]\}_{\stackrel{n\geq N}{\phi_{\a}(z):\a\in\D^n}}$$forms a $2\cdot 4^{-N}$ cover of $U_{\Phi_t}(z,\m,1)$. Now let $s\in(0,1)$ be sufficiently large that $$(4-c)<4^s.$$It follows now that we have the following bound on the $s$-dimensional Hausdorff measure of $U_{\Phi_t}(z,\m,1)$
\begin{align*}
\mathcal{H}^s(U_{\Phi_t}(z,\m,1))&\leq \lim_{N\to\infty}\sum_{n=N}^{\infty}\sum_{\phi_{\a}(z):\a\in\D^n}Diam([\phi_{\a}(z)-4^{-n},\phi_{\a}(z)+4^{-n}])^s\\
&\stackrel{\eqref{slower growth}}{=} \lim_{N\to\infty}\mathcal{O}\left(\sum_{n=N}^{\infty} (4-c)^n4^{-ns}\right)\\
&=0.
\end{align*} In the last line we used that $(4-c)<4^s$ to guarantee $\sum_{n=1}^{\infty} (4-c)^n4^{-ns}<\infty$. Therefore $\dim_{H}(U_{\Phi_t}(z,\m,1))\leq s$ for any $z\in[0,1+t]$.

\end{proof}
Adapting the proof of the first part of Proposition \ref{overlap prop}, we can show that the following simple lemma holds.
\begin{lemma}
	\label{obvious lemma}
Let $t\in [0,1]$, $z\in [0,1+t],$ and $s>0$. For $n$ sufficiently large, there exists distinct $\a,\a'\in \D^n$ such that $$|\phi_{\a}(z)-\phi_{\a'}(z)|\leq \frac{s}{4^n}$$ if and only if there exists $1\leq p,q\leq 2^{n}-1$ such that $$|qt-p|\leq \frac{s}{2^n}.$$
\end{lemma}

Lemma \ref{obvious lemma} will be used in the proofs of all the full measure statements in Theorem \ref{precise result}. It immediately yields the following proposition which corresponds to statement $3$ from Theorem \ref{precise result}.

\begin{prop}
If $t$ is badly approximable, then for any $z\in[0,1+t]$ and $h:\mathbb{N}\to[0,\infty)$ satisfying $\sum_{n=1}^{\infty}h(n)=\infty,$ we have that Lebesgue almost every $x\in[0,1+t]$ is contained in $U_{\Phi_t}(z,\m,h)$.
\end{prop}

\begin{proof}
Since $t$ is badly approximable, we know by \eqref{property3} that there exists $c_t>0$ such that \begin{equation}
\label{badly approximable}
|qt-p|\geq \frac{c_t}{q}
\end{equation} for all $(p,q)\in\mathbb{Z}\times\mathbb{N}$. Equation \eqref{badly approximable} implies that for any $1\leq p,q\leq 2^{n}-1$ we have $$|qt-p|>\frac{c_t}{2^n}.$$ Applying Lemma \ref{obvious lemma}, we see that for any $z\in [0,1+t],$ for all $n$ sufficiently large, if $\a,\a'\in \D^n$ are distinct then $$|\phi_{\a}(z)-\phi_{\a'}(z)|>\frac{c_t}{4^n}.$$ Therefore, for any $z\in [0,1+t]$ we have $$S\left(\{\phi_{\a}(z)\}_{\a\in \D^n},\frac{c_t}{4^n}\right)=\{\phi_{\a}(z)\}_{\a\in \D^n}$$ for all $n$ sufficiently large. Our result now follows by an application of Proposition \ref{separated full measure}.
\end{proof}
For our other full measure statements a more delicate analysis is required. We need to identify integers $n$ for which the set of images $\{\phi_{\a}(z)\}_{\a\in \D^n}$ are well separated. This we do in the following two lemmas. 

\begin{lemma}
	\label{diophantine lemma}
Let $s>0$. For $n$ sufficiently large, if $n$ satisfies $$2sq_{m}\leq 2^{n}-1<q_{m}$$ for some $m$, then for any $z\in [0,1+t]$ we have $$|\phi_{\a}(z)-\phi_{\a'}(z)|>\frac{s}{4^n},$$ for distinct $\a,\a'\in \D^n$.
\end{lemma}
\begin{proof}
Fix $s>0$. If $2^{n}-1<q_{m}$, then by \eqref{property1} and \eqref{property4}, for all $1\leq p,q\leq 2^{n}-1$ we have $$|qt-p|\geq |q_{m-1}t-p_{m-1}|\geq \frac{1}{2q_{m}}.$$ If $2sq_{m}\leq 2^{n}-1$ as well, then the above implies that for all $1\leq p,q\leq 2^{n}-1$ we have  $$|qt-p|\geq \frac{s}{2^{n}-1}> \frac{s}{2^n}.$$ Applying Lemma \ref{obvious lemma} completes our proof.
\end{proof}
Lemma \ref{diophantine lemma} demonstrates that if $2^{n}-1$ is strictly less than but close to some denominator arising from the partial quotients of $t$, then at the $n$-th level we have good separation properties. The following lemma demonstrates a similar phenomenon but instead relies upon the digits appearing in the continued fraction expansion. To properly states this lemma we need to define the following sequence. Given $t$ with corresponding sequence of partial quotients $(p_m/q_m)$, we define the sequence of integers $(m_n)$ via the inequalities: 
$$q_{m_n}\leq 2^{n}-1<q_{m_n+1}.$$
\begin{lemma}
	\label{continued fraction lemma}
Let $s>0$. For $n$ sufficiently large, if $n$ is such that $\zeta_{m_n+1}\leq (3s)^{-1},$ then for any $z\in[0,1+t]$ we have $$|\phi_{\a}(z)-\phi_{\a'}(z)|>\frac{s}{4^n}$$ for distinct $\a,\a'\in \D^n$. 
\end{lemma}
\begin{proof}
By \eqref{property1}, \eqref{property2}, and \eqref{property4}, we know that for any $1\leq p,q\leq 2^{n}-1$ we have $$|qt-p|\geq |q_{m_n}t-p_{m_n}|\geq \frac{1}{(q_{m_n+1}+q_{m_n})}= \frac{1}{(\zeta_{m_n+1}+1)q_{m_n}+q_{m_n-1}}>  \frac{1}{3\zeta_{m_n+1}q_{m_n}}.$$ Now using our assumption $\zeta_{m_n+1}\leq (3s)^{-1},$ we may conclude that for any $1\leq p,q\leq 2^{n}-1$ we have $$|qt-p|\geq \frac{s}{q_{m_n}}>\frac{s}{2^n}.$$
Applying Lemma \ref{obvious lemma}, we may conclude our proof.
\end{proof}

With Lemma \ref{diophantine lemma}  and Lemma \ref{continued fraction lemma} in mind we introduce the following definition. We say that $n$ is a good $s$-level if either $$2sq_{m}\leq 2^{n}-1<q_{m}$$ for some $m$, or if $$\zeta_{m_n +1}\leq (3s)^{-1}.$$ It follows from Lemma \ref{diophantine lemma} and Lemma \ref{continued fraction lemma} that if $n$ is a good $s$-level then $$S\left(\{\phi_{\a}(z)\}_{\a\in \D^n},\frac{s}{4^n}\right)=\{\phi_{\a}(z)\}_{\a\in \D^n}$$ for any $z\in [0,1+t]$. 

The following proposition implies statement $2$ from Theorem \ref{precise result}.

\begin{prop}
	\label{Fred}
If $t\notin \mathbb{Q},$ then there exists $h:\mathbb{N}\to[0,\infty)$ depending upon the continued fraction expansion of $t,$ such that $\lim_{n\to\infty}h(n)= 0,$ and for any $z\in[0,1+t]$ Lebesgue almost every $x\in [0,1+t]$ is contained in $U_{\Phi_t}(z,\m,h).$ 
\end{prop}

\begin{proof}
Fix $t\notin \mathbb{Q}$ and let $s=1/8$. For any $m\in\mathbb{N},$ it follows from the definition that $n$ is a good $1/8$-level if $n$ satisfies
\begin{equation}
\label{blahblah}
\frac{q_{m}}{4}\leq 2^{n}-1<q_{m}.
\end{equation}For any $m$ sufficiently large there is clearly at least one value of $n$ satisfying \eqref{blahblah}. As such there are infinitely many good $1/8$ levels. Now let $h:\mathbb{N}\to [0,\infty)$ be a function satisfying $\lim_{n\to\infty}h(n)=0$ and  $$\sum_{\stackrel{n}{n \textrm{ is a good }1/8{\textrm{-level}}}}h(n)=\infty.$$ Now as remarked above, if $n$ is a good $1/8$-level, then $$S\left(\{\phi_{\a}(z)\}_{\a\in \D^n},\frac{1}{8\cdot 4^n}\right)=\{\phi_{\a}(z)\}_{\a\in \D^n}$$ for any $z\in[0,1+t]$. We may now apply Proposition \ref{separated full measure} and conclude that for any $z\in[0,1+t],$ Lebesgue almost every $x\in [0,1+t]$ is contained in $U_{\Phi_t}(z,\m,h)$ for this choice of $h$. 
\end{proof}

In the proof of Proposition \ref{Fred} we showed that if $t\notin \mathbb{Q},$ then for infinitely many $n\in\mathbb{N}$ we have $$S(\{\phi_{\a}(z)\}_{\a\in \D^n},\frac{1}{8\cdot 4^n})=\{\phi_{\a}(z)\}_{\a\in \D^n}.$$ Theorem \ref{overlap or optimal} now follows from this observation and Proposition \ref{overlap prop}.

The following proposition implies statement $5$ from Theorem \ref{precise result}.
\begin{prop}
	\label{propa}
Suppose $t\notin \mathbb{Q}$ is such that for any $\epsilon>0,$ there exists $L\in\mathbb{N}$ for which the following inequality holds for $M$ sufficiently large: $$\sum_{\stackrel{1\leq m \leq M}{\frac{q_{m+1}}{q_m}\geq L}}\log_{2}(\zeta_{m+1}+1) \leq \epsilon M.$$ Then for any $z\in[0,1+t]$ and $h\in H^*,$ Lebesgue almost every $x\in[0,1+t]$ is contained in $U_{\Phi_t}(z,\m,h)$. 
\end{prop}

\begin{proof}
Fix $t$ satisfying the hypothesis of our proposition and $h\in H^*$. By definition, there exists $\epsilon>0$ such that for any $B\subset \mathbb{N}$ satisfying $\underline{d}(B)\geq 1-\epsilon$ we have 
\begin{equation}
\label{bump}\sum_{n\in B}h(n)=\infty.
\end{equation}Now let us fix $s$ to be sufficiently small so that 
\begin{equation}
\label{s condition}
\sum_{\stackrel{1\leq m\leq 2N+2}{q_{m+1}/q_{m}\geq 1/3s}}\log_2 (\zeta_{m+1}+1)\leq \epsilon N
\end{equation}for $N$ sufficiently large.  

We observe that if $n$ is not a good $s$-level then by \eqref{property2} we must have $$\frac{q_{m_n+1}}{q_{m_n}}>\frac{1}{3s}.$$ 
Using \eqref{property2} and an induction argument, one can also show that $$q_{m}\geq 2^{\frac{m-2}{2}}$$ for all $m\geq 1$. Combining these observations, it follows that if $1\leq n\leq N$ and $n$ is not a good $s$-interval, then there exists $1\leq m\leq 2N+2$ such that $\frac{q_{m+1}}{q_{m}}\geq 1/3s$ and $q_m\leq 2^n-1<q_{m+1}$. As such we have the bound
\begin{equation}
\label{count bounda}\#\{1\leq n\leq N: n \textrm{ is not a good }s\textrm{-interval}\}\leq \sum_{\stackrel{1\leq m\leq 2N+2}{q_{m+1}/q_{m}\geq 1/3s}}\#\{n:q_{m}\leq 2^{n}-1< q_{m+1}\}.
\end{equation}By \eqref{property2} we know that for any $m\in\mathbb{N}$ we have $$\#\{n:q_{m}\leq 2^{n}-1< q_{m+1}\}\leq \log_2 (\zeta_{m+1}+1).$$ Substituting this bound into \eqref{count bounda} and applying \eqref{s condition}, we obtain 
$$
\#\{1\leq n\leq N: n \textrm{ is not a good }s\textrm{-interval}\}\leq \sum_{\stackrel{1\leq m\leq 2N+2}{q_{m+1}/q_{m}\geq 1/3s}}\log_2 (\zeta_{m+1}+1)\leq \epsilon N$$ for $N$ sufficiently large. It follows therefore that $$\underline{d}(n: n\textrm{ is a good } s\textrm{-level})\geq 1-\epsilon.$$ In which case, by \eqref{bump} we have \begin{equation}
\label{robot}\sum_{\stackrel{n}{n\textrm{ is a good } s\textrm{-level}}}h(n)=\infty.
\end{equation} We know that for a good $s$-level we have $$S\left(\{\phi_{\a}(z)\}_{\a\in \D^n},\frac{s}{ 4^n}\right)=\{\phi_{\a}(z)\}_{\a\in \D^n},$$ for all $z\in [0,1+t]$. Therefore combining \eqref{robot} with Proposition \ref{separated full measure} finishes our proof.
\end{proof}

The following proposition implies statement $6$ from Theorem \ref{precise result}.
\begin{prop} Suppose $\mu$ is an ergodic invariant measure for the Gauss map, and satisfies $$\sum_{m=1}^{\infty}\mu\Big(\left[\frac{1}{m+1},\frac{1}{m}\right]\Big)\log_2 (m +1)<\infty.$$ Then for $\mu$-almost every $t,$ we have that for any $z\in[0,1+t]$ and $h\in H^*,$ Lebesgue almost every $x\in[0,1+t]$ is contained in $U_{\Phi_t}(z,\m,h).$ In particular, for Lebesgue almost every $t\in[0,1],$ we have that for any $z\in[0,1+t]$ and $h\in H^*$, Lebesgue almost every $x\in[0,1+t]$ is contained in $U_{\Phi_t}(z,\m,h)$.
\end{prop}

\begin{proof}
Let $\mu$ be a measure satisfying the hypothesis of our proposition. To prove the first part of our result we will show that $\mu$-almost every $t$ satisfies the hypothesis of Proposition \ref{propa}. 

Recall that the Gauss map $T:[0,1]\setminus\mathbb{Q}\to[0,1]\setminus\mathbb{Q}$ is given by $$T(x)=\frac{1}{x}-\Bigl\lfloor\frac{1}{x}\Bigr\rfloor.$$ It is well known that the dynamics of the Gauss map and the continued fraction expansion of a number $t$ are intertwined. In particular, it is known that 
\begin{equation}
\label{gauss}
\zeta_{m+1}=\zeta \textrm{ if and only if }T^{m}(t)\in\left(\frac{1}{\zeta+1},\frac{1}{\zeta}\right).
\end{equation} By \eqref{property2} we know that $q_{m+1}/q_{m}\geq L$ implies $\zeta_{m+1}\geq L-1$. Using \eqref{gauss} and this observation, we have that for any $t\notin\mathbb{Q}$
\begin{equation}
\label{Hat}\sum_{\stackrel{1\leq m \leq M}{\frac{q_{m+1}}{q_m}\geq L}}\log _2(\zeta_{m+1}+1) \leq \sum_{1\leq m\leq M}\chi_{(0,\frac{1}{L-1})}(T^m(t))\log_2(f(T^m(t))+1).
\end{equation} Where $f:(0,1]\to\mathbb{N}$ is given by $$f(t)=N\textrm{ if }t\in\left(\frac{1}{N+1},\frac{1}{N}\right].$$ By our assumptions on $\mu$, we know that for any $\epsilon>0$ we can pick $L$ sufficiently large such that $$\sum_{m=L-1}^{\infty}\mu\left(\left[\frac{1}{m+1},\frac{1}{m}\right]\right)\log_2(m+1)<\epsilon.$$

Assuming that we have picked such an $L$ sufficiently large, we know by the Birkhoff ergodic theorem that for $\mu$-almost every $t$ we have \begin{align*}
\lim_{M\to\infty}\frac{1}{M}\sum_{1\leq m\leq M}\chi_{(0,\frac{1}{L-1})}(T^m(t))\log_2(f(T^m(t))+1)&=\int \chi_{(0,\frac{1}{L-1})}\log(f(t)+1)d\mu(t)\\
&=\sum_{m=L-1}^{\infty}\mu\Big(\Big[\frac{1}{m+1},\frac{1}{m}\Big]\Big)\log_2(m+1)\\
&<\epsilon.
\end{align*}
 Combining the above with \eqref{Hat} shows that $\mu$-almost every $t$ satisfies the hypothesis of Proposition \ref{propa}. Applying Proposition \ref{propa} completes the first half of our proof.

To deduce the Lebesgue almost every part of our proposition we remark that the Gauss measure given by $$\mu_{G}(A)=\frac{1}{\log 2}\int_{A} \frac{1}{1+x}\, dx$$ is an ergodic invariant measure for the Gauss map and is equivalent to the Lebesgue measure restricted to $[0,1]$. One can easily check that $\mu_{G}$ satisfies $$\mu_{G}\left(\left[\frac{1}{m+1},\frac{1}{m}\right]\right)=\mathcal{O}\left(\frac{1}{m^2}\right),$$ which clearly implies $$\sum_{m=1}^{\infty}\mu_{G}\left(\left[\frac{1}{m+1},\frac{1}{m}\right]\right)\log (m +1)<\infty.$$ Applying the first part of this proposition completes the proof.

\end{proof}

The following proposition proves statement $4$ from Theorem \ref{precise result}.
\begin{prop}
	Suppose $t\notin\mathbb{Q}$ is not badly approximable. Then there exists $h:\mathbb{N}\to[0,\infty)$ such that $\sum_{n=1}^{\infty}h(n)=\infty,$ yet $U_{\Phi_t}(z,\m,h)$ has zero Lebesgue measure for any $z\in[0,1+t]$. 
\end{prop}

\begin{proof}
Let $t\notin\mathbb{Q}$ and suppose $t$ is not badly approximable. We will prove that for some $s>0$ we have 
\begin{equation}
\label{WTSQ}
\liminf_{n\to\infty}\frac{T(\{\phi_{\a}(z)\}_{\a\in \D^{n}},\frac{s}{4^{n}})}{4^{n}}=0,
\end{equation} for all $z\in[0,1+t]$. Proposition \ref{fail prop} then guarantees for each $z\in[0,1+t]$ the existence of a $h$ satisfying $\sum_{n=1}^{\infty}h(n)=\infty,$ such that $U_{\Phi_t}(z,\m,h)$ has zero Lebesgue measure. What is not clear from the statement of Proposition \ref{fail prop} is whether there exists a $h$ which satisfies this property simultaneously for all $z\in [0,1+t]$. Examining the proof of Proposition \ref{fail prop}, we see that the function $h$ that is constructed only depends upon the speed at which $$\frac{T(\{\phi_{\a}(z)\}_{\a\in \D^{n}},\frac{s}{4^{n}})}{4^{n}}$$ converges to zero along a subsequence. Since 
\begin{equation}
\label{18}
T\left(\{\phi_{\a}(z)\}_{\a\in \D^{n}},\frac{s}{4^{n}}\right)=T\left(\{\phi_{\a}(z')\}_{\a\in \D^{n}},\frac{s}{4^{n}}\right)
\end{equation} for any $z,z'\in[0,1+t]$ and $n\in\mathbb{N}$, it is clear that the speed of convergence to zero along any subsequence is independent of $z$. In particular, the sequence $(n_j)$ constructed in \eqref{separated upper bound} is independent of the choice of $z$. Therefore the $h$ constructed in Proposition \ref{fail prop} will work for all $z\in[0,1+t]$ simultaneously. As such to prove our proposition it is sufficient to show that \eqref{WTSQ} holds for all $z\in [0,1+t]$. 

It also follows from \eqref{18} that to prove there exists $s>0$ such that \eqref{WTSQ} holds for all $z\in[0,1+t]$, it suffices to prove that there exists $s>0$ such that \eqref{WTSQ} for a specific $z\in[0,1+t].$ As such let us now fix $z=0$. It can be shown that for any $n\in \N$ we have $$\{\phi_{\a}(0)\}_{\a\in\D^n}=\left\{\frac{p+qt}{2^n}: 0\leq p\leq 2^{n}-1,\,0\leq q\leq 2^n -1\right\}.$$

Since $t$ is not badly approximable, there exists a sequence $(m_k)$ such that  $\zeta_{m_k+1}\geq k^3$ for all $k\in N$. In which case, by \eqref{property1} and \eqref{property2} we know that
\begin{equation}
\label{stepsize}
|q_{m_k}t-p_{m_k}|\leq \frac{1}{k^3q_{m_k}}
\end{equation} for each $k\in\mathbb{N}$. Without loss of generality we assume $q_{m_k}t-p_{m_k}>0$ for all $k$. This assumption will simplify some of our later arguments.

Define the sequence $(n_k)$ via the inequalities 
\begin{equation}
\label{approximate N}
2^{n_k}\leq k^2q_{m_k}< 2^{n_k+1}.
\end{equation} Consider the set of $(p,q)\in\mathbb{N}^2$ satisfying 
\begin{equation}
\label{glove1}kp_{m_k}\leq p\leq 2^{n_k}-1
\end{equation}
and 
\begin{equation}
\label{glove2}0\leq q\leq 2^{n_k}-1-kq_{m_k}.
\end{equation}	
Note that for any $(p,q)\in \N^2$ satisfying \eqref{glove1} and \eqref{glove2} we have $$0\leq p-ip_{m_k} \leq 2^{n_k}-1$$ and $$0\leq q+iq_{m_k} \leq 2^{n_k}-1$$ for all $0\leq i\leq k$.

Given $k\in \N$ we let $$z_1=\inf\Big\{\frac{p+tq}{2^{n_k}}: (p,q) \textrm{ satisfy }\eqref{glove1}\textrm{ and }\eqref{glove2}\Big\}.$$ Equations \eqref{stepsize} and \eqref{approximate N} imply that for all $0\leq i\leq k$ we have $$z_1+\frac{i(q_{m_k}t-p_{m_k})}{2^{n_k}}\in \left[z_1,z_1+\frac{k}{2^{n_k}k^3q_{m_k}}\right]\subseteq \left[z_1,z_1+\frac{1}{4^{n_k}}\right].$$ 
Assume we have defined $z_1,\ldots,z_l$, we then define $z_{l+1}$ to be $$z_{l+1}=\inf\left\{\frac{p+qt}{2^{n_k}}:\frac{p+qt}{2^{n_k}}> z_l+\frac{1}{4^{n_k}} \textrm{ and }(p,q) \textrm{ satisfy }\eqref{glove1}\textrm{ and }\eqref{glove2}\right\},$$ assuming the set we are taking the infimum over is non-empty. By an analogous argument to that given above, it can be shown that for all $0\leq i\leq k$ we have
$$z_{l+1}+\frac{i(q_{m_k}t-p_{m_k})}{2^{n_k}}\in  \left[z_{l+1},z_{l+1}+\frac{1}{4^{n_k}}\right].$$ This process must eventually end yielding $z_1,\ldots,z_{L(k)}.$ By our construction, we known that if $(p,q)$ satisfy \eqref{glove1} and \eqref{glove2}, then there must exist $1\leq l\leq L(k)$ such that $$\frac{p+qt}{2^{n_k}}\in \left[z_l,z_l+\frac{1}{4^{n_k}}\right].$$ It also follows from our construction that each interval $[z_l,z_l+4^{-{n_k}}]$ contains at least $k+1$ distinct points of the form $\frac{p+qt}{2^{n_k}}$ where $0\leq p\leq 2^{n_k}-1$ and $0\leq q\leq 2^{n_k}-1.$ Since there are only $4^{n_k}$ such points we have 
\begin{equation}
\label{cake1}
L(k)\leq \frac{4^{n_k}}{k+1}.
\end{equation} We also have the bound 
\begin{equation}
\label{cake2}
\#\left\{(p,q):\textrm{either }\eqref{glove1} \textrm{ or }\eqref{glove2}\textrm{ is not satisfied}\right\}=\mathcal{O}(2^{n_k} kp_{m_k}+2^{n_k} kq_{m_k})).
\end{equation} Now let $S(\{\frac{p+tq}{2^{n_k}}\},\frac{1}{4^{n_k}})$ be a maximal $4^{-n_k}$ separated subset of $\{\frac{p+tq}{2^{n_k}}\},$ or equivalently of $\{\phi_{\a}(0)\}_{\a\in \D^{n_k}}$. Then we have
\begin{align}
\label{mario}
\frac{T(\{\frac{p+tq}{2^{n_k}}\},\frac{1}{4^{n_k}})}{4^{n_k}}&=\frac{\#\{(p,q): \eqref{glove1} \textrm{ and }\eqref{glove2} \textrm{ are satisfied and }\frac{p+tq}{2^{n_k}}\in S(\{\frac{p+tq}{2^{n_k}}\},\frac{1}{4^{n_k}})\}}{4^{n_k}} \\
&+\frac{\#\{(p,q):\textrm{either }\eqref{glove1} \textrm{ or }\eqref{glove2}\textrm{ is not satisfied and }\frac{p+tq}{2^{n_k}}\in S(\{\frac{p+tq}{2^{n_k}}\},\frac{1}{4^{n_k}})\}}{4^{n_k}}.\nonumber
\end{align}
If $(p,q)$ satisfy \eqref{glove1} and \eqref{glove2}, then as stated above $\frac{p+tq}{2^{n_k}}\in [z_l,z_l+\frac{1}{4^{n_k}}]$ for some $1\leq l\leq L(k)$. Clearly a $4^{-n_k}$ separated set can only contain one point from each interval $[z_l,z_l+\frac{1}{4^{n_k}}].$ Therefore 
\begin{equation}
\label{cake3}\#\left\{(p,q): \eqref{glove1} \textrm{ and }\eqref{glove2} \textrm{ are satisfied and }\frac{p+tq}{2^n}\in S\Big(\Big\{\frac{p+tq}{2^{n_k}}\Big\},\frac{1}{4^{n_k}}\Big)\right\}\leq L(k).
\end{equation} Substituting the bounds  \eqref{cake1}, \eqref{cake2}, and \eqref{cake3} into \eqref{mario}, we obtain
$$\frac{T(\{\frac{p+tq}{2^{n_k}}\},\frac{1}{4^{n_k}})}{4^{n_k}}=\mathcal{O}\left(\frac{1}{k+1}+\frac{kp_{m_{k}}+kq_{m_{k}}}{2^{n_k}}\right).$$ Employing \eqref{approximate N} and the fact $q_{m_k}\asymp p_{m_k},$ we obtain $$\frac{T(\{\frac{p+tq}{2^{n_k}}\},\frac{1}{4^{n_k}})}{4^{n_k}}=\mathcal{O}\left(\frac{1}{k}\right).$$ Therefore $$\lim_{k\to\infty}\frac{T(\{\phi_{\a}(z)\}_{\a\in \D^{n_k}},\frac{1}{4^{n_k}})}{4^{n_k}}=0$$ and our proof is complete.
\end{proof}

\section{Proof of Theorem \ref{Colette thm}}
\label{Colette section}
In this section we prove Theorem \ref{Colette thm}. We start with a reformulation of what it means for an IFS to be consistently separated with respect to a measure $\m$.

\begin{thm}
	\label{new colette}
Let $\m$ be a slowly decaying measure. An IFS has the CS property with respect to $\m$ if and only if there exists $z\in X$ and $s>0$ such that $$\liminf_{n\to\infty}\frac{T(\{\phi_{\a}(z)\}_{\a\in L_{\m,n}},\frac{s}{R_{\m,n}^{1/d}})}{R_{\m,n}}>0.$$
\end{thm}
\begin{proof}
Suppose that for any $z\in X$ and $s>0$ we have $$\liminf_{n\to\infty}\frac{T(\{\phi_{\a}(z)\}_{\a\in L_{\m,n}},\frac{s}{R_{\m,n}^{1/d}})}{R_{\m,n}}=0.$$ Then by Proposition \ref{fail prop}, Lemma \ref{arbitrarily small}, and the fact $R_{\m,n}^{-1}\asymp \m([\a])$, for any $z\in X$ there exists $h:\mathbb{N}\to[0,\infty)$ such that $\sum_{n=1}^{\infty}h(n)=\infty,$ yet $U_{\Phi}(z,\m,h)$ has zero Lebesgue measure. Therefore the IFS cannot satisfy the CS property with respect to $\m$. So we have shown the rightwards implication in our if and only if.

Now suppose that there exists $z\in X$ and $s>0$ such that $$\liminf_{n\to\infty}\frac{T(\{\phi_{\a}(z)\}_{\a\in L_{\m,n}},\frac{s}{R_{\m,n}^{1/d}})}{R_{\m,n}}>0.$$ Then there exists $c>0$ such that for all $n$ sufficiently large we have $$\frac{T(\{\phi_{\a}(z)\}_{\a\in L_{\m,n}},\frac{s}{R_{\m,n}^{1/d}})}{R_{\m,n}}>c.$$ Combining the fact that $R_{\m,n}^{-1}\asymp \m([\a])$ for $\a\in L_{\m,n}$ together with Lemma \ref{arbitrarily small}, we see that Proposition \ref{fixed omega} implies that the set $U_{\Phi}(z,\m,h)$ has positive Lebesgue measure for any $h$ satisfying $\sum_{n=1}^{\infty}h(n)=\infty.$ Therefore our IFS satisfies the CS property with respect to $\m$ and we have proved the leftwards implication of our if and only if.
\end{proof}
The reformulation of the CS property provided by Theorem \ref{new colette} better explains why we used the terminology consistently separated to describe this property.

With the reformulation provided by Theorem \ref{new colette}, we can give a short proof of Theorem \ref{Colette thm}.

\begin{proof}[Proof of Theorem \ref{Colette thm}]
Let $\m$ be a slowly decaying $\sigma$-invariant ergodic probability measure. Suppose that $\mu$, the pushforward of $\m,$ is not absolutely continuous. Then by Proposition \ref{absolute continuity}, for any $z\in X$ and $s>0$ we have $$\lim_{n\to\infty}\frac{T(\{\phi_{\a}(z)\}_{\a\in L_{\m,n}},\frac{s}{R_{\m,n}^{1/d}})}{R_{\m,n}}=0.$$ By Theorem \ref{new colette} it follows that the IFS $\Phi$ does not satisfy the CS property with respect to $\m$. 
\end{proof}

\section{Proof of Theorem \ref{overlapping conformal theorem}}
\label{conformal section}
In this section we prove Theorem \ref{overlapping conformal theorem}. Recall that Theorem \ref{overlapping conformal theorem} relates to conformal iterated function systems. The parameter $\dim_{S}(\Phi)$ is the unique solution to $$P(s\cdot \log |\phi_{a_1}'(\pi(\sigma(a_j)))|)=0.$$ Moreover, $\m_{\Phi}$ is the unique measure supported on $\D^{\mathbb{N}}$ satisfying $$h_{\m_{\Phi}}+\int \dim_{S}(\Phi)\cdot \log |\phi_{a_1}'(\pi(\sigma(a_j)))|d\m_{\Phi}=0.$$ To prove Theorem \ref{overlapping conformal theorem} we need to state some additional properties of the measure $\m_{\Phi}$:
\begin{itemize}
	\item Let $x\in X$ and $(a_j)$ be such that $\pi(a_j)=x$. Then for any $r\in(0,Diam(X)),$ there exists $N(r)\in\mathbb{N}$ such that \begin{equation}
	\label{prop1}X_{a_1\cdots a_{N(r)}}\subseteq B(x,r)\textrm{ and } Diam(X_{a_1\cdots a_{N(r)}})\approx r.
	\end{equation}
	\item For any $\a\in \D^*$ we have 
	\begin{equation}
	\label{prop2}\m_{\Phi}([\a])\asymp Diam(X_{\a})^{\dim_{S}(X)}.
	\end{equation}
	\item For any $\a,\b\in\D^*$ we have 
	\begin{equation}
\label{prop3}	\m_{\Phi}([\a\b])\asymp \m_{\Phi}([\a])\m_{\Phi}([\b]).
	\end{equation}		
	\item For any $\a,\b\in\D^*$ we have 
	\begin{equation}
	\label{prop5}
	Diam(X_{\a\b})\asymp Diam(X_{\a})Diam(X_{\b}).
	\end{equation}
	\item There exists $\gamma\in(0,1)$ such that 
	\begin{equation}
	\label{prop4}
	\m_{\Phi}([\a])	=\mathcal{O}(\gamma^{|\a|}).
	\end{equation}
\end{itemize}
For a proof of these properties we refer the reader to \cite{Fal2}, \cite{PRSS}, and \cite{Rue}. 

Before giving our proof we make an observation. Given $\theta:\mathbb{N}\to[0,\infty)$ we have the following equivalences:
\begin{align*} \sum_{n=1}^{\infty}\sum_{\a\in\D^n}(Diam(X_\a)\theta(n))^{\dim_{S}(X)}=\infty
\iff&\sum_{n=1}^{\infty}\theta(n)^{\dim_{S}(\Phi)}\sum_{\a\in\D^n}Diam(X_\a)^{\dim_{S}(X)}=\infty\\
\stackrel{\eqref{prop2}}{\iff}&\sum_{n=1}^{\infty}\theta(n)^{\dim_{S}(\Phi)}\sum_{\a\in\D^n}\m_{\Phi}([\a])=\infty\\
\iff&\sum_{n=1}^{\infty}\theta(n)^{\dim_{S}(\Phi)}=\infty.
\end{align*}So the hypothesis of Theorem \ref{overlapping conformal theorem} can be restated in terms of the divergence of $\sum_{n=1}^{\infty}\theta(n)^{\dim_{S}(\Phi)}.$
\begin{proof}[Proof of Theorem \ref{overlapping conformal theorem}]
	We split our proof into individual steps.\\
	
\noindent\textbf{Step $1$. Lifting to $\D^{\mathbb{N}}$.}\\
Let us fix $z\in X$ and $\theta$ satisfying the hypothesis of our theorem. Moreover, we let $(z_j)\in \D^{\N}$ be a sequence such that $\pi(z_j)=z$. For any $\a\in \D^*$ consider the ball $$B(\phi_{\a}(z),Diam(X_\a)\theta(|\a|)).$$ By \eqref{prop1} we know that there exists $N(\a,\theta)$ such that 
\begin{equation}
\label{inclusionZZ}X_{\a z_1\cdots z_{N(\a,\theta)}}\subseteq B(\phi_{\a}(z),Diam(X_\a)\theta(|\a|))
\end{equation}and 
\begin{equation}
\label{hamster}
Diam(X_{\a z_1 \cdots z_{N(\a,\theta)}})\asymp Diam(X_\a)\theta(|\a|).
\end{equation} In what follows we let $$\a_{\theta}:=\a z_1\cdots z_{N(\a,\theta)}.$$Equation \eqref{inclusionZZ} implies the following:
\begin{align*}
\mu_{\Phi}(W_{\Phi}(z,\theta))&=\m_{\Phi}((b_j):\pi(b_j)\in W_{\Phi}(z,\theta))\\
&\geq  \m_{\Phi}((b_j):(b_j)\in [\a_{\theta}] \textrm{ for i.m. }\a\in \D^*).
\end{align*} To complete our proof it therefore suffices to show that 
\begin{equation}
\label{need to show}\m_{\Phi}((b_j):(b_j)\in [\a_{\theta}] \textrm{ for i.m. }\a\in \D^*)=1.
\end{equation} Note that we have 
\begin{equation}
\label{divergencezz}\sum_{n=1}^{\infty}\sum_{\a\in \D^n}\m_{\Phi}([\a_{\theta}])=\infty.
\end{equation}This is because of our underlying divergence assumption and the following:
$$\sum_{n=1}^{\infty}\sum_{\a\in \D^n}\m_{\Phi}([\a_{\theta}])\stackrel{\eqref{prop2}}{\asymp}\sum_{n=1}^{\infty}\sum_{\a\in \D^n}Diam(X_{\a_{\theta}})^{\dim_{S}(\Phi)}\stackrel{\eqref{hamster}}{\asymp}\sum_{n=1}^{\infty}\sum_{\a\in \D^n}(Diam(X_\a)\theta(|\a|))^{\dim_{S}(\Phi)}.$$

\noindent \textbf{Step 2. A density theorem for $\D^{\mathbb{N}}$.}\\ 
To prove \eqref{need to show} we will make use of a density argument. Since we are working in $\D^{\mathbb{N}}$ we do not have the Lebesgue density theorem. Instead we have the statement: suppose $E\subset \D^{\mathbb{N}}$ satisfies $\m_{\Phi}(E)>0,$ then for $\m_{\Phi}$-almost every $(c_j)\in E$ we have 
\begin{equation}
\label{sequence density}\lim_{M\to\infty}\frac{\m_{\Phi}([c_1\cdots c_M]\cap E)}{\m_{\Phi}([c_1\cdots c_M])}=1.
\end{equation} One can see that this statement holds using the results of Rigot \cite{Rig}. In particular, we can equip $\D^\mathbb{N}$ with a metric so that $\m_{\Phi}$ is doubling measure. We can then apply Theorem $2.15$ and Theorem $3.1$ from \cite{Rig}. Using \eqref{sequence density}, we see that to prove \eqref{need to show}, it suffices to show that for any $(c_j)\in \D^{\mathbb{N}},$ there exists $d>0$ such that 
\begin{equation}
\label{needtoshowb}
\frac{\m_{\Phi}([c_1\cdots c_M]\cap \{(b_j):(b_j)\in [\a_{\theta}] \textrm{ for i.m. }\a\in \D^*\})}{\m_{\Phi}([c_1\cdots c_M])}>d
\end{equation} for all $M$ sufficiently large. The rest of the proof now follows from a similar argument to that given by the author in \cite{Bak2}. The difference being here we are now working in the sequence space $\D^\N$ rather than $\mathbb{R}^d$. We include the relevant details for the sake of completeness. \\

\noindent\textbf{Step 3. Defining $E_n$ and verifying the hypothesis of Lemma \ref{Erdos lemma}.}\\
Let us fix $(c_j)\in\D^{\mathbb{N}}$ and $M\in\mathbb{N}$. In what follows we let $\c=c_1\cdots c_M$. For $n\geq M$ let $$E_n:=\left\{[\a_{\theta}]:\a\in\D^n\textrm{ and }a_1\cdots a_M=\c\right\},$$ and let   $$E:=\limsup_{n\to\infty} E_n.$$ Note that $$E\subseteq [\c]\cap \big\{(b_j):(b_j)\in [\a_{\theta}] \textrm{ for i.m. }\a\in \D^*\big\}.$$ Therefore to prove \eqref{needtoshowb}, it is sufficient to prove that there exists $d>0$ independent of $M$ such that 
\begin{equation}
\label{needtoshowc}
\m_{\Phi}(E)>d\m_{\Phi}([\c]).
\end{equation} Note that $$\sum_{n=M}^{\infty}\m_{\Phi}(E_n)=\infty.$$This follows from  
\begin{align*}
\sum_{n=M}^{\infty}\m_{\Phi}(E_n)&=\sum_{n=M}^{\infty}\sum_{\stackrel{\a\in \D^n}{a_1\cdots a_M=\c}}\m_{\Phi}([\a_{\theta}])\\
&=\sum_{n=M}^{\infty}\sum_{\b\in \D^{n-M}}\m_{\Phi}([\c\b z_1\cdots  z_{N(\c\b,\theta)}])\\
&\stackrel{\eqref{prop2}}{\asymp} \sum_{n=M}^{\infty}\sum_{\b\in \D^{n-M}} Diam(X_{\c\b z_1\cdots z_{N(\c\b,\theta)}})^{\dim_{S}(\Phi)}\\
&\stackrel{\eqref{hamster}}{\asymp} \sum_{n=M}^{\infty}\sum_{\b\in \D^{n-M}}(Diam(X_{ \c\b})\theta(n))^{\dim_{S}(\Phi)}\\
&\stackrel{\eqref{prop5}}{\asymp} \sum_{n=M}^{\infty}\sum_{\b\in \D^{n-M}} Diam(X_{\c})^{\dim_{S}(\Phi)}(Diam(X_{\b})\theta(n))^{\dim_{S}(\Phi)}\\
&\stackrel{\eqref{prop2}}{\asymp}  Diam(X_{\c})^{\dim_{S}(\Phi)}\sum_{n=M}^{\infty}\theta(n)^{\dim_{S}(\Phi)}\sum_{\b\in \D^{n-M}}\m_{\Phi}([\b])\\
&=Diam(X_{\c})^{\dim_{S}(\Phi)}\sum_{n=M}^{\infty}\theta(n)^{\dim_{S}(\Phi)}\\
&=\infty.
\end{align*}
In the last line we made use of our underlying hypothesis and the equivalence stated before our proof. Importantly we see that the collection of sets $\{E_n\}_{n\geq M}$ satisfies the hypothesis of Lemma \ref{Erdos lemma}. \\

\noindent\textbf{Step 4. Bounding $\sum_{n,m=M}^Q\m_{\Phi}(E_n\cap E_m).$}\\
To apply Lemma \ref{Erdos lemma} we need to show that the following bound holds:
\begin{equation}
\label{superman}
\sum_{n,m=M}^Q\m_{\Phi}(E_n\cap E_m)=\mathcal{O}\left(\m_{\Phi}([\c])\left(\sum_{n=M}^{Q}\theta(n)^{\dim_{S}(\Phi)} + \left(\sum_{n=M}^{Q}\theta(n)^{\dim_{S}(\Phi)}\right)^2 \right)\right).
\end{equation} 
Let $\a\in \D^n$ be such that $a_1\cdots a_M=\c$ and $m\geq M$. As a first step in our proof of \eqref{superman} we will bound $$\m_{\Phi}([\a_\theta]\cap E_m).$$ There are two cases that naturally arise, when $m>|\a|+N(\a,\theta)$ and when $|\a|<m\leq |\a|+N(\a,\theta).$ Let us consider first the case $|\a|<m\leq |\a|+N(\a,\theta).$ If $|\a|< m\leq |\a|+N(\a,\theta)$ then there is at most one $\a'\in \D^m$ such that 
$$[\a_{\theta}]\cap [\a'_{\theta}]\neq\emptyset.$$ Moreover this $\a'$ must equal $\a z_1\cdots z_{m-n}.$ This gives us the bound:
\begin{align*}
\m_{\Phi}([\a_\theta]\cap E_m)&=\m_{\Phi}([\a_\theta]\cap[\a'_{\theta}])\\
&\leq \m_{\Phi}([\a'_{\theta}])\\
&\stackrel{\eqref{prop2}}{\asymp} Diam(X_{\a'_{\theta}})^{\dim_{S}(\Phi)}\\
&\stackrel{\eqref{hamster}}{\asymp} (Diam(X_{\a'})\theta(m))^{\dim_{S}(\Phi)}\\
&\stackrel{\eqref{prop5}}{\asymp}(Diam(X_{\a})Diam(X_{z_1\cdots z_{m-n}})\theta(m))^{\dim_{S}(\Phi)}\\
&\stackrel{\eqref{prop2}}{\asymp}\m_{\Phi}([\a])\m_{\Phi}([z_1\cdots z_{m-n}])\theta(m)^{\dim_{S}(\Phi)}\\
&\leq\m_{\Phi}([\a])\m_{\Phi}([z_1\cdots z_{m-n}])\theta(n)^{\dim_{S}(\Phi)}\\
&\stackrel{\eqref{prop4}}{=}\mathcal{O}\left(\m_{\Phi}([\a])\theta(n)^{\dim_{S}(\Phi)}\gamma^{m-n}\right).
\end{align*}
In the penultimate line we used that $\theta$ is decreasing. Thus we have shown that if $|\a|<m\leq |\a|+N(\a,\theta)$ then
\begin{equation}
\label{level 1 bound}
\m_{\Phi}([\a_{\theta}]\cap E_m)=\mathcal{O}\left(\m_{\Phi}([\a])\theta(n)^{\dim_{S}(\Phi)}\gamma^{m-n}\right).
\end{equation}
We now consider the case where $m>|\a|+N(\a,\theta)|.$ In this case, if $\a'\in \D^m$ and $$[\a_{\theta}]\cap [\a'_{\theta}]\neq\emptyset,$$ we must have $$a_1'\cdots a_{|\a|+N(\a,\theta)}'=\a_{\theta}.$$ Using this observation we obtain:
\begin{align*}
\m_{\Phi}([\a_{\theta}]\cap E_m)&=\sum_{\stackrel{\a'\in \D^m}{a_1'\cdots a_{|\a|+N(\a,\theta)}'=\a_{\theta}}}\m_{\Phi}([\a'_{\theta}])\\
&=\sum_{\b'\in D^{m-n-N(\a,\theta)}}\m_{\Phi}([\a_{\theta}\b'z_1\cdots  z_{N(\a_{\theta}\b',\theta)}])\\
&\stackrel{\eqref{prop2}}{\asymp}\sum_{\b'\in D^{m-n-N(\a,\theta)}}Diam(X_{\a_{\theta}\b'z_1\cdots z_{N(\a_{\theta}\b',\theta)}})^{\dim_{S}(\Phi)}\\
&\stackrel{\eqref{hamster}}{\asymp}\sum_{\b'\in D^{m-n-N(\a,\theta)}}(Diam(X_{\a_{\theta}\b'})\theta(m))^{\dim_{S}(\Phi)}\\
&\stackrel{\eqref{prop5}}{\asymp}(Diam(X_{\a_{\theta}})\theta(m))^{\dim_{S}(\Phi)}\sum_{\b'\in D^{m-n-N(\a,\theta)}}Diam(X_{\b'})^{\dim_{S}(\Phi)}\\
&\stackrel{\eqref{prop2}}{\asymp}(Diam(X_{\a_{\theta}})\theta(m))^{\dim_{S}(\Phi)}\sum_{\b'\in D^{m-n-N(\a,\theta)}}\m_{\Phi}([\b'])\\
&\stackrel{\eqref{hamster}}{\asymp}(Diam(X_\a)\theta(n)\theta(m))^{\dim_{S}(\Phi)}\\
&\stackrel{\eqref{prop2}}{\asymp}\m_{\Phi}([\a])\theta(n)^{\dim_{S}(\Phi)}\theta(m)^{\dim_{S}(\Phi)}.
\end{align*}Thus we have shown that if $m>|\a|+N(\a,\theta)|$ then
\begin{equation}
\label{level 2 bound}
\m_{\Phi}([\a_{\theta}]\cap E_m)\asymp\m_{\Phi}([\a])\theta(n)^{\dim_{S}(\Phi)}\theta(m)^{\dim_{S}(\Phi)}.
\end{equation}Combining \eqref{level 1 bound} and \eqref{level 2 bound} we obtain the bound
\begin{equation}
\label{level 3 bound}
\m_{\Phi}([\a_{\theta}]\cap E_m)=\mathcal{O}\left(\m([\a])\theta(n)^{\dim_{S}(\Phi)}\gamma^{m-n}+\m_{\Phi}([\a])\theta(n)^{\dim_{S}(\Phi)}\theta(m)^{\dim_{S}(\Phi)}\right).
\end{equation}Importantly this bound holds for all $m>n$. 

Applying \eqref{level 3 bound} we obtain:
\begin{align}
\label{triple split}
\sum_{n,m=M}^Q\m_{\Phi}(E_n\cap E_m)&=\sum_{n=M}^{Q}\m_{\Phi}(E_n)+2\sum_{n=M}^{Q-1}\sum_{m=n+1}^{Q}\m_{\Phi}(E_n\cap E_m)\nonumber\\
&=\sum_{n=M}^{Q}\m_{\Phi}(E_n)+2\sum_{n=M}^{Q-1}\sum_{\stackrel{\a\in \D^n}{a_1\cdots a_M=\c}}\sum_{m=n+1}^{Q}\m_{\Phi}([\a_{\theta}]\cap E_m)\nonumber\\
&\stackrel{\eqref{level 3 bound}}{=}\sum_{n=M}^{Q}\m_{\Phi}(E_n)+\mathcal{O}\left(\sum_{n=M}^{Q-1}\sum_{\stackrel{\a\in \D^n}{a_1\cdots a_M=\c}}\sum_{m=n+1}^{Q}\m_{\Phi}([\a])\theta(n)^{\dim_{S}(\Phi)}\gamma^{m-n}\right)\nonumber\\
&+\mathcal{O}\left(\sum_{n=M}^{Q-1}\sum_{\stackrel{\a\in \D^n}{a_1\cdots a_M=\c}}\sum_{m=n+1}^{Q}\m_{\Phi}([\a])\theta(n)^{\dim_{S}(\Phi)}\theta(m)^{\dim_{S}(\Phi)}\right).
\end{align}
We now analyse each term in \eqref{triple split} individually. Repeating the arguments given at the end of Step $3,$ we can show that
\begin{equation}
\label{piece1}
\sum_{n=M}^{Q}\m_{\Phi}(E_n)\asymp \m_{\Phi}([\c])\sum_{n=M}^Q\theta(n)^{\dim_{S}(\Phi)}.
\end{equation}Focusing on the second term in \eqref{triple split} we obtain:
\begin{align}
\label{piece2}
&\sum_{n=M}^{Q-1}\sum_{\stackrel{\a\in \D^n}{a_1\cdots a_M=\c}}\sum_{m=n+1}^{Q}\m_{\Phi}([\a])\theta(n)^{\dim_{S}(\Phi)}\gamma^{m-n}\nonumber\\
&\stackrel{\eqref{prop2}}{\asymp} \m_{\Phi}([\c])\sum_{n=M}^{Q-1}\sum_{\b\in \D^{n-M}}\sum_{m=n+1}^{Q}\m_{\Phi}([\b])\theta(n)^{\dim_{S}(\Phi)}\gamma^{m-n}\nonumber\\
&\asymp \m_{\Phi}([\c])\sum_{n=M}^{Q-1}\theta(n)^{\dim_{S}(\Phi)}\sum_{\b\in \D^{n-M}}\m_{\Phi}([\b])\sum_{m=n+1}^{Q}\gamma^{m-n}\nonumber\\
&=\mathcal{O}\left(\m_{\Phi}([\c])\sum_{n=M}^{Q-1}\theta(n)^{\dim_{S}(\Phi)}\right).
\end{align}In the last line we used that $\gamma\in(0,1)$ so $\sum_{m=n+1}^{Q}\gamma^{m-n}$ can be bounded above by a constant independent of $n$ and $Q$. 

We now focus on the third term in \eqref{triple split}:
\begin{align}
\label{piece3}
&\sum_{n=M}^{Q-1}\sum_{\stackrel{\a\in \D^n}{a_1\cdots a_M=\c}}\sum_{m=n+1}^{Q}\m_{\Phi}([\a])\theta(n)^{\dim_{S}(\Phi)}\theta(m)^{\dim_{S}(\Phi)}\nonumber\\
& \stackrel{\eqref{prop2}}{\asymp}\m_{\Phi}([\c])\sum_{n=M}^{Q-1}\sum_{\b\in \D^{n-M}}\sum_{m=n+1}^{Q}\m_{\Phi}([\b])\theta(n)^{\dim_{S}(\Phi)}\theta(m)^{\dim_{S}(\Phi)}\nonumber\\
&=\m_{\Phi}([\c])\sum_{n=M}^{Q-1}\theta(n)^{\dim_{S}(\Phi)}\sum_{\b\in \D^{n-M}}\m_{\Phi}([\b])\sum_{m=n+1}^{Q}\theta(m)^{\dim_{S}(\Phi)}\nonumber\\
&\leq \m_{\Phi}([\c])\left(\sum_{n=M}^{Q}\theta(n)^{\dim_{S}(\Phi)}\right)^2.
\end{align}Substituting \eqref{piece1}, \eqref{piece2}, and \eqref{piece3} into \eqref{triple split} we obtain 
$$\sum_{n,m=M}^Q\m_{\Phi}(E_n\cap E_m)=\mathcal{O}\left(\m_{\Phi}([\c])\left(\sum_{n=M}^{Q}\theta(n)^{\dim_{S}(\Phi)} + \left(\sum_{n=M}^{Q}\theta(n)^{\dim_{S}(\Phi)}\right)^2 \right)\right).$$ Therefore \eqref{superman} holds. \\


\noindent\textbf{Step 5. Applying Lemma \ref{Erdos lemma}.}\\
Since $\sum_{n=M}^{\infty}\theta(n)^{\dim_{S}(\Phi)}=\infty$ there exists $Q$ such that $\sum_{n=M}^{Q}\theta(n)^{\dim_{S}(\Phi)}>1.$ Therefore for $Q$ sufficiently large we have 
\begin{equation}
\label{luigi}\sum_{n=M}^{Q}\theta(n)^{\dim_{S}(\Phi)}<\left(\sum_{n=M}^{Q}\theta(n)^{\dim_{S}(\Phi)}\right)^2.
\end{equation}It follows now by \eqref{superman}, \eqref{piece1}  and \eqref{luigi} that there exists some $d>0$ independent of $M$ such that
$$\limsup_{Q\to\infty}\frac{(\sum_{n=M}^{Q}\m_{\Phi}(E_n))^2}{\sum_{n,m=M}^Q\m_{\Phi}(E_n\cap E_m)}\geq \limsup_{Q\to\infty}\frac{d\cdot\left(\m_{\Phi}([\c]) \sum_{n=M}^{Q}\theta(n)^{\dim_{S}(\Phi)}\right)^2}{\m_{\Phi}([\c])\left(\sum_{n=M}^{Q}\theta(n)^{\dim_{S}(\Phi)}\right)^2}= d \m_{\Phi}([\c]).$$
Applying Lemma \ref{Erdos lemma} it follows that $$\m_{\Phi}(\limsup_{n\to\infty} E_n)\geq d \m_{\Phi}([\c]).$$ This implies \eqref{needtoshowc} and completes our proof. 

\end{proof}
\section{Applications of the mass transference principle}
\label{misc}
The main results of this paper give conditions ensuring a limsup set of the form $W_{\Phi}(z,\Psi)$ or $U_{\Phi}(z,\m,h)$ has positive or full Lebesgue measure. For these results it is necessary to assume that some appropriate volume sum diverges. If the relevant volume sum converged, then the limsup set in question would automatically have zero Lebesgue measure by the Borel-Cantelli lemma. It is still an interesting problem to determine the metric properties of a limsup set when the volume sum converges. Thankfully there is a powerful tool for determining the size of a limsup set when the volume sum converges. This tool is known as the mass transference principle and is due to Beresnevich and Velani \cite{BerVel}. We provide a brief account of this technique below. 

We say that a set $X\subset\mathbb{R}^d$ is Ahlfors regular if $$\mathcal{H}^{\dim_{H}(X)}(X\cap B(x,r))\asymp r^{\dim_{H}(X)}$$ for all $x\in X$ and $0<r<Diam(X).$ Given $s>0$ and a ball $B(x,r),$ we define $$B^s:=B(x,r^{s/\dim_{H}(X)}).$$ The theorem stated below is a weaker version of a statement proved in \cite{BerVel}. It is sufficient for our purposes.

\begin{thm}
\label{Mass transference principle}
Let $X$ be Ahlfors regular and $(B_j)$ be a sequence of balls with radii converging to zero. Let $s>0$ and suppose that for any ball $B$ in $X$ we have $$\mathcal{H}^{\dim_{H}(X)}(B\cap \limsup_{j\to\infty}B_j^s)=\mathcal{H}^{\dim_{H}(X)}(B).$$ Then, for any ball $B$ in $X$ $$\mathcal{H}^s(B\cap \limsup_{j\to\infty} B_j)=\mathcal{H}^s(B).$$
\end{thm}
Theorem \ref{Mass transference principle} can be applied in conjunction with Theorems \ref{1d thm}, \ref{translation thm} and \ref{precise result}, to prove many Hausdorff dimension results for the limsup sets $W_{\Phi}(z,\Psi)$ and $U_{\Phi}(z,\m,h)$ when the appropriate volume sum converges. We simply have to restrict to a subset of the parameter space where we know that the corresponding attractor will always be Ahlfors regular. For the sake of brevity we content ourselves with the following statement for the family of iterated function systems studied in Section \ref{Specific family}. This statement is a consequence of Theorem \ref{precise result} and Theorem \ref{Mass transference principle}.

\begin{thm}
Suppose $t\notin \mathbb{Q},$ then for any $z\in [0,1+t]$ and $s>0$ we have $$\dim_{H}(W_{\Phi}(z,4^{-|\a|(1+s)}))=\frac{1}{1+s}$$and $$\mathcal{H}^{\frac{1}{1+s}}(W_{\Phi}(z,4^{-|\a|(1+s)}))=\infty.$$
\end{thm}

\section{Examples}
\label{Examples}
The purpose of this section is to provide some explicit examples to accompany the main results of this paper.
\subsection{IFSs satisfying the CS property}
Here we provide two classes of IFSs that satisfy the CS property with respect to some measure $\m$. These IFSs will have contraction ratios lying in a special class of algebraic integers known as Garsia numbers. A Garsia number is a positive real algebraic integer with norm $\pm 2$ whose Galois conjugates all have modulus strictly greater than $1$. Examples of Garsia numbers include $\sqrt[n]{2}$ for any $n\in\mathbb{N}$, and $1.76929\ldots,$ the appropriate root of $x^3-2x-2=0$. The lemma below is due to Garsia \cite{Gar}, for a short proof see \cite{Bak2}.

\begin{lemma}
	\label{Garsia separation}
	Let $\lambda$ be the reciprocal of a Garsia number. Then there exists $s>0$ such that for any two distinct $\a,\a'\in\{-1,1\}^n$ we have
$$\Big|\sum_{j=0}^{n-1}a_j\lambda^j-\sum_{j=0}^{n-1}a_j'\lambda^j\Big|> \frac{s}{2^n}.$$
\end{lemma}

\begin{example}
Let $\m$ be the $(1/2,1/2)$ Bernoulli measure and for each $\lambda\in(1/2,1),$ let the corresponding IFS be $$\Phi_{\lambda}:=\{\phi_{-1}(x)=\lambda x -1,\phi_1(x)=\lambda x+1\}.$$  For any $\a,\a'\in\{-1,1\}^{n}$ and $z\in [\frac{-1}{1-\lambda},\frac{1}{1-\lambda}],$ it can be shown that $$\phi_{\a}(z)-\phi_{\a'}(z)=\sum_{j=0}^{n-1}a_j\lambda^j-\sum_{j=0}^{n-1}a_j'\lambda^j.$$ Therefore by Lemma \ref{Garsia separation}, if $\lambda$ is the reciprocal of a Garsia number, for any $z\in [\frac{-1}{1-\lambda},\frac{1}{1-\lambda}]$ and distinct $\a,\a'\in\{-1,1\}^n,$ we have $$|\phi_{\a}(z)-\phi_{\a'}(z)|> \frac{s}{2^n}.$$ It follows that for any $z\in [\frac{-1}{1-\lambda},\frac{1}{1-\lambda}]$ we have $$S\left(\{\phi_{\a}(z)\}_{\a\in\{-1,1\}^n},\frac{s}{2^n}\right)=  \{\phi_{\a}(z)\}_{\a\in\{-1,1\}^n}$$ for all $n\in\mathbb{N}$. Applying Proposition \ref{separated full measure} we see that for any $z\in[\frac{-1}{1-\lambda},\frac{1}{1-\lambda}]$ and $h:\mathbb{N}\to [0,\infty)$ satisfying $\sum_{n=1}^{\infty}h(n)=\infty,$ we have that Lebesgue almost every $x\in [\frac{-1}{1-\lambda},\frac{1}{1-\lambda}]$ is contained in $U_{\Phi_{\lambda}}(z,\m,h).$ Therefore if $\lambda$ is the reciprocal of a Garsia number, then the IFS $\Phi_{\lambda}$ has the CS property with respect to $\m$. This fact is a consequence of the main result of \cite{Bak}. The proof given there relied upon certain counting estimates due to Kempton \cite{Kem}. The argument given in the proof of Proposition \ref{separated full measure} doesn't rely on any such counting estimates. Instead we make use of the fact that the Bernoulli convolution is equivalent to the Lebesgue measure and is expressible as the weak star limit of weighted Dirac masses supported on elements of the set $\{\phi_{\a}(z)\}_{\a\in\{-1,1\}^n}$. 
\end{example}

\begin{example}
Let $\m$ be the $(1/4,1/4,1/4,1/4)$ Bernoulli measure and let our IFS be \begin{align*}
\Phi_{\lambda_1,\lambda_2}:=\{&\phi_1(x,y)=(\lambda_1 x+1,\lambda_2 y+1),\phi_2(x,y)=(\lambda_1 x+1,\lambda_2 y-1),\\
&\phi_3(x,y)=(\lambda_1 x-1,\lambda_2 y+1),\phi_4(x,y)=(\lambda_1 x-1,\lambda_2 y-1)\},
\end{align*} where $\lambda_1,\lambda_2\in(1/2,1)$. For each $\Phi_{\lambda_1,\lambda_2}$ the corresponding attractor is $[\frac{-1}{1-\lambda_1},\frac{1}{1-\lambda_1}]\times [\frac{-1}{1-\lambda_2},\frac{1}{1-\lambda_2}]$.  If both $\lambda_1$ and $\lambda_2$ are reciprocals of Garsia numbers, then it follows from Lemma \ref{Garsia separation} that for some $s>0,$ for any $z\in [\frac{-1}{1-\lambda_1},\frac{1}{1-\lambda_1}]\times [\frac{-1}{1-\lambda_2},\frac{1}{1-\lambda_2}],$ we have $$|\phi_{\a}(z)-\phi_{\a'}(z)|> \frac{s}{2^n}$$ for distinct $\a,\a'\in \{1,2,3,4\}^n.$ Therefore
$$S\left(\{\phi_{\a}(z)\}_{\a\in\{1,2,3,4\}^n},\frac{s}{2^n}\right)=  \{\phi_{\a}(z)\}_{\a\in\{-1,1\}^n}$$ for any $z\in [\frac{-1}{1-\lambda_1},\frac{1}{1-\lambda_1}]\times [\frac{-1}{1-\lambda_2},\frac{1}{1-\lambda_2}]$ for all $n\in\mathbb{N}.$

Note that $d=2$ and each of our contractions have the same matrix part. Applying Proposition \ref{separated full measure}, we see that that for any $z\in [\frac{-1}{1-\lambda_1},\frac{1}{1-\lambda_1}]\times [\frac{-1}{1-\lambda_2},\frac{1}{1-\lambda_2}]$ and $h:\mathbb{N}\to [0,\infty)$ satisfying $\sum_{n=1}^{\infty}h(n)=\infty,$ we have that Lebesgue almost every $x\in [\frac{-1}{1-\lambda_1},\frac{1}{1-\lambda_1}]\times [\frac{-1}{1-\lambda_2},\frac{1}{1-\lambda_2}]$ is contained in $U_{\Phi_{\lambda_1,\lambda_2}}(z,\m,h).$ Therefore when $\lambda_1,\lambda_2$ are both reciprocals of Garsia numbers, the IFS $\Phi_{\lambda_1,\lambda_2}$ satisfies the CS property with respect to $\m$. It is perhaps also worth mentioning that by Proposition \ref{separated full measure}, if both $\lambda_1$ and $\lambda_2$ are reciprocals of Garsia numbers, then the pushforward of $\m$ is absolutely continuous. 

\end{example}

\subsection{The non-existence of Khintchine like behaviour without exact overlaps}
In \cite{Bak2} the author asked whether the only mechanism preventing an IFS from observing some sort of Khintchine like behaviour was the presence of exact overlaps. The example below, which is based upon Example 1.2 from \cite{Hochman2}, shows that there are other mechanisms preventing Khintchine like behaviour.

\begin{example}
Pick $t^*\in(0,2/3)$ so that the IFS $$\Phi_{t^*}:=\left\{\phi_{1}(x)=\frac{x}{3},\,\phi_{2}(x)=\frac{x+1}{3},\,\phi_{3}(x)=\frac{x+2}{3},\,\phi_{4}(x)=\frac{x+t^*}{3}\right\}.$$ does not contain an exact overlap. Now consider the following IFS acting on $\mathbb{R}^2$:
\begin{align*}
\Phi_{t^*}':=\{&\phi_1'(x,y)=(x/3,y/3),\phi_2'(x,y)=((x+1)/3,y/3),\\
&\phi_3'(x,y)=((x+2)/3,y/3),\phi_4'(x,y)=((x+t^*)/3,y/3),\\
&\phi_5'(x,y)=(x/3,(y+2)/3),\phi_6'(x,y)=((x+1)/3,(y+2)/3),\\
&\phi_7'(x,y)=((x+2)/3,(y+2)/3),\phi_8'(x,y)=((x+t^*)/3,(y+2)/3)\}.
\end{align*}
The attractor $X$ for $\Phi_{t^*}'$ is $[0,1]\times C$, where $C$ is the middle third Cantor set. Therefore $\dim_{H}(X)=1+\frac{\log 2}{\log 3}$. Since $\Phi_{t^*}$ did not contain an exact overlap, it follows that $\Phi_{t^*}'$ also does not contain an exact overlap.

 Let $\gamma\approx 0.279$ be such that $$8\gamma^{1+\frac{\log 2}{\log 3}}=1.$$ So in particular we have 
 \begin{equation}
\label{nearly} \sum_{n=1}^{\infty}\sum_{\a\in \D^n}\gamma^{n(1+\frac{\log 2}{\log 3})}=\infty.
 \end{equation} If it were the case that our IFS exhibited Khintchine like behaviour, then with \eqref{nearly} in mind, at the very least we would expect that there exists $z\in X$ such that the set $$W:=\Big\{(x,y)\in \mathbb{R}^2:|(x,y)-\phi_{\a}'(z)|\leq \gamma^{|\a|} \textrm{ for i.m. }\a\in \bigcup_{n=1}^{\infty} \{1,\ldots,8\}^n\Big\}$$ has Hausdorff dimension equal to $1+\frac{\log 2}{\log 3}$. We now show that in fact $\dim_{H}(W)<1+\frac{\log 2}{\log 3}.$ 

Let $$\Phi'':=\left\{\phi_1''(y)=\frac{y}{3},\phi_2''(y)=\frac{y+2}{3}\right\}.$$ Clearly $\Phi''$ has the middle third Cantor set as its attractor. We now make the simple observation that if $(x,y)\in \mathbb{R}^2$ satisfies $|(x,y)-\phi_{\a}'(z)|\leq \gamma^{n}$ for some $\a \in\{1,\ldots,8\}^n$ for $z=(z_1,z_2)$, then $|y-\phi_{\a}''(z_2)|\leq \gamma^{n}$ for some $\a\in\{1,2\}^n$. This means that if $|(x,y)-\phi_{\a}'(z)|\leq \gamma^{n}$ for some $\a\in\{1,\ldots,8\}^n,$ then $(x,y)$ must be contained in one of $2^n$ horizontal strip of height $2\gamma^{n}$ and width $1$. Such a strip can be covered by $C(1/\gamma)^{n}$ balls of diameter $\gamma^{n}$ for some $C>0$ independent of $n$. It follows that the set of 
$(x,y)\in\mathbb{R}^2$ satisfying $|(x,y)-\phi_{\a}'(z)|\leq \gamma^{n}$ for some $\a\in\{1,\ldots,8\}^n,$ can be covered by $C(2/\gamma)^n$ balls of diameter $\gamma^{n}$. For each $n$ let $U_n$ be such a collection of balls. By construction, for any $N\in\N$ the set $\cup_{n\geq N}\{B\in U_n\}$ is a $\gamma^N$ cover of $W$. 

Now let 
\begin{equation}
\label{scondition}
s>\frac{\log \gamma - \log 2}{\log \gamma}\approx 1.542.
\end{equation} Then 
\begin{align*}
\mathcal{H}^s\left(W\right)\leq \lim_{N\to\infty}\sum_{n=N}^{\infty}\sum_{B\in U_n}Diam(B)^s\leq \lim_{N\to\infty}\sum_{n=N}^{\infty}C(2/\gamma)^n\cdot \gamma^{sn}=0.
\end{align*}In the final equality we used \eqref{scondition} to guarantee $\sum_{n=1}^{\infty}C(2/\gamma)^n\cdot \gamma^{sn}<\infty.$
 We have shown that $\mathcal{H}^s(W)=0$ for any $s>\frac{\log \gamma - \log 2}{\log \gamma}$. Therefore $\dim_{H}(W)\leq \frac{\log \gamma - \log 2}{\log \gamma}.$ Since $\frac{\log \gamma - \log 2}{\log \gamma}\approx 1.542$ and $1+\frac{\log 2}{\log 3}\approx 1.631,$ we have $\dim_{H}(W)<1+\frac{\log 2}{\log 3}$ as required.
\end{example}
Note that this example can easily be generalised to demonstrate a similar phenomenon when the underlying attractor has positive Lebesgue measure.

\section{Final discussion and open problems}
\label{Final discussion}
A number of problems and questions naturally arise from the results of this paper. The first and likely most difficult question is the following:
\begin{itemize}
	\item Can one derive general, verifiable conditions for an IFS under which we can conclude it exhibits Khintchine like behaviour? 
\end{itemize} This question seems to be very difficult and appears to be out of reach of our current methods. As such it seems that a more reasonable immediate goal would be to prove results for general parameterised families of iterated function systems. One can define a parameterised family of iterated function systems in the following general way. Suppose that $U$ is an open subset of $\mathbb{R}^k,$ and for each $u\in U$ we have an IFS given by $$\Phi_u:=\{\phi_{i,u}(x)=A_{i}(u)(x)+t_{i}(u)\}_{i=1}^l,$$ where for each $1\leq i\leq l$ we have $A_{i}:U\to GL(d,\mathbb{R})\cap\{A:\|A\|<1\}$ and $t_{i}:U\to \mathbb{R}^d.$ For each $u\in U$ we denote the attractor corresponding to this iterated function system by $X_u$. We would like to be able to describe what, if any, Khintchine like behaviour is observed for $\Phi_u$ for a typical $u\in U$. The methods of this paper do not extend to this general a setting and only work when some transversality condition is assumed. We expect that the conjecture stated below holds under some weak assumptions on the functions $A_i$ and $t_i.$

For a $\sigma$-invariant ergodic probability measure $\m,$ and a fixed $u\in U,$ we denote the corresponding Lyapunov exponents by $\lambda_1(\m,u),\ldots,\lambda_d(\m,u).$

\begin{conjecture}
	\label{conjecture}
	Let $\m$ be a slowly decaying $\sigma$-invariant ergodic probability measure and suppose that $\h(\m)>-(\lambda_1(\m,u)+\cdots+\lambda_{d}(\m,u))$ for Lebesgue almost every $u\in U$. Then the following statements hold:
	\begin{itemize}
		\item For Lebesgue almost every $u\in U,$ for any $z\in X_u$ and $h\in H^*,$ Lebesgue almost every $x\in X_u$ is contained in $U_{\Phi_u}(z,\m,h)$.
			\item For Lebesgue almost every $u\in U,$ for any $z\in X_u,$ there exists $h:\mathbb{N}\to[0,\infty)$ such that $\sum_{n=1}^{\infty}h(n)=\infty,$ yet $U_{\Phi_u}(z,\m,h)$ has zero Lebesgue measure. 
	\end{itemize}
\end{conjecture}

Much of the analysis of this paper was concerned with the sequence \begin{equation}
\label{important sequence}
\left(\frac{T(\{\phi_{\a}(z)\}_{\a\in L_{\m,n}},\frac{s}{R_{\m,n}^{1/d}})}{R_{\m,n}}\right)_{n=1}^{\infty},
\end{equation} where $z\in X$ and $\m$ is some slowly decaying $\sigma$-invariant ergodic probability measure. In fact each of our main results was obtained by deriving some quantitative information about the values this sequence takes for typical values of $n$. The behaviour of this sequence provides another useful method for measuring how an IFS overlaps. For the parameterised families considered above, we conjecture that the statement below is true under some weak assumptions on the maps $A_i$ and $t_i.$
\begin{conjecture}
\label{conjecture2}
Let $\m$ be a slowly decaying $\sigma$-invariant ergodic probability measure and suppose that $\h(\m)>-(\lambda_1(\m,u)+\cdots+\lambda_{d}(\m,u))$ for Lebesgue almost every $u\in U$. Then for Lebesgue almost every $u\in U$, for any $z\in X_u,$ for $s$ sufficiently small we have 
$$0=\liminf_{n\to\infty}\frac{T(\{\phi_{\a}(z)\}_{\a\in L_{\m,n}},\frac{s}{R_{\m,n}^{1/d}})}{R_{\m,n}}<\limsup_{n\to\infty}\frac{T(\{\phi_{\a}(z)\}_{\a\in L_{\m,n}},\frac{s}{R_{\m,n}^{1/d}})}{R_{\m,n}}=1.$$ 
\end{conjecture}
One of the interesting ideas to arise from this paper is the notion of an IFS satisfying the CS property with respect to a measure $\m$. Proceeding via analogy with Theorem \ref{precise result}, we expect that given a measure $\m$, it is the case that within a parameterised family of IFSs the CS property will not typically be satisfied with respect to $\m$. Indeed if Conjecture \ref{conjecture2} were true then this statement would follow from Proposition \ref{fail prop}. That being said, we still expect that for a parameterised family of IFSs, it will often be the case that there exists a large subset of the parameter space where the IFS does satisfy the CS property with respect to $\m$. We conjecture that the statement below is true under some weak assumptions on the maps $A_i$ and $t_i.$

\begin{conjecture}
	\label{conjecture3}
Let $\m$ be a slowly decaying $\sigma$-invariant ergodic probability measure and suppose that $\h(\m)>-(\lambda_1(\m,u)+\cdots+\lambda_{d}(\m,u))$ for Lebesgue almost every $u\in U$. Then there exists $U'\subset U$ such that $\dim_{H}(U')=k,$ and for any $u\in U'$ the IFS $\Phi_u$ satisfies the CS property with respect to $\m$. 
\end{conjecture}Theorem \ref{precise result} supports the validity of Conjectures \ref{conjecture}, \ref{conjecture2}, and \ref{conjecture3}. 

Theorem \ref{Colette thm} states that satisfying the CS property with respect to $\m$ implies the pushforward $\mu$ is absolutely continuous. The CS property appears to only be satisfied in exceptional circumstances. As such it is natural to wonder whether there exists a more easily verifiable condition phrased in terms of limsup sets, which implies the absolute continuity of $\mu$. We pose the following question:

\begin{itemize}
	\item Let $\mu$ be the pushforward of a measure $\m$. What is the smallest class of functions, such that if for some $z\in X$ the set $U_{\Phi}(z,\m,h)$ has positive Lebesgue measure for all $h$ belonging to this class, then $\mu$ will be absolutely continuous?
\end{itemize}  

Much of the work presented in this paper is inspired by the classical theorem of Khintchine stated as Theorem \ref{Khintchine} in our introduction. Along with Khintchine's theorem, one of the first results encountered in a course on Diophantine approximation is the following result due to Dirichlet.
\begin{thm}[Dirichlet]
	\label{Dirichlet}
For any $x\in \R$ and $Q\in\N$, there exists $1\leq q\leq Q$ and $p\in\mathbb{Z}$ such
$$\left|x-\frac{p}{q}\right|< \frac{1}{qQ}.$$ Therefore, for any $x\in\mathbb{R}$ there exists infinitely many $(p,q)\in\mathbb{Z}\times\mathbb{N}$ satisfying $$\left|x-\frac{p}{q}\right|< \frac{1}{q^2}.$$
\end{thm}For us the interesting feature of Dirichlet's theorem lies in the fact that it is a statement for all $x\in \mathbb{R}$. In our setting it is obvious that for any IFS $\Phi$, for any $z\in X$ we have 
\begin{equation}
\label{Dirichleta}X=\left\{x\in\mathbb{R}^d:|x-\phi_{\a}(z)|\leq Diam(X_{\a})\textrm{ for i.m. }\a\in \D^*\right\}.
\end{equation} The results of this paper demonstrate that for many overlapping IFSs, given a $z\in X,$ then Lebesgue almost every point in $X$ can be approximated by images of $z$ infinitely often at a scale decaying to zero at an exponentially faster rate than $Diam(X_{\a})$. See for example Theorem \ref{precise result} where Lebesgue almost every point can be approximated at the scale $4^{-|\a|},$ yet $Diam(X_{\a})=2^{-|\a|}.$ With Theorem \ref{Dirichlet} in mind, it is natural to wonder whether there exists conditions under which \eqref{Dirichleta} can be improved upon.
\begin{itemize}
	\item Can one construct an IFS for which there exists $s>1$ such that $$X=\left\{x\in\mathbb{R}^d:|x-\phi_{\a}(z)|\leq Diam(X_{\a})^s\textrm{ for i.m. }\a\in \D^*\right\}.$$ Alternatively one could ask whether there exists $s>1$ such that these sets differ by a finite or countable set.
\end{itemize}
We remark here that for the family of IFSs $\{\lambda x -1,\lambda x+1\},$ it can be shown that there exists $\lambda\in(1/2,0.668)$ and $z\in[\frac{-1}{1-\lambda},\frac{1}{1-\lambda}],$ such that Lebesgue almost every $x\in[\frac{-1}{1-\lambda},\frac{1}{1-\lambda}]$ can be approximated by images of $z$ at the scale $2^{-|\a|},$ yet there exists a set of positive Hausdorff dimension within $[\frac{-1}{1-\lambda},\frac{1}{1-\lambda}]$ that cannot be approximated by images of $z$ at a scale better than $\lambda^{|\a|}$. For more details on this example we refer the reader to the discussion at the end of \cite{Bak}.

We conclude now by emphasising one of the technical difficulties that is present within this paper that is not present within similar works on this topic. In many situations, if $\mu=\mu'\ast \mu'',$ and we have some method for measuring how evenly distributed a measure is with $\mathbb{R}^d$ (examples of methods of measurement include: absolute continuity, entropy, and $L^q$ dimension), then often $\mu$ will be at least as evenly distributed as $\mu'$ with respect to this method of measurement. One may in fact see a strict increase in how evenly distributed $\mu$ is with respect to this method of measurement (see for example \cite{Hochman,Shm}). A useful feature of the pushforward of Bernoulli measures is that they are often equipped with some sort of convolution structure. In many papers this convolution structure and the idea described above can be exploited to obtain results (see for example \cite{Hochman,SSS,Shm,ShmSol,Sol,Var}). Within this paper, the relevant method for measuring how evenly distributed a measure is, is to study the sequence given by \eqref{important sequence}. On a technical level, one of the main difficulties for us is that this method of measurement does not behave well under convolution. This is easy to see with an example. Let $\m$ be the $(1/2,1/2)$ Bernoulli measure and let our IFS be $\{\phi_1(x)=\frac{x}{2},\phi_2(x)=\frac{x+1}{2}\}.$ For this IFS the attractor is $[0,1].$ We denote the pushforward of $\m$ by $\mu'$. It is easy to see that for any $z\in [0,1]$ and $n\in\mathbb{N}$, we have
\begin{equation}
\label{optimall}\frac{T(\{\phi_{\a}(z)\}_{\a\in \{1,2\}n},\frac{1}{2\cdot 2^n})}{2^n}=1.
\end{equation}
So $\mu'$ exhibits an optimal level of separation. Now let $t\in (0,1)\cap \mathbb{Q}$ and consider the IFS  $\{\phi_1(x)=\frac{x}{2},\phi_2(x)=\frac{x+t}{2}\}.$ For this IFS the attractor is $[0,t].$ We denote the pushforward of $\m$ for this IFS by $\mu''$. It is easy to see that for $\mu''$ we also have the optimal level of separation described by \eqref{optimall}. Consider the measure $\mu=\mu'\ast \mu''.$ This measure is simply the pushforward of the $(1/4,1/4,1/4,1/4)$ Bernoulli measure with respect to the IFS $$\Big\{\phi_1(x)=\frac{x}{2},\phi_2(x)=\frac{x+1}{2},\phi_3(x)=\frac{x+t}{2},\phi_{4}(x)=\frac{x+1+t}{2}\Big\},$$ i.e. the IFS studied in Theorem \ref{precise result}. Examining the proof of Proposition \ref{overlap prop}, we see that for any $t\in(0,1)\cap\mathbb{Q},$ there exists $c>0$ such that for any $z\in[0,1+t]$ and $s>0$ we have \begin{equation}
\label{optimal fail}
T\left(\{\phi_{\a}(z)\}_{\a\in \{1,2,3,4\}^n},\frac{s}{4^n}\right)=\mathcal{O}((4-c)^n).
\end{equation} Equation \eqref{optimal fail} demonstrates that we no longer have the strong separation properties that we saw 
earlier for our two measures $\mu'$ and $\mu''$. We have in fact seen that after convolving $\mu'$ and $\mu''$ there is a drop in how evenly distributed the resulting measure is within $\mathbb{R}$. One could view this failure to improve under convolution as a consequence of how sensitive our method of measurement is to exact overlaps.

\medskip\noindent {\bf Acknowledgments.} The author would like to thank the anonymous referee for their valuable comments. The author would like to thank Tom Kempton and Boris Solomyak for providing some useful feedback on an initial draft, and Ariel Rapaport for pointing out the reference \cite{Shm3}. This research was supported by the EPSRC grant EP/M001903/1.

\end{document}